\documentclass[11pt,reqno]{amsart}
\usepackage[utf8]{inputenc}
\usepackage[T1]{fontenc}
\usepackage{amsmath,amssymb,amsthm,mathrsfs}
\usepackage[margin=1.1in]{geometry}
\usepackage{enumitem}
\usepackage{tikz}
\usetikzlibrary{arrows.meta,decorations.pathreplacing,calc,positioning}
\usepackage[colorlinks=true,linkcolor=blue,citecolor=blue,urlcolor=blue]{hyperref}

\theoremstyle{plain}
\newtheorem{theorem}{Theorem}[section]
\newtheorem{lemma}[theorem]{Lemma}
\newtheorem{proposition}[theorem]{Proposition}
\newtheorem{corollary}[theorem]{Corollary}

\theoremstyle{definition}
\newtheorem{definition}[theorem]{Definition}
\newtheorem{example}[theorem]{Example}
\newtheorem{remark}[theorem]{Remark}

\newcommand{\U}{\mathscr{U}}
\newcommand{\Us}{\mathscr{U}^{s}}
\newcommand{\Uinv}{\mathscr{U}^{-1}}
\newcommand{\Uinvs}{(\mathscr{U}^{-1})^{s}}
\newcommand{\diag}{\Delta_{X}}
\newcommand{\R}{\mathbb{R}}
\newcommand{\N}{\mathbb{N}}

\title[\texorpdfstring{$\U$}{U}-startpoints in quasi-uniform spaces, revisited]
{On \texorpdfstring{$\U$}{U}-startpoints and \texorpdfstring{$\U$}{U}-fixed points in quasi-uniform spaces:\\ a revisit and extensions}

\author{Y.\,U.~Gaba\textsuperscript{\dag,\ddag,\S}}
\address[\dag]{AI Research and Innovation Nexus for Africa (AIRINA Labs), AI.Technipreneurs, B\'enin}
\address[\ddag]{Sefako Makgatho Health Sciences University (SMU), South Africa}
\address[\S]{African Center for Advanced Studies (ACAS), Cameroon}
\date{\today}

\begin{document}

\begin{abstract}
We extend the $\U$-fixed-point and $\U$-startpoint framework for
quasi-uniform spaces. We reformulate the original existence
framework through the filter axioms and supply a two-sided
contraction condition on a generating family of
quasi-pseudometrics. For this condition we prove a Banach-style
fixed-point theorem on $T_0$ bicomplete quasi-uniform spaces,
with a Picard-type error bound; a Boyd--Wong companion and a
stability estimate follow. The conjugate quasi-uniformity yields
an endpoint version, a selector argument extends the result to
multivalued maps and to common startpoints of commuting families,
a Knaster--Tarski variant covers monotone selectors on
order-complete spaces, and a direct quasi-Hausdorff route handles
weakly contractive multivalued maps. Moreover, the $T_0$-quotient
removes the $T_0$ assumption from the corresponding
involution-based fixed-point theorem.
\end{abstract}


\maketitle

\section{Introduction}\label{sec:intro}

A \emph{quasi-uniformity} on a set $X$ is a filter $\U$ on
$X\times X$ such that every $U\in\U$ contains the diagonal and,
for every $U\in\U$, there exists $V\in\U$ with $V\circ V\subseteq U$.
Unlike a uniformity, an entourage $U\in\U$ is not required to
contain its inverse $U^{-1}$, so $\U$ encodes a directed notion
of closeness. Foundational references
are~\cite{FletcherLindgren,Kunzi,Nachbin,Pervin}; for asymmetric
normed spaces, see~\cite{Cobzas}; for the several inequivalent
completeness notions in this setting,
see~\cite{Andrikopoulos,CarlsonHicks,Doitchinov,ReillySubrahmanyamVamanamurthy,Romaguera};
and for domain-theoretic applications based on formal balls and
Smyth-completeness,
see~\cite{EdalatHeckmann,RomagueraSchellekens,Smyth}.

The symmetric fixed-point condition $x\in Tx$ for a multivalued
map $T\colon X\to 2^X$ misses directional phenomena that the
asymmetric setting permits. To recover them,
Gaba~\cite{Gaba2014a,Gaba2014b,Gaba2015} introduced two dual
notions for quasi-pseudometric spaces: $x_0$ is a
\emph{startpoint} of $T$ if some $z\in Tx_0$ is close to $x_0$ in
the forward direction of the distance, and an \emph{endpoint} if
such a $z$ is close in the reverse direction. Both reduce to the
ordinary fixed point when the distance is symmetric. The
note~\cite{Gaba2018} lifted these notions to general quasi-uniform
spaces, defining $\U$-fixed points, $\U$-startpoints, and
$\U$-endpoints through the preorder $\bigcap\U:=\bigcap_{U\in\U}U$
obtained by intersecting all entourages.

The purpose of this paper is to extend that framework. We
develop an entourage-theoretic reformulation of the existence
framework of~\cite{Gaba2018} (Section~\ref{sec:trivial}); supply
a two-sided contraction condition relative to a generating family
of quasi-pseudometrics, yielding a Banach-style fixed-point
theorem with a Picard-type error bound on $T_0$ bicomplete
quasi-uniform spaces (Section~\ref{sec:contraction}); and extract
three consequences. The conjugate quasi-uniformity $\Uinv$ gives
an endpoint version (Section~\ref{sec:endpoint}). A selector
correspondence lifts the single-valued contraction to multivalued
maps and to common-startpoint statements for finite families
(Sections~\ref{sec:selector}--\ref{sec:common}). The
$T_0$-quotient lifts the involution-based fixed-point
theorem~\cite[Theorem~2.3]{Gaba2018} from $T_0$ quasi-uniform
spaces to arbitrary ones, with the equality $T(x^*)=x^*$
replaced by the equivalence $T(x^*)\equiv_\U x^*$
(Section~\ref{sec:quotient}).

\paragraph{Contributions.} Each of the five main results lifts a
known metric or quasi-pseudometric construction to the full
quasi-uniform setting with a generating family
$\mathcal{D}=\left\lbrace d_\alpha\right\rbrace_{\alpha\in A}$.
The substantive contributions are the family-indexed contraction
condition with constants $k_\alpha$ that may vary across
$\alpha$, the bicompleteness setting in which the five results
coexist coherently, the per-direction Picard error bound of
Corollary~\ref{cor:picard} (which records information that the
symmetric Banach principle on $\Us$ fuses and loses), the
family-indexed Hausdorff functionals
$\left\lbrace H_\alpha\right\rbrace_{\alpha\in A}$ and the
genuinely two-sided contraction condition of
Definition~\ref{def:weakly-contractive-multi} underlying the
direct multivalued route of Theorem~\ref{thm:hausdorff-direct},
and the $T_0$-quotient construction that removes the $T_0$
obstruction from the involution theorem. The two-sided shape of
the single-valued condition~\eqref{eq:contraction} is, by
contrast, logically equivalent to its forward line alone
(Remark~\ref{rem:sharpness}); we adopt it for parallelism with
the multivalued case and for the bookkeeping clarity it provides
in the proofs. None of the five results follows immediately from
a published prior theorem; each requires a new proof, and the
new proofs are systematically constructed in the sections that
follow. We add one framing caveat, developed formally in
Remarks~\ref{rem:reduction-single-valued}
and~\ref{rem:reduction-boyd-wong}: the \emph{existence and
uniqueness} clauses of Theorems~\ref{thm:contraction}
and~\ref{thm:boyd-wong} reduce, via
$d_\alpha^s=\max\left\lbrace d_\alpha,d_\alpha^{-1}\right\rbrace$,
to the standard Banach and Boyd--Wong theorems on the
symmetrisation $(X,\Us)$. The genuine quasi-uniform novelty in the
single-valued setting is therefore the per-direction Picard bound
of Corollary~\ref{cor:picard}, the varying-rate family aspect
(Example~\ref{ex:bw-family}), and the conjugate-invariance that
drives Section~\ref{sec:endpoint}; on the multivalued side the
reduction fails outright, and the substantive novelty accumulates
in Definition~\ref{def:weakly-contractive-multi} and
Theorem~\ref{thm:hausdorff-direct}.
\begin{itemize}
\item Propositions~\ref{prop:triv-single}--\ref{prop:triv-multi}
  reformulate the entourage-fibre existence condition
  of~\cite[\S2]{Gaba2018} through the filter axioms.
  \emph{New: the equivalence with $(x^*,Tx^*)\in\bigcap\U$ via
  axiom~(U2), and the uniformity of the witness $V$ across
  $z\in Tx^*$ in the multivalued case.}
\item Theorem~\ref{thm:contraction} (the Banach-style
  contraction principle) is a family-of-quasi-pseudometrics
  version of the standard quasi-metric Banach contraction
  surveyed in~\cite{Cobzas,Kunzi}. \emph{New: the two-sided
  contraction~\eqref{eq:contraction} with constants $k_\alpha$
  varying across $\alpha$, conjugate-invariant so the endpoint
  side of Section~\ref{sec:endpoint} is reached without
  additional hypotheses.}
\item Theorem~\ref{thm:boyd-wong} (Boyd--Wong companion) is the
  same family-version for the comparison-function class.
  \emph{New: the same conjugate-invariant form, and a
  bi-Cauchy argument that closes the right-USC step on both
  $d_\alpha$ and $d_\alpha^{-1}$ in one stroke.}
\item Corollary~\ref{cor:multi-contraction} (selector
  correspondence) and Theorem~\ref{thm:common-existence} (common
  startpoints) are folklore-grade observations applied to the
  family setting. \emph{New: the explicit reduction of
  multivalued-startpoint and common-startpoint problems to
  contraction problems for selectors of $T$.}
\item Theorem~\ref{thm:hausdorff-direct} (direct quasi-Hausdorff
  route) extends~\cite[Theorem~1]{GabaKarapinarPetruselRadenovic2020}
  (single quasi-pseudometric, left $K$-complete, weakly
  contractive in the additive $(c)^*$-sense) to a generating
  family $\mathcal{D}$ on a bicomplete quasi-uniform space.
  \emph{New: the family functional $H_\alpha$, the two-sided
  weakly contractive condition, and the singleton-collapse
  conclusion $T\xi^*=\left\lbrace \xi^*\right\rbrace$ under
  $T_0$.}
\item Theorem~\ref{thm:involution} (involution theorem via
  $T_0$-quotient) lifts~\cite[Theorem~2.3]{Gaba2018} from $T_0$
  to arbitrary quasi-uniform spaces. \emph{New: the
  $T_0$-quotient $(\widetilde X,\widetilde\U)$ construction
  (Section~\ref{sec:quotient}) and the descent
  Lemma~\ref{lem:descent} that absorb the $T_0$ obstruction.}
\end{itemize}

\paragraph{Related work.} Contraction principles on
quasi-pseudometric and quasi-uniform spaces have been studied
from several directions that the present paper does not
duplicate. Alegre, Mar\'in, and
Romaguera~\cite{AlegreMarinRomaguera} prove a fixed-point
theorem for $w$-distance-based contractions on left $K$-complete
quasi-metric spaces, a setting strictly weaker than
bicompleteness but stronger in the contraction condition. Mar\'in,
Romaguera, and Tirado~\cite{MarinRomagueraTirado} treat set-valued
contractions on preordered quasi-metric spaces; the order
structure used there is independent of the quasi-uniform structure
and differs from the $\le_\U$ preorder employed in our
Knaster--Tarski variant. Secelean, Mathew, and
Wardowski~\cite{SeceleanMathewWardowski} establish $F$-contractions
on quasi-metric spaces, which are not subsumed by either the
linear or Boyd--Wong comparison-function regime treated here.
Sivri~\cite{Sivri} develops Ekeland-style variational principles on
$T_1$-quasi-uniform spaces, a complementary line that produces
approximate critical points rather than exact fixed points. The
constructive Knaster--Tarski theorems of Cousot and
Cousot~\cite{CousotCousot} and Esp\'inola and
Wi\'snicki~\cite{EspinolaWisnicki} are the classical-order-theoretic
predecessors of the variant we record in
Proposition~\ref{prop:knaster-tarski}. The multivalued
start-point theorem of Gaba, Karap{\i}nar, Petru\c{s}el, and
Radenovi\'c~\cite{GabaKarapinarPetruselRadenovic2020} is the
single-quasi-pseudometric instance of
Theorem~\ref{thm:hausdorff-direct} and is connected to it
explicitly in Remark~\ref{rem:hausdorff-vs-selector} and the
single-$d$ specialisation of
Theorem~\ref{thm:hausdorff-direct}.

\section{Preliminaries}\label{sec:prelim}

\paragraph{Notation.} We write $\U$ for a quasi-uniformity on a
set $X$, $\Uinv=\left\lbrace U^{-1}:U\in\U\right\rbrace$
for its conjugate, and $\Us=\U\vee\Uinv$ for its symmetrisation.
A generating family of quasi-pseudometrics is denoted
$\mathcal{D}=\left\lbrace d_\alpha\right\rbrace_{\alpha\in A}$.
For a multivalued map $T\colon X\to 2^X$, we use $z\in Tx$ for
membership in the image, $H$ or $H_\alpha$ for Hausdorff-type
functionals on subsets of $X$, and $\bigcap\U:=\bigcap_{U\in\U}U$
for the associated preorder. All of these conventions follow
\cite{Cobzas,FletcherLindgren,Kunzi,GabaKarapinarPetruselRadenovic2020,Gaba2018}.

We collect the structures used throughout: the quasi-uniformity
$\U$ and its associated preorder $\bigcap\U$, the symmetrisation
$\Us$ and bicompleteness, generation by quasi-pseudometrics, and
the startpoint--endpoint vocabulary of~\cite{Gaba2018}.

\subsection{Quasi-uniformities}
Throughout, $X$ is a non-empty set. A \emph{quasi-uniformity} on
$X$ is a filter $\U$ on $X\times X$ such that
\begin{enumerate}[label=(U\arabic*)]
  \item every $U\in\U$ contains the diagonal
        $\diag=\left\lbrace (x,x):x\in X\right\rbrace$;
  \item for every $U\in\U$ there exists $V\in\U$ with
        $V\circ V\subseteq U$, where
        \[
          V\circ V:=\left\lbrace (x,z)\in X\times X\;:\;
                       \exists\,y\in X\text{ with }(x,y)\in V\text{ and }(y,z)\in V\right\rbrace.
        \]
\end{enumerate}
The pair $(X,\U)$ is a \emph{quasi-uniform space} and the elements
of $\U$ are its \emph{entourages}. Each entourage encodes a
``closeness threshold'' that need not be symmetric: containing
$(x,y)$ in $U$ records that $x$ is close to $y$ in the forward
direction, with no commitment about the backward direction. To
recover symmetric information one passes to the \emph{conjugate}
(or \emph{dual}) quasi-uniformity $\Uinv=\left\lbrace U^{-1}:U\in\U \right \rbrace$, where
$U^{-1}=\left\lbrace (x,y):(y,x)\in U \right \rbrace$, and to the \emph{symmetrisation}
$\Us=\U\vee\Uinv$, generated by the entourages $U\cap U^{-1}$ as
$U$ ranges over $\U$. The symmetrisation $\Us$ is a genuine
uniformity and is the device for importing uniform-space results
into the asymmetric setting. For $x\in X$ and $U\in\U$, the
\emph{$U$-ball at $x$} is $U(x)=\left\lbrace y\in X:(x,y)\in U \right \rbrace$.

Intersecting all entourages produces a relation rather than a
single entourage. The set $\bigcap\U:=\bigcap_{U\in\U}U$ is the
graph of a preorder on $X$~\cite{Nachbin}; we call
$(X,\bigcap\U)$ the \emph{associated preorder} and write
$x\le_\U y$ for $(x,y)\in\bigcap\U$ when the preorder is being
named explicitly. The asymmetry of $\U$ is exactly what allows
this preorder to be non-trivial. The quasi-uniformity $\U$ is
\emph{$T_0$} when $\bigcap\U$ is antisymmetric, equivalently when
$\Us$ is Hausdorff~\cite[\S 2]{Kunzi}; the $T_0$ condition is what
licences the move from preorder relations $x\equiv_\U Tx$ to
genuine equalities $x=Tx$, and its absence is what
Section~\ref{sec:quotient} addresses via the $T_0$-quotient.

\subsection{Generation by quasi-pseudometrics}
A quasi-uniformity is often specified by giving a distance-like
function and reading off its entourages. A
\emph{quasi-pseudometric} on $X$ is a function
$d\colon X\times X\to[0,\infty)$ with $d(x,x)=0$ and the
triangle inequality $d(x,z)\le d(x,y)+d(y,z)$. The two axioms of
a metric that are not asked for here (symmetry $d(x,y)=d(y,x)$ and separation $d(x,y)=0\Rightarrow x=y$) are dropped
independently: omitting symmetry alone yields a
quasi-pseudometric, and reinstating separation while leaving
symmetry off recovers a quasi-metric.

Every family $\mathcal{D}$ of quasi-pseudometrics generates a
quasi-uniformity having
$\left\lbrace U_{d,\varepsilon}:d\in\mathcal{D},\,\varepsilon>0 \right \rbrace$ as a
subbase, where $U_{d,\varepsilon}=\left\lbrace (x,y):d(x,y)<\varepsilon \right \rbrace$
is the $\varepsilon$-fibre of $d$. Conversely~\cite{FletcherLindgren,Kunzi},
every quasi-uniformity is generated by a (saturated) family of
quasi-pseudometrics, and the generating family is \emph{countable}
when $\U$ has a countable base. This back-and-forth is what
licences us to state the contraction principle of
Section~\ref{sec:contraction} through quasi-pseudometric
inequalities while reading the conclusion at the level of the
quasi-uniformity.

\subsection{Nets, bi-Cauchy nets, and bicompleteness}
Convergence and completeness in a general (quasi-)uniform space
must be phrased in terms of nets, not sequences: the entourage
filter may have no countable base, and sequences may then fail to
detect convergence or completeness. We recall the vocabulary.

A \emph{directed set} is a partially ordered set $(\Lambda,\le)$
in which every pair of elements has a common upper bound: for
all $\lambda_1,\lambda_2\in\Lambda$ there is
$\lambda_3\in\Lambda$ with $\lambda_1\le\lambda_3$ and
$\lambda_2\le\lambda_3$. A \emph{net} on $X$ is a function
$\Lambda\to X$, $\lambda\mapsto x_\lambda$, indexed by a directed
set $\Lambda$, and we write it $(x_\lambda)_{\lambda\in\Lambda}$
or simply $(x_\lambda)$. Sequences are the special case
$\Lambda=\N$ with its usual order; arbitrary nets cover the
non-countable indexing situations one meets in general
uniformities. A net $(x_\lambda)$ in $(X,\U)$ is \emph{Cauchy
for $\U$} if for every $U\in\U$ there is $\lambda_0\in\Lambda$
with $(x_\lambda,x_\mu)\in U$ whenever
$\lambda_0\le\lambda\le\mu$; it \emph{converges} to $x\in X$ if
for every $U\in\U$ there is $\lambda_0$ with
$(x,x_\lambda)\in U$ for $\lambda\ge\lambda_0$. A net is
\emph{bi-Cauchy} if it is Cauchy with respect to the
symmetrisation $\Us$, that is, Cauchy for both $\U$ and $\Uinv$
simultaneously.

The space $(X,\U)$ is \emph{bicomplete} if every bi-Cauchy net
converges in $\Us$, equivalently if the uniform space $(X,\Us)$
is complete in the usual sense. When $\U$ is generated by
quasi-pseudometrics $\left\lbrace d_\alpha \right \rbrace_{\alpha\in A}$, the bi-Cauchy
condition unpacks explicitly: for every $\alpha\in A$ and every
$\varepsilon>0$ there is $\lambda_0$ with
$d_\alpha(x_\lambda,x_\mu)<\varepsilon$ and
$d_\alpha(x_\mu,x_\lambda)<\varepsilon$ whenever
$\lambda,\mu\ge\lambda_0$. The two inequalities reflect the two
directions of $d_\alpha$ that the symmetrisation tracks at once,
and they are what the proof of Theorem~\ref{thm:contraction}
verifies along the iterates of a contractive map.

\subsection{\texorpdfstring{$\U$}{U}-fixed points, startpoints, endpoints}
Fixed-point statements in symmetric settings ask for the equality
$x_0=Tx_0$, which in the asymmetric setting demands forward and
backward agreement at once and is therefore often too strong.
Gaba introduced two dual relaxations for quasi-pseudometric
spaces in~\cite{Gaba2014a,Gaba2014b,Gaba2015}, and lifted them to
arbitrary quasi-uniform spaces in~\cite{Gaba2018}, by separating
the two directions of $\bigcap\U$: a \emph{$\U$-startpoint}
demands proximity in the forward direction, a \emph{$\U$-endpoint}
in the reverse, and a \emph{$\U$-fixed point} is the
single-valued analogue of a $\U$-startpoint. The three notions
agree when $\U$ is symmetric, and become strictly finer
otherwise.

\begin{definition}[\cite{Gaba2018}]\label{def:fixed-start-end}
Let $(X,\U)$ be a quasi-uniform space.
\begin{enumerate}[label=(\alph*)]
  \item A point $x_0\in X$ is a \emph{$\U$-fixed point} of a
        single-valued map $T\colon X\to X$ if
        $(x_0,Tx_0)\in\bigcap\U$.
  \item A point $x_0\in X$ is a \emph{$\U$-startpoint} of a multivalued
        map $T\colon X\to 2^X$ if there exists $z\in Tx_0$ with
        $(x_0,z)\in\bigcap\U$. It is a \emph{$\U$-endpoint} if there
        exists $z\in Tx_0$ with $(z,x_0)\in\bigcap\U$.
\end{enumerate}
A point $z\in Tx_0$ satisfying $(x_0,z)\in\bigcap\U$ (respectively
$(z,x_0)\in\bigcap\U$) is called a \emph{startpoint witness}
(respectively an \emph{endpoint witness}) for $x_0$, and we say
that $x_0$ is a $\U$-startpoint \emph{witnessed by} $z$
(respectively a $\U$-endpoint \emph{witnessed by} $z$). The
witnessing element of a $\U$-startpoint need not be unique;
Example~\ref{ex:start-vs-end} below exhibits points with multiple
witnesses.
\end{definition}

\begin{remark}\label{rem:trivial-cases}
If both $(x_0,Tx_0)$ and $(Tx_0,x_0)$ belong to $\bigcap\U$ and
$\U$ is $T_0$, then $Tx_0=x_0$. Without the $T_0$ hypothesis
$\bigcap\U$ is only a preorder, and a $\U$-fixed point need not be
a fixed point of $T$ in the usual sense.
\end{remark}

The next example fixes intuition. It shows that, even on a
familiar space, the $\U$-startpoint, $\U$-endpoint and ordinary
fixed-point conditions can pick out genuinely different sets of
points.

\begin{example}[Startpoints, endpoints, and fixed points compared]\label{ex:start-vs-end}
Let $X=[0,1]$ with the upper quasi-pseudometric
$d(x,y)=\max\left\lbrace x-y,0 \right \rbrace$, so that $d(x,y)=0$ iff $x\le y$. Then
$\bigcap\U_d=\left\lbrace(x,y)\in X^2:x\le y\right \rbrace$ is the
standard order on $[0,1]$. For the multivalued map
\[
  T(x):=\left\lbrace x+\tfrac12,\;x-\tfrac12\right \rbrace\cap[0,1],
\]
which sends a point to its two reachable neighbours under a half-step,
the three notions split as follows:
\begin{enumerate}[label=\textup{(\arabic*)}]
  \item Every $x\in[0,\tfrac12]$ is a $\U$-startpoint of $T$, witnessed by
        $x+\tfrac12\in T(x)$, since $x\le x+\tfrac12$.
  \item Every $x\in[\tfrac12,1]$ is a $\U$-endpoint of $T$, witnessed by
        $x-\tfrac12\in T(x)$, since $x-\tfrac12\le x$.
  \item Only $x=\tfrac12$ is simultaneously a $\U$-startpoint and a
        $\U$-endpoint, witnessed by $z=1$ on the startpoint side and
        $z=0$ on the endpoint side. The point $\tfrac12$ is not a
        fixed point of $T$, since
        $\tfrac12\notin T(\tfrac12)=\left\lbrace 0,1\right \rbrace$.
\end{enumerate}
The $\U$-startpoints and $\U$-endpoints together cover $X$ but
overlap only at $\tfrac12$. Figure~\ref{fig:start-vs-end} shows
the situation. The example also illustrates the asymmetry of the
preorder: distance from $x$ to $x+\tfrac12$ is $0$ (the forward
direction is free), and distance from $x+\tfrac12$ to $x$ is
$\tfrac12$ (the backward direction is costly).
\end{example}

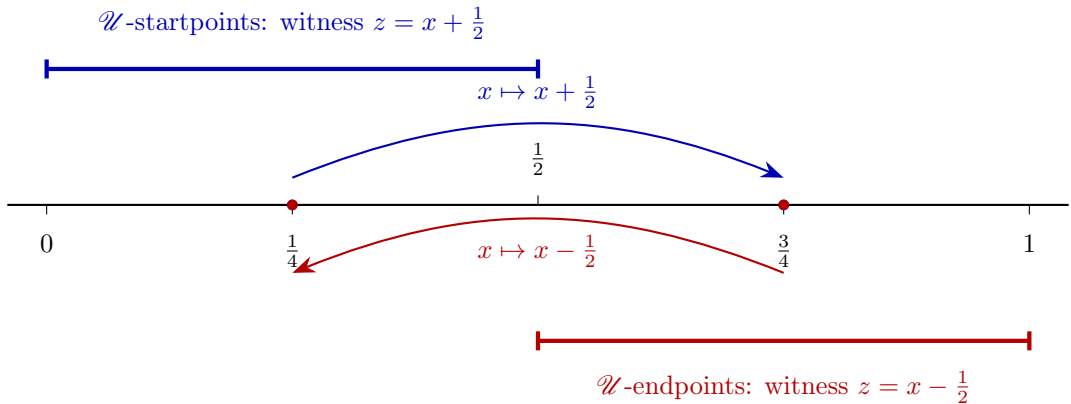
\begin{figure}[htbp]
\centering
\begin{tikzpicture}[x=13cm, y=1.8cm, >={Stealth[length=3mm]}]
  \draw[line width=1.6pt, blue!70!black] (0,1.00) -- (0.5,1.00);
  \draw[line width=1.6pt, blue!70!black] (0,0.93) -- (0,1.07);
  \draw[line width=1.6pt, blue!70!black] (0.5,0.93) -- (0.5,1.07);
  \node[above=4pt, font=\small, blue!70!black] at (0.25,1.07)
       {$\U$-startpoints: witness $z=x+\tfrac12$};
  \draw[->, thick, blue!70!black]
       (0.25,0.20) to[bend left=22]
       node[midway, above=2pt, font=\small] {$x\mapsto x+\tfrac12$}
       (0.75,0.20);
  \filldraw[blue!70!black] (0.25,0) circle (1.8pt);
  \filldraw[blue!70!black] (0.75,0) circle (1.8pt);
  \draw[thick] (-0.04,0) -- (1.04,0);
  \foreach \x/\lab in {0/$0$, 0.25/$\tfrac14$, 0.75/$\tfrac34$, 1/$1$} {
    \draw (\x,0) -- (\x,-0.07);
    \node[below=4pt, font=\small] at (\x,-0.07) {\lab};
  }
  \draw (0.5,0) -- (0.5,0.07);
  \node[above=4pt, font=\small] at (0.5,0.07) {$\tfrac12$};
  \draw[->, thick, red!70!black]
       (0.75,-0.50) to[bend right=22]
       node[midway, below=2pt, font=\small] {$x\mapsto x-\tfrac12$}
       (0.25,-0.50);
  \filldraw[red!70!black] (0.75,0) circle (1.8pt);
  \filldraw[red!70!black] (0.25,0) circle (1.8pt);
  \draw[line width=1.6pt, red!70!black] (0.5,-1.00) -- (1,-1.00);
  \draw[line width=1.6pt, red!70!black] (0.5,-0.93) -- (0.5,-1.07);
  \draw[line width=1.6pt, red!70!black] (1,-0.93) -- (1,-1.07);
  \node[below=4pt, font=\small, red!70!black] at (0.75,-1.07)
       {$\U$-endpoints: witness $z=x-\tfrac12$};
\end{tikzpicture}
\caption{The map
$T(x)=\left\lbrace x+\tfrac12,\,x-\tfrac12\right \rbrace\cap[0,1]$
on $X=[0,1]$ with the upper quasi-pseudometric. The blue bracket
above marks the set of $\U$-startpoints, $[0,\tfrac12]$, each
witnessed by its forward half-step; the red bracket below marks
the set of $\U$-endpoints, $[\tfrac12,1]$, each witnessed by its
backward half-step. The two sets meet only at $x=\tfrac12$, the
unique point with both witnesses available.}
\label{fig:start-vs-end}
\end{figure}

\section{An elaborated entourage formulation}\label{sec:trivial}

The existence framework of~\cite{Gaba2018} is built on a
non-emptiness condition for a certain intersection of entourage
fibres at the candidate point. We rewrite this condition through
the filter axioms~(U1)--(U2), placing it in direct correspondence
with the $\bigcap\U$-membership relation that defines $\U$-fixed
points and $\U$-startpoints in
Definition~\ref{def:fixed-start-end}. The resulting
characterisation feeds the $T_0$-quotient argument of
Section~\ref{sec:quotient} and identifies a verifiable structural
property of $T$ that supplies the contraction principle of
Section~\ref{sec:contraction}.

\begin{proposition}\label{prop:triv-single}
Let $(X,\U)$ be a quasi-uniform space, $T\colon X\to X$ and
$x^*\in X$. The following are equivalent:
\begin{enumerate}[label=\textup{(\roman*)}]
  \item $U(x^*)\cap U^{-1}(Tx^*)\neq\emptyset$ for every $U\in\U$;
  \item $(x^*,Tx^*)\in\bigcap\U$, i.e.\ $x^*$ is a $\U$-fixed point of $T$.
\end{enumerate}
\end{proposition}

\begin{proof}
(ii)$\Rightarrow$(i): Pick any $U\in\U$. Setting $y=x^*$ gives
$(x^*,y)=(x^*,x^*)\in\diag\subseteq U$ and $(y,Tx^*)=(x^*,Tx^*)\in\bigcap\U\subseteq U$,
so $y\in U(x^*)\cap U^{-1}(Tx^*)$.

(i)$\Rightarrow$(ii): Fix $W\in\U$ and pick $V\in\U$ with $V\circ V\subseteq W$
(axiom (U2)). By (i) applied to $V$, there is $y\in X$ with
$(x^*,y)\in V$ and $(y,Tx^*)\in V$, hence $(x^*,Tx^*)\in V\circ V\subseteq W$.
Since $W\in\U$ was arbitrary, $(x^*,Tx^*)\in\bigcap\U$.
\end{proof}

\begin{proposition}\label{prop:triv-multi}
Let $(X,\U)$ be a quasi-uniform space, $T\colon X\to 2^X$ and
$x^*\in X$ with $Tx^*\neq\emptyset$. The following are equivalent:
\begin{enumerate}[label=\textup{(\roman*)}]
  \item $U(x^*)\cap U^{-1}(z)\neq\emptyset$ for every $U\in\U$ and every
        $z\in Tx^*$;
  \item $(x^*,z)\in\bigcap\U$ for every $z\in Tx^*$.
\end{enumerate}
In particular, condition \textup{(i)} implies that \emph{every}
$z\in Tx^*$ witnesses $x^*$ as a $\U$-startpoint of $T$.
\end{proposition}

\begin{proof}
We prove both directions explicitly, paralleling the structure of
the proof of Proposition~\ref{prop:triv-single} but tracking the
pointwise variation in $z\in Tx^*$.

\emph{(ii) $\Rightarrow$ (i).} Suppose $(x^*,z)\in\bigcap\U$ for
every $z\in Tx^*$. Fix any $U\in\U$ and any $z\in Tx^*$. Setting
$y:=x^*$ gives $(x^*,y)=(x^*,x^*)\in\diag\subseteq U$ (axiom
(U1)) and $(y,z)=(x^*,z)\in\bigcap\U\subseteq U$. Hence
$y=x^*\in U(x^*)\cap U^{-1}(z)$, so this intersection is
non-empty.

\emph{(i) $\Rightarrow$ (ii).} Suppose
$U(x^*)\cap U^{-1}(z)\neq\emptyset$ for every $U\in\U$ and every
$z\in Tx^*$. Fix any $z\in Tx^*$ and any $W\in\U$. By axiom
(U2), there exists $V\in\U$ with $V\circ V\subseteq W$. This
$V$ depends only on $W$, not on $z$. By hypothesis (i) applied
to this $V$ and this $z$, there is $y\in X$ with
$(x^*,y)\in V$ and $(y,z)\in V$, hence
$(x^*,z)\in V\circ V\subseteq W$. Since $W\in\U$ was arbitrary,
$(x^*,z)\in\bigcap_{W\in\U}W=\bigcap\U$. The choice of $z$ was
also arbitrary, so $(x^*,z)\in\bigcap\U$ for every $z\in Tx^*$.

The ``in particular'' clause restates condition (ii): the
relation $(x^*,z)\in\bigcap\U$ is exactly the assertion that
$z$ witnesses $x^*$ as a $\U$-startpoint of $T$
(Definition~\ref{def:fixed-start-end}(b)).
\end{proof}

\begin{remark}\label{rem:original-results}
\cite[Theorems~2.1 and 2.4]{Gaba2018} establish the existence of
$\U$-fixed points and $\U$-startpoints under condition~(i) of
Propositions~\ref{prop:triv-single} and~\ref{prop:triv-multi}.
The propositions above place that condition in entourage-theoretic
form via the filter axioms, making the relationship between the
hypothesis and the preorder $\bigcap\U$ transparent. The next
section uses this elaboration to exhibit a verifiable structural
property of $T$, a contraction relative to a generating family of quasi-pseudometrics, that combined with a completeness
assumption on $(X,\U)$, implies condition~(i).
\end{remark}

\section{A contraction principle in quasi-uniform spaces}\label{sec:contraction}

Since a quasi-uniformity is generated by a family of
quasi-pseudometrics, the Banach principle adapts by asking the
contractive inequality to hold on each member of the family at
once. We call the resulting notion a $\mathcal{D}$-contraction.
Banach-type fixed-point theorems for asymmetric distance
structures form a well-developed literature, surveyed
in~\cite{Cobzas,Kunzi}; the condition below differs in two ways.
It engages a generating family $\mathcal{D}$ rather than a single
quasi-metric, with constants $k_\alpha$ that may vary with
$\alpha$; and it is conjugate-invariant, so the endpoint side of
Section~\ref{sec:endpoint} is reached without further hypotheses.

Throughout this section, $\U$ is generated by a family
$\mathcal{D}=\left\lbrace d_\alpha:\alpha\in A\right\rbrace$ of
quasi-pseudometrics on $X$. Such a family always
exists~\cite{FletcherLindgren,Kunzi}, and the arguments below use
no saturation property of $\mathcal{D}$.

\begin{definition}\label{def:contraction}
A self-map $T\colon X\to X$ is a \emph{$\mathcal{D}$-contraction} if
there is a family $\left\lbrace k_\alpha:\alpha\in A \right \rbrace\subseteq[0,1)$ such that,
for every $\alpha\in A$ and every $x,y\in X$,
\begin{equation}\label{eq:contraction}
  d_\alpha(Tx,Ty)\le k_\alpha\,d_\alpha(x,y)
  \quad\text{and}\quad
  d_\alpha(Ty,Tx)\le k_\alpha\,d_\alpha(y,x).
\end{equation}
\end{definition}

Both lines of~\eqref{eq:contraction} are used in the proof below:
the first controls bi-Cauchy behaviour, and the second secures
$\Us$-continuity of $T$. The condition is satisfied in all the
concrete settings considered in this paper:
Examples~\ref{ex:asymmetric-norm} and~\ref{ex:formal-balls} below,
and Examples~\ref{ex:discrete-selector}
and~\ref{ex:continuous-selector} of Section~\ref{sec:selector}.

\begin{theorem}\label{thm:contraction}
Let $(X,\U)$ be a $T_0$ bicomplete quasi-uniform space generated by a
family $\mathcal{D}=\left\lbrace d_\alpha \right \rbrace_{\alpha\in A}$ of
quasi-pseudometrics, and let $T\colon X\to X$ be a $\mathcal{D}$-contraction.
Then $T$ has a unique fixed point $x^*\in X$. Moreover, for every
$x_0\in X$ the sequence of iterates $(T^n x_0)_{n\ge0}$ converges to
$x^*$ in the symmetrisation $\Us$.
\end{theorem}

\begin{proof}
\emph{Existence.} Fix $x_0\in X$ and set $x_n=T^n x_0$ for $n\ge1$.
For each $\alpha\in A$ and $n\ge0$,
\[
  d_\alpha(x_n,x_{n+1})\le k_\alpha\, d_\alpha(x_{n-1},x_n)
  \le \cdots\le k_\alpha^n\, d_\alpha(x_0,x_1),
\]
and analogously $d_\alpha(x_{n+1},x_n)\le k_\alpha^n\,d_\alpha(x_1,x_0)$.
For $n<m$ the triangle inequality gives
\[
  d_\alpha(x_n,x_m)\le \sum_{i=n}^{m-1} d_\alpha(x_i,x_{i+1})
  \le \frac{k_\alpha^n}{1-k_\alpha}\, d_\alpha(x_0,x_1),
\]
and similarly
\[
  d_\alpha(x_m,x_n)\le \frac{k_\alpha^n}{1-k_\alpha}\, d_\alpha(x_1,x_0).
\]
Both bounds tend to $0$ as $n\to\infty$, uniformly in $m>n$. Hence
$(x_n)$ is bi-Cauchy. By bicompleteness, $x_n\to x^*$ in $\Us$ for
some $x^*\in X$.

For every $\alpha\in A$, the bounds
$d_\alpha(Tx,Ty)\le k_\alpha\,d_\alpha(x,y)\le d_\alpha(x,y)$ and
$d_\alpha(Ty,Tx)\le k_\alpha\,d_\alpha(y,x)\le d_\alpha(y,x)$ make
$T$ non-expansive in each $d_\alpha$ and in each $d_\alpha^{-1}$,
hence non-expansive in the symmetric pseudometric
$d_\alpha^{s}:=\max\left\lbrace d_\alpha,d_\alpha^{-1} \right \rbrace$. Since the family
$\left\lbrace d_\alpha^{s} \right \rbrace_{\alpha\in A}$ generates the symmetrisation
$\Us$, $T$ is uniformly continuous on $(X,\Us)$. The previous step
gave $x_n\to x^*$ in $\Us$, so $Tx_n\to Tx^*$ in $\Us$. Combining
this with $Tx_n=x_{n+1}\to x^*$ and the Hausdorff property of
$\Us$ (which holds because $\U$ is $T_0$; this is the only step
in the existence \emph{half} of the proof that uses the $T_0$
hypothesis, which is also invoked below in the uniqueness step),
we obtain $Tx^*=x^*$. (Without $T_0$ the same argument delivers only
$Tx^*\equiv_\U x^*$, the conclusion handled in
Section~\ref{sec:quotient}.)

\emph{Uniqueness.} If $Tx^*=x^*$ and $Ty^*=y^*$, then for every
$\alpha\in A$,
\[
  d_\alpha(x^*,y^*)=d_\alpha(Tx^*,Ty^*)\le k_\alpha\, d_\alpha(x^*,y^*),
\]
forcing $d_\alpha(x^*,y^*)=0$, and likewise $d_\alpha(y^*,x^*)=0$.
Hence $(x^*,y^*)\in\bigcap\U\cap\bigcap\Uinv$, so $x^*=y^*$ by $T_0$.
\end{proof}

\begin{corollary}[Picard error bound]\label{cor:picard}
Under the hypotheses of Theorem~\ref{thm:contraction}, for every
$x_0\in X$, every $\alpha\in A$ and every $n\ge 0$,
\[
  d_\alpha(T^n x_0,x^*)\le \frac{k_\alpha^n}{1-k_\alpha}\,d_\alpha(x_0,Tx_0),
  \qquad
  d_\alpha(x^*,T^n x_0)\le \frac{k_\alpha^n}{1-k_\alpha}\,d_\alpha(Tx_0,x_0).
\]
\end{corollary}

\begin{proof}
Set $x_n=T^n x_0$. The Cauchy estimate derived in the proof of
Theorem~\ref{thm:contraction} gives, for all $m>n\ge 0$,
\[
  d_\alpha(x_n,x_m)\le \frac{k_\alpha^n}{1-k_\alpha}\,d_\alpha(x_0,x_1).
\]
By the triangle inequality,
$d_\alpha(x_n,x^*)\le d_\alpha(x_n,x_m)+d_\alpha(x_m,x^*)$
for every $m\ge n$. Since $d_\alpha\le d_\alpha^{s}=\max\left\lbrace d_\alpha,d_\alpha^{-1}\right\rbrace$
pointwise on $X\times X$, and $x_m\to x^*$ in $\Us$ gives
$d_\alpha^{s}(x_m,x^*)\to 0$, we have
$d_\alpha(x_m,x^*)\to 0$ as $m\to\infty$. Taking
$\limsup_{m\to\infty}$ on both sides of the triangle inequality,
$d_\alpha(x_n,x^*)\le k_\alpha^n(1-k_\alpha)^{-1}d_\alpha(x_0,x_1)+0
=k_\alpha^n(1-k_\alpha)^{-1}d_\alpha(x_0,x_1)$.
The second inequality follows from the analogous Cauchy estimate
$d_\alpha(x_m,x_n)\le k_\alpha^n(1-k_\alpha)^{-1}d_\alpha(x_1,x_0)$
by the same passage to the limit $m\to\infty$.
\end{proof}

\begin{example}[Theorem~\ref{thm:contraction} and Corollary~\ref{cor:picard} in action]\label{ex:picard-numerical}
Take $X=\R$ with the upper quasi-pseudometric
$d(x,y)=(x-y)^+$, so the symmetrisation is the Euclidean
uniformity and $(\R,\U_d)$ is $T_0$ and bicomplete. Consider the
affine map $T(x):=\tfrac{x}{2}$. For all $x,y\in\R$,
\[
  d(Tx,Ty)=\left(\tfrac{x-y}{2}\right)^+=\tfrac12\,d(x,y),
  \qquad
  d(Ty,Tx)=\tfrac12\,d(y,x),
\]
so both lines of~\eqref{eq:contraction} hold with $k=\tfrac12$
and $T$ is a $\left\lbrace d \right \rbrace$-contraction. The equation $T(x^*)=x^*$ forces
$x^*=0$, the fixed point produced by Theorem~\ref{thm:contraction}.

\emph{Starting from $x_0=1$ (above the fixed point).} The iterates
are $T^n(1)=2^{-n}$, decreasing monotonically to $0$. The seed
distances are $d(x_0,Tx_0)=d(1,\tfrac12)=\tfrac12$ and
$d(Tx_0,x_0)=d(\tfrac12,1)=0$. Corollary~\ref{cor:picard} gives
\[
  d(T^n(1),0)\le\frac{\left(\tfrac{1}{2}\right)^n}{1-\tfrac{1}{2}}\cdot\tfrac12=\left(\tfrac{1}{2}\right)^n,
  \qquad
  d(0,T^n(1))\le\frac{\left(\tfrac{1}{2}\right)^n}{1-\tfrac{1}{2}}\cdot 0=0,
\]
which match the exact values $d(2^{-n},0)=2^{-n}$ and
$d(0,2^{-n})=0$: both bounds are attained with equality.

\emph{Starting from $x_0=-1$ (below the fixed point).} Now
$T^n(-1)=-2^{-n}$ increases monotonically to $0$. The seed
distances swap roles: $d(x_0,Tx_0)=d(-1,-\tfrac12)=0$ and
$d(Tx_0,x_0)=d(-\tfrac12,-1)=\tfrac12$. The two halves of
Corollary~\ref{cor:picard} now give
\[
  d(T^n(-1),0)\le 0,\qquad d(0,T^n(-1))\le\left(\tfrac{1}{2}\right)^n,
\]
again matching the exact values $d(-2^{-n},0)=0$ and
$d(0,-2^{-n})=2^{-n}$.

The same contraction, run from opposite sides of the fixed point,
exhibits the asymmetric Picard bound on opposite sides: from
above, the forward direction $d$ carries the geometric decay and
the conjugate direction sees nothing; from below, the conjugate
direction carries the decay and the forward direction is trivial.
The two halves of Corollary~\ref{cor:picard} are therefore
genuinely independent statements, each tight in its own regime,
and neither is implied by the other on a general
quasi-pseudometric space.
\end{example}

\begin{example}[A family $\mathcal{D}$ with rates varying across $\alpha$]\label{ex:family-multi-rate}
The hypothesis of Theorem~\ref{thm:contraction} asks for a
\emph{family} $\mathcal{D}=\left\lbrace d_\alpha \right \rbrace_{\alpha\in A}$, with
contraction constants $\left\lbrace k_\alpha \right \rbrace$ allowed to vary across the
index. The next example shows that this freedom is essential: a
single quasi-pseudometric realising $\U$ does not capture the
structure, and the rate genuinely differs across $\alpha$.

Take $X=\R^2$ with the two quasi-pseudometrics
\[
  d_1\left((x_1,x_2),(y_1,y_2)\right):=(x_1-y_1)^+,
  \quad
  d_2\left((x_1,x_2),(y_1,y_2)\right):=(x_2-y_2)^+,
\]
each reading off one coordinate. The symmetrisations
$d_i^s=|x_i-y_i|$ together generate the Euclidean uniformity on
$\R^2$, so $(\R^2,\U_{\mathcal{D}})$ with
$\mathcal{D}=\left\lbrace d_1,d_2 \right \rbrace$ is $T_0$ and bicomplete. Neither $d_i$
alone is $T_0$, since each ignores one coordinate; the two
together are needed.

Define $T(x_1,x_2):=(\tfrac{x_1}{2},\,\tfrac{x_2}{3})$. A direct computation gives
\[
  d_1(Tx,Ty)=\tfrac12\,d_1(x,y),\qquad
  d_2(Tx,Ty)=\tfrac13\,d_2(x,y),
\]
with the analogous conjugate inequalities, so $T$ is a
$\mathcal{D}$-contraction with $k_{d_1}=\tfrac12\ne\tfrac13=k_{d_2}$.
The unique fixed point is $x^*=(0,0)$, and the iterates from
$x_0=(1,1)$ are $T^n(1,1)=(2^{-n},3^{-n})$.
Corollary~\ref{cor:picard} gives the two bounds
\[
  d_1\left(T^n(1,1),x^*\right)\le 2^{-n},\qquad
  d_2\left(T^n(1,1),x^*\right)\le 3^{-n},
\]
each attained with equality. The two coordinates carry genuinely
different geometric rates; the family-indexed estimate records
both simultaneously. A single-quasi-metric formulation forcing one constant $k$ to govern the whole map would have to take
$k=\max\left\lbrace k_{d_1},k_{d_2} \right \rbrace=\tfrac12$, losing the strictly faster
decay $3^{-n}$ in the second coordinate. The convergence in
$\Us$ inherits the slower of the two rates, but the per-$\alpha$
estimates of Corollary~\ref{cor:picard} are sharper.
\end{example}

\begin{remark}\label{rem:vs-2.1}
Theorem~\ref{thm:contraction} attaches to the entourage framework
of~\cite{Gaba2018}, as recast in
Propositions~\ref{prop:triv-single} and~\ref{prop:triv-multi}, an
existence statement whose two structural ingredients can be checked
directly from the data $(X,\U,\mathcal{D},T)$: a contraction
property of $T$ relative to a generating family $\mathcal{D}$, and
bicompleteness of $(X,\U)$.
\end{remark}

\begin{remark}[Relation to existing contraction principles in asymmetric settings]\label{rem:vs-prior}
Contraction principles for asymmetric distance structures form a
well-developed literature; the modern reference is the
monograph~\cite{Cobzas}, and the quasi-uniform survey~\cite{Kunzi}
collects the underlying machinery. Theorem~\ref{thm:contraction}
differs from the standard quasi-metric Banach contractions in two
respects. First, it works directly with a generating family
$\mathcal{D}$ of quasi-pseudometrics, with constants $k_\alpha$
that may vary with $\alpha$ (Example~\ref{ex:product}); the
conclusion is therefore a fixed-point statement at the level of the
quasi-uniformity $\U$, rather than in a single quasi-metric
realising it. Second, the two-sided
form~\eqref{eq:contraction} is automatically invariant under
$\mathcal{D}\mapsto\mathcal{D}^{-1}$ (Remark~\ref{rem:sharpness}),
which feeds the endpoint conjugate of
Section~\ref{sec:endpoint} without further hypotheses.
\end{remark}

\begin{remark}[Where the substantive content of Theorem~\ref{thm:contraction} lives]\label{rem:reduction-single-valued}
The two-sided condition~\eqref{eq:contraction} implies, by
$d_\alpha^s=\max\left\lbrace d_\alpha,d_\alpha^{-1} \right \rbrace$,
$d_\alpha^s(Tx,Ty)\le k_\alpha\,d_\alpha^s(x,y)$ for every
$\alpha\in A$ and every $x,y\in X$. Hence $T$ is a Banach
contraction on the $T_0$ complete uniform space $(X,\Us)$
generated by $\left\lbrace d_\alpha^s \right \rbrace_{\alpha\in A}$, and the
\emph{existence} and \emph{uniqueness} of the fixed point in
Theorem~\ref{thm:contraction} follow from the standard uniform-space
Banach principle applied to $\Us$. The substantive
quasi-uniform content of the single-valued theorem therefore lives
elsewhere: in the per-direction Picard bound of
Corollary~\ref{cor:picard}, which the symmetric estimate on $\Us$
fuses into a single symmetric bound and loses; in the
conjugate-invariance of~\eqref{eq:contraction} that makes the
endpoint dualisation of Section~\ref{sec:endpoint} automatic; and
in the family-of-quasi-pseudometrics form with constants $k_\alpha$
varying across $\alpha$ (Examples~\ref{ex:family-multi-rate}
and~\ref{ex:product}), which the symmetric estimate on $\Us$
collapses to the slowest rate $\max_\alpha k_\alpha$. The
multivalued analogue in
Definition~\ref{def:weakly-contractive-multi} does not admit the
same reduction: the existential witness $\eta\in T\xi$ realising
the forward inequality of the definition need not realise the
backward one at the same pair, and the swap-closure that trivialises
the single-valued case fails. This asymmetry between the
single-valued and multivalued theorems is why the substantive
novelty of the paper accumulates on the multivalued side
(Theorem~\ref{thm:hausdorff-direct}) and on the direction-tracking
per-$\alpha$ bookkeeping.
\end{remark}

\begin{remark}[Conjugate symmetry of \eqref{eq:contraction}]\label{rem:sharpness}
For the single-valued condition~\eqref{eq:contraction}, the two
lines are not independent. Quantifying the first line
$d_\alpha(Tx,Ty)\le k_\alpha\,d_\alpha(x,y)$ over all
$(x,y)\in X\times X$ and then swapping $(x,y)\leftrightarrow(y,x)$
gives $d_\alpha(Ty,Tx)\le k_\alpha\,d_\alpha(y,x)$, i.e.\ the
second line. Conversely the second line implies the first. Both
lines hold over $X\times X$ if and only if either one does, so
the second line of~\eqref{eq:contraction} is logically
redundant. We retain it for two reasons. First, the two-line
form makes the proof of Theorem~\ref{thm:contraction} read in
parallel for $d_\alpha$ and $d_\alpha^{-1}$ at every step,
which is the natural bookkeeping for the asymmetric setting and
clarifies which estimate feeds which Cauchy bound. Second, the
analogous condition for multivalued maps in
Definition~\ref{def:weakly-contractive-multi} is genuinely
two-sided. There the existential witness $\eta\in T\xi$ in the
first line need not work for the second line at the same pair,
and the swap-closure argument fails, so the two-line shape
provides a uniform template across the single-valued and
multivalued settings. The conjugate-invariance of
$\mathcal{D}\mapsto\mathcal{D}^{-1}$ that drives
Section~\ref{sec:endpoint} is therefore automatic in the
single-valued case and is the substantive content of the
formulation in the multivalued case.
\end{remark}

\begin{remark}[Mixing of forward and conjugate at the iterate level]\label{rem:iterate-coupling}
For a $\mathcal{D}$-contraction $T$, applying the first line
of~\eqref{eq:contraction} at the iterate pair $(x_n,x_{n+1})$
gives $d_\alpha(x_{n+1},x_{n+2})\le k_\alpha\,d_\alpha(x_n,x_{n+1})$,
while applying it at $(x_{n+1},x_n)$, which is the same inequality
by Remark~\ref{rem:sharpness}, gives
$d_\alpha(x_{n+2},x_{n+1})\le k_\alpha\,d_\alpha(x_{n+1},x_n)$.
The two iterated estimates control $d_\alpha$ and
$d_\alpha^{-1}$ respectively along the iteration and together
deliver bi-Cauchyness in $\Us$. $\Us$-continuity of $T$ then
follows from non-expansiveness in each $d_\alpha^s$ via the
asymmetric triangle inequality.
\end{remark}

The next two examples illustrate the role played by the
completeness and strict-contraction hypotheses of
Theorem~\ref{thm:contraction}: dropping either invalidates the
conclusion.

\begin{example}[Failure without bicompleteness]\label{ex:non-bicomplete}
\sloppy
Take $X=(0,1]$ with the upper quasi-pseudometric
$d(x,y)=(x-y)^+$. The symmetrisation is the standard Euclidean
metric, in which $(0,1]$ fails to be closed: the bi-Cauchy
sequence $x_n=\tfrac{1}{n}$ has no limit in $X$. So $(X,\U_d)$ is $T_0$
but not bicomplete. The map $T(x):=\tfrac{x}{2}$ is monotone and
$\tfrac12$-Lipschitz on $\R$, hence a $\left\lbrace d \right \rbrace$-contraction with
constant $\tfrac12$ (Example~\ref{ex:asymmetric-norm}). The
iterates from any $x_0\in(0,1]$ are $T^n(x_0)=\tfrac{x_0}{2^n}$, tending
to $0$ in $\Us$; since $0\notin X$, $T$ has no fixed point in $X$. So
every hypothesis of Theorem~\ref{thm:contraction} other than
bicompleteness is satisfied, yet the conclusion fails: the
candidate limit escapes the space.
\end{example}

\begin{example}[Strict contraction is essential]\label{ex:non-strict}
On $X=\R$ with the upper quasi-pseudometric $d(x,y)=(x-y)^+$, the
translation $T(x):=x+1$ satisfies $d(Tx,Ty)=d(x,y)$ and
$d(Ty,Tx)=d(y,x)$ for every $x,y\in\R$, so both lines
of~\eqref{eq:contraction} hold with constant $k=1$. The map is
non-expansive in $d$ and in $d^{-1}$ but has no fixed point: the
iterates $T^n(x_0)=x_0+n$ diverge to $+\infty$ in $\Us$, and
$x+1=x$ has no solution in $\R$. This shows that the hypothesis
$\left\lbrace k_\alpha \right \rbrace\subseteq[0,1)$ in Theorem~\ref{thm:contraction}
cannot be relaxed to $\left\lbrace k_\alpha \right \rbrace\subseteq[0,1]$: the strict
inequality $k_\alpha<1$ is what drives the geometric decay of
consecutive distances and, with it, bi-Cauchyness of the
iterates.
\end{example}

\subsection{A Boyd--Wong-type extension}\label{subsec:boyd-wong}

Boyd and Wong~\cite{BoydWong} weakened the Banach contraction
condition on a metric space by replacing the constant $k$ with a
function $\varphi$ that compresses positive distances; the
Matkowski class~\cite{Matkowski} gives the more general
framework in which $\varphi$ need only be nondecreasing and
satisfy $\varphi^n(t)\to 0$. The same substitution makes sense in
our setting: replace each $k_\alpha$ in~\eqref{eq:contraction} by a
comparison function $\varphi_\alpha$ satisfying the standard
Boyd--Wong conditions. The conclusion of
Theorem~\ref{thm:contraction} survives, with one cost: the
geometric Picard estimate of Corollary~\ref{cor:picard} weakens to
a non-geometric rate determined by $\varphi_\alpha$.

\begin{definition}\label{def:boyd-wong}
A self-map $T\colon X\to X$ is a \emph{$\mathcal{D}$-Boyd--Wong
contraction} if there is a family $\left\lbrace \varphi_\alpha:\alpha\in A \right \rbrace$
of functions $\varphi_\alpha\colon[0,\infty)\to[0,\infty)$ such
that each $\varphi_\alpha$ is nondecreasing, right upper
semicontinuous (that is, $\limsup_{s\to t^+}\varphi_\alpha(s)\le\varphi_\alpha(t)$
for every $t\ge 0$), $\varphi_\alpha(0)=0$, and $\varphi_\alpha(t)<t$
for every $t>0$, and such that, for every $\alpha\in A$ and every
$x,y\in X$,
\begin{equation}\label{eq:bw-contraction}
  d_\alpha(Tx,Ty)\le \varphi_\alpha\left(d_\alpha(x,y)\right)
  \quad\text{and}\quad
  d_\alpha(Ty,Tx)\le \varphi_\alpha\left(d_\alpha(y,x)\right).
\end{equation}
\end{definition}

A $\mathcal{D}$-contraction with constants $k_\alpha$ is the
special case $\varphi_\alpha(t)=k_\alpha t$. The Boyd--Wong class
also covers maps whose contraction rate degrades near the fixed
point, such as $\varphi(t)=t-t^2$ on $[0,\tfrac12]$, where the
ratio $\varphi(t)/t$ tends to $1$ as $t\to 0$. A different
weakening of the linear case is the $(c)^*$-comparison condition
of~\cite{GabaKarapinarPetruselRadenovic2020}, which replaces the
right-USC hypothesis on $\varphi_\alpha$ by the summability
implication $\sum\gamma(t_n)<\infty\Rightarrow\sum t_n<\infty$;
both families coincide with the Banach regime when
$\gamma(t)=kt$, but neither contains the other in general.

\begin{theorem}[Boyd--Wong-type contraction principle]\label{thm:boyd-wong}
Let $(X,\U)$ be a $T_0$ bicomplete quasi-uniform space generated by
a family $\mathcal{D}=\left\lbrace d_\alpha \right \rbrace_{\alpha\in A}$ of
quasi-pseudometrics, and let $T\colon X\to X$ be a
$\mathcal{D}$-Boyd--Wong contraction. Then $T$ has a unique fixed
point $x^*\in X$, and for every $x_0\in X$ the iterates $T^n x_0$
converge to $x^*$ in $\Us$.
\end{theorem}

\begin{proof}
\emph{Step 1: Decay of consecutive distances.} Fix $\alpha\in A$
and $x_0\in X$, and set $a_n:=d_\alpha(x_n,x_{n+1})$ and
$b_n:=d_\alpha(x_{n+1},x_n)$ with $x_n=T^nx_0$. Applying the first
line of~\eqref{eq:bw-contraction} to the pair $(x_{n-1},x_n)$,
\[
  a_n=d_\alpha(Tx_{n-1},Tx_n)\le\varphi_\alpha\left(d_\alpha(x_{n-1},x_n)\right)=\varphi_\alpha(a_{n-1}),
\]
and applying the second line to the same pair gives
$b_n\le\varphi_\alpha(b_{n-1})$. Since $\varphi_\alpha(t)<t$ for
$t>0$, both $(a_n)$ and $(b_n)$ are nonincreasing and converge to
some limits $\ell,\ell'\ge 0$. Right upper semicontinuity of
$\varphi_\alpha$ at $\ell$ gives $\ell\le\varphi_\alpha(\ell)$,
forcing $\ell=0$; the analogous argument for $b_n$ gives
$\ell'=0$. Hence $a_n\to 0$ and $b_n\to 0$.

\emph{Step 2: Bi-Cauchyness of the iterates.} Suppose for
contradiction that $(x_n)$ is not Cauchy in $d_\alpha$. Then there
exist $\epsilon>0$ and an infinite family of pairs of indices
whose first coordinates tend to $\infty$ realising
$d_\alpha(\cdot,\cdot)\ge\epsilon$; from it extract a strictly
increasing sequence $n_k\to\infty$ such that for each $k$ some
$m>n_k$ satisfies $d_\alpha(x_{n_k},x_m)\ge\epsilon$. For each $k$
pick $m_k$ minimal with this property, well-defined by the
well-ordering of $\N$, so that
$d_\alpha(x_{n_k},x_j)<\epsilon$ for $n_k<j<m_k$. For all
sufficiently large $k$ we have $m_k\ge n_k+2$, since otherwise
$d_\alpha(x_{n_k},x_{m_k})=a_{n_k}\to 0$ along a subsequence by
Step~1, contradicting the lower bound $\epsilon$. The triangle
inequality gives
\begin{align*}
  \epsilon
  &\le d_\alpha(x_{n_k},x_{m_k})\\
  &\le d_\alpha(x_{n_k},x_{m_k-1})+a_{m_k-1}\\
  &<\epsilon+a_{m_k-1},
\end{align*}
\sloppy
The lower bound is by the minimal choice of $m_k$. The upper
bound combines $d_\alpha(x_{n_k},x_{m_k-1})<\epsilon$, which
holds by the same minimality (since $n_k<m_k-1<m_k$), with the
triangle inequality at $x_{m_k-1}$. Since $a_{m_k-1}\to 0$ by
Step~1, the sandwich forces
$d_\alpha(x_{n_k},x_{m_k})\to\epsilon$, with approach through
values $\ge\epsilon$. The right upper semicontinuity of
$\varphi_\alpha$ at $\epsilon$ applies precisely to sequences
converging to $\epsilon$ from above.
\fussy

On the other hand, the triangle inequality applied at
$x_{n_k+1}$ and then at $x_{m_k+1}$ gives
\begin{align*}
  d_\alpha(x_{n_k},x_{m_k})
  &\le d_\alpha(x_{n_k},x_{n_k+1})+d_\alpha(x_{n_k+1},x_{m_k+1})+d_\alpha(x_{m_k+1},x_{m_k})\\
  &=a_{n_k}+d_\alpha(x_{n_k+1},x_{m_k+1})+b_{m_k},
\end{align*}
where $a_{n_k}=d_\alpha(x_{n_k},x_{n_k+1})$ and $b_{m_k}=d_\alpha(x_{m_k+1},x_{m_k})$.
Combined with $a_{n_k},b_{m_k}\to 0$ from Step~1,
\[
  \liminf_{k\to\infty}d_\alpha(x_{n_k+1},x_{m_k+1})\ge\epsilon.
\]
The first line of~\eqref{eq:bw-contraction}, applied to
$(x_{n_k},x_{m_k})$, gives
$d_\alpha(x_{n_k+1},x_{m_k+1})\le\varphi_\alpha(d_\alpha(x_{n_k},x_{m_k}))$,
and right upper semicontinuity of $\varphi_\alpha$ at $\epsilon$,
applied to the sequence $d_\alpha(x_{n_k},x_{m_k})$ approaching
$\epsilon$ from above, yields
\[
  \limsup_{k\to\infty}d_\alpha(x_{n_k+1},x_{m_k+1})
  \le\varphi_\alpha(\epsilon)<\epsilon.
\]
This contradicts the previous $\liminf$. Hence $(x_n)$ is Cauchy
in $d_\alpha$. The second line of~\eqref{eq:bw-contraction} is
exactly the statement
$d_\alpha^{-1}(Tx,Ty)\le\varphi_\alpha(d_\alpha^{-1}(x,y))$,
i.e.\ $T$ is also a Boyd--Wong $\varphi_\alpha$-contraction with
respect to the conjugate quasi-pseudometric $d_\alpha^{-1}$; the
same Step~2 argument applied to $d_\alpha^{-1}$ shows that
$(x_n)$ is Cauchy in $d_\alpha^{-1}$. Therefore $(x_n)$ is
bi-Cauchy, and by
bicompleteness $x_n\to x^*$ in $\Us$ for some $x^*\in X$.

\emph{Step 3: $x^*$ is a fixed point.} The condition
$\varphi_\alpha(t)\le t$ together with~\eqref{eq:bw-contraction}
makes $T$ non-expansive in every $d_\alpha$ and in the conjugate
$d_\alpha^{-1}$. The map $T$ is therefore uniformly continuous on
$(X,\Us)$, so $Tx_n\to Tx^*$ in $\Us$. Combined with
$Tx_n=x_{n+1}\to x^*$ and the Hausdorff property of $\Us$ (which
holds because $\U$ is $T_0$), we obtain $Tx^*=x^*$.

\emph{Step 4: Uniqueness.} Suppose $Ty^*=y^*$ for some $y^*\in X$.
Write $t:=d_\alpha(x^*,y^*)$. Applying~\eqref{eq:bw-contraction}
to the pair $(x^*,y^*)$,
\[
  t=d_\alpha(Tx^*,Ty^*)\le\varphi_\alpha(t).
\]
If $t>0$ then $\varphi_\alpha(t)<t$, contradicting the
inequality. Hence $d_\alpha(x^*,y^*)=0$ for
every $\alpha\in A$, and analogously $d_\alpha(y^*,x^*)=0$ for
every $\alpha$. Therefore $(x^*,y^*),(y^*,x^*)\in\bigcap\U$, and
$T_0$ of $\U$ gives $x^*=y^*$.
\end{proof}

\begin{remark}\label{rem:bw-rate}
The Picard error bound of Corollary~\ref{cor:picard} has no direct
geometric analogue under Theorem~\ref{thm:boyd-wong}:
$d_\alpha(T^n x_0,x^*)\to 0$ holds, but the rate is governed by
the iterates $\varphi_\alpha^n(d_\alpha(x_0,Tx_0))$, which need
not be summable.
\end{remark}

\begin{remark}[Reduction of Theorem~\ref{thm:boyd-wong} to Boyd--Wong on $\Us$]\label{rem:reduction-boyd-wong}
The reduction observed in Remark~\ref{rem:reduction-single-valued}
propagates to the Boyd--Wong case. By monotonicity of
$\varphi_\alpha$ and $d_\alpha,d_\alpha^{-1}\le d_\alpha^s$,
\[
  d_\alpha^s(Tx,Ty)
  =\max\left\lbrace d_\alpha(Tx,Ty),d_\alpha(Ty,Tx)\right \rbrace
  \le\varphi_\alpha\left(d_\alpha^s(x,y)\right),
\]
so $T$ is a Boyd--Wong $\varphi_\alpha$-contraction on the $T_0$
complete uniform space $(X,\Us)$, and the existence and uniqueness
clauses of Theorem~\ref{thm:boyd-wong} follow from the standard
Boyd--Wong theorem on $\Us$. The genuine quasi-uniform content of
Theorem~\ref{thm:boyd-wong} lives in the per-direction bookkeeping
of the mixed-rate example (Example~\ref{ex:bw-family}) and in the
conjugate-invariant form of Definition~\ref{def:boyd-wong} that
feeds the endpoint corollary
(Corollary~\ref{cor:boyd-wong-endpoint}); the multivalued
Definition~\ref{def:weakly-contractive-multi} again does not reduce
symmetrically, and it is on that side of the paper that the
family-and-two-sided formalism is essential rather than convenient.
\end{remark}

\begin{remark}[A converse direction]\label{rem:converse}
One can ask the question in reverse: on which quasi-uniform spaces
does every $\mathcal{D}$-contraction have a fixed point? The
metric-space analogue has been studied since Bessaga's converse
to the Banach principle~\cite{Bessaga} and refined in several
directions; in each version the class of test maps characterises
a corresponding completeness notion. A quasi-uniform analogue should likewise characterise
bicompleteness relative to $\mathcal{D}$, but a precise statement
requires care about which test class to use, and we do not pursue
the question here.
\end{remark}

\begin{example}[A Boyd--Wong contraction with sub-geometric rate]\label{ex:boyd-wong-logistic}
Take $X=[0,\tfrac12]$ with the upper quasi-pseudometric
$d(x,y)=(x-y)^+$ of Example~\ref{ex:asymmetric-norm}, restricted
to $[0,\tfrac12]$. The symmetrisation is the Euclidean metric, in
which $[0,\tfrac12]$ is closed in $\R$, so $(X,\U_d)$ is $T_0$
and bicomplete. Define $T(x):=x-x^2=x(1-x)$; the image lies in
$[0,\tfrac14]\subseteq[0,\tfrac12]$, so $T$ is well-defined and
monotone nondecreasing on $X$. Take
\[
  \varphi(t):=\begin{cases} t-t^2 & t\in[0,\tfrac12],\\
                            \tfrac14 & t\ge\tfrac12,\end{cases}
\]
which is nondecreasing, continuous (hence right upper
semicontinuous), $\varphi(0)=0$, and $\varphi(t)<t$ for $t>0$.
For $x\ge y$ in $[0,\tfrac12]$ the factorisation
$Tx-Ty=(x-y)\left(1-(x+y)\right)$ together with
$x+y\ge x-y$ (equivalent to $y\ge 0$) gives
\[
  d(Tx,Ty)=(x-y)\left(1-(x+y)\right)\le(x-y)\left(1-(x-y)\right)=\varphi\left(d(x,y)\right),
\]
and the conjugate inequality
$d(Ty,Tx)\le\varphi(d(y,x))$ follows from the same factorisation
with the arguments swapped. For $x<y$ both sides vanish. So $T$
is a $\mathcal{D}$-Boyd--Wong contraction with $\mathcal{D}=\left\lbrace d \right \rbrace$.

The fixed-point equation $x^*-(x^*)^2=x^*$ forces $x^*=0$, so
Theorem~\ref{thm:boyd-wong} gives the Picard iterates
$x_{n+1}=x_n(1-x_n)$ converging to $0$ from every $x_0\in X$. The
rate, however, is sub-geometric: the standard analysis of the
logistic recursion at parameter~$1$ gives $x_n\sim \tfrac{1}{n}$, so
$x_{n+1}/x_n\to 1$ and no constant $k<1$ controls $d(T^n x_0,0)$.
The example exhibits the rate behaviour announced in
Remark~\ref{rem:bw-rate}: bi-Cauchyness and convergence hold, but
the geometric Picard bound of Corollary~\ref{cor:picard} no
longer applies.
\end{example}

\begin{example}[Family $\mathcal{D}$ with mixed Boyd--Wong / Banach rates]\label{ex:bw-family}
The family aspect of Theorem~\ref{thm:boyd-wong} parallels the
one for Theorem~\ref{thm:contraction}: the comparison functions
$\varphi_\alpha$ may differ across $\alpha\in A$, and a single
$\varphi$ does not capture the structure. We exhibit an instance
where one coordinate satisfies a genuine Boyd--Wong inequality
with sub-geometric rate and the other a Banach inequality with
geometric rate.

Take $X=[0,\tfrac12]^2$ with the per-coordinate upper
quasi-pseudometrics
\[
  d_1\left((x_1,x_2),(y_1,y_2)\right):=(x_1-y_1)^+,
  \qquad
  d_2\left((x_1,x_2),(y_1,y_2)\right):=(x_2-y_2)^+,
\]
generating a $T_0$ bicomplete quasi-uniformity (the
symmetrisations together give the Euclidean uniformity on
$[0,\tfrac12]^2$, a compact, hence complete, space). Define
\[
  T(x_1,x_2):=\left(x_1(1-x_1),\,\frac{x_2}{2}\right);
\]
both components map $[0,\tfrac12]$ into itself, so $T$ is
well-defined.

Pair $d_1$ with the logistic comparison function
$\varphi_1(t):=\min\left\lbrace t-t^2,\tfrac14 \right \rbrace$ of
Example~\ref{ex:boyd-wong-logistic}, and $d_2$ with the linear
$\varphi_2(t):=\tfrac{t}{2}$. Then
\[
  d_1(Tx,Ty)\le\varphi_1\left(d_1(x,y)\right),
  \quad
  d_2(Tx,Ty)=\tfrac12\,d_2(x,y)=\varphi_2\left(d_2(x,y)\right),
\]
the first by the factorisation of
Example~\ref{ex:boyd-wong-logistic}, the second by a direct
computation; the conjugate inequalities follow by the same
arguments with swapped arguments. So $T$ is a
$\mathcal{D}$-Boyd--Wong contraction with comparison family
$\left\lbrace \varphi_1,\varphi_2 \right \rbrace$. The special form $\varphi_2(t)=k\,t$
recovers the Banach regime on the second coordinate, and
Theorem~\ref{thm:boyd-wong} accommodates it as the linear case
of Definition~\ref{def:boyd-wong}.

The unique fixed point is $x^*=(0,0)$, and the iterates from
$x_0=(\tfrac12,\tfrac12)$ have the form
$T^n(\tfrac12,\tfrac12)=\left(a_n,\,2^{-(n+1)}\right)$, where the
first-coordinate sequence $a_{n+1}=a_n(1-a_n)$ with
$a_0=\tfrac12$ decays sub-geometrically as $a_n\sim \tfrac{1}{n}$. The
two coordinates exhibit qualitatively different rates---geometric
$2^{-n}$ in the second, sub-geometric $\tfrac{1}{n}$ in the first---both
delivered by Theorem~\ref{thm:boyd-wong}, with convergence in
$\Us$ inheriting the slower of the two. A single comparison
function $\varphi$ governing both coordinates would have to
satisfy $\varphi\ge\varphi_1$, losing the geometric Banach rate
in $d_2$.
\end{example}

\subsection{Stability of the fixed point under perturbations}\label{subsec:stability}

The contraction theorem is well-behaved under small perturbations
of the data, in the following sense.

\begin{proposition}\label{prop:stability}
Let $(X,\U)$ be a $T_0$ bicomplete quasi-uniform space generated
by $\mathcal{D}=\left\lbrace d_\alpha\right\rbrace$, and let
$T,T'\colon X\to X$ be $\mathcal{D}$-contractions with common
contraction constants
$\left\lbrace k_\alpha\right\rbrace\subseteq[0,1)$ and unique
fixed points $x^*,y^*\in X$. Then for every $\alpha\in A$,
\[
  d_\alpha(x^*,y^*)\le\frac{1}{1-k_\alpha}\sup_{x\in X}d_\alpha(Tx,T'x),
\]
and the conjugate counterpart
$d_\alpha(y^*,x^*)\le(1-k_\alpha)^{-1}\sup_{x\in X}d_\alpha(T'x,Tx)$
holds. In particular, if $\varepsilon:=\sup_{x\in X}d_\alpha^s(Tx,T'x)<\infty$
then $d_\alpha^s(x^*,y^*)\le\varepsilon/(1-k_\alpha)$. When the
supremum on the right is infinite the bound is vacuous but
correct; non-trivial content of the proposition requires the
perturbation $T'$ of $T$ to be uniformly bounded in
$d_\alpha^s$.
\end{proposition}

\begin{proof}
Fix $\alpha$. Since $Tx^*=x^*$ and $T'y^*=y^*$,
\[
  d_\alpha(x^*,y^*)=d_\alpha(Tx^*,T'y^*).
\]
The triangle inequality applied at $Ty^*$ and the contraction
property of $T$ give
\begin{align*}
  d_\alpha(Tx^*,T'y^*)
  &\le d_\alpha(Tx^*,Ty^*)+d_\alpha(Ty^*,T'y^*)\\
  &\le k_\alpha\,d_\alpha(x^*,y^*)+\sup_{x\in X}d_\alpha(Tx,T'x).
\end{align*}
Writing $\varepsilon_\alpha:=\sup_{x\in X}d_\alpha(Tx,T'x)$ and
rearranging,
\[
  (1-k_\alpha)\,d_\alpha(x^*,y^*)\le\varepsilon_\alpha,
\]
which gives the stated bound for $d_\alpha(x^*,y^*)$. For the
conjugate direction, the same argument with the roles of $T$ and
$T'$ exchanged yields
\[
  d_\alpha(y^*,x^*)
  \le d_\alpha(T'y^*,T'x^*)+d_\alpha(T'x^*,Tx^*)
  \le k_\alpha\,d_\alpha(y^*,x^*)+\sup_{x\in X}d_\alpha(T'x,Tx),
\]
hence
\[
  d_\alpha(y^*,x^*)\le(1-k_\alpha)^{-1}\sup_{x\in X}d_\alpha(T'x,Tx).
\]
The forward and conjugate bounds use the quantities
\[
  \sup_{x\in X}d_\alpha(Tx,T'x)
  \quad\text{and}\quad
  \sup_{x\in X}d_\alpha(T'x,Tx)
\]
respectively; in the asymmetric setting these are in general
distinct. If $\varepsilon:=\sup_x d_\alpha^s(Tx,T'x)$ then both
are bounded by $\varepsilon$, and the symmetric bound
\[
  d_\alpha^s(x^*,y^*)=\max\left\lbrace d_\alpha(x^*,y^*),d_\alpha(y^*,x^*) \right \rbrace\le\varepsilon/(1-k_\alpha)
\]
follows.
\end{proof}

\begin{remark}\label{rem:stability-discussion}
Proposition~\ref{prop:stability} says that the fixed point of a
$\mathcal{D}$-contraction depends Lipschitz-continuously on the
map: a small uniform perturbation of $T$ moves the fixed point by
at most a factor $(1-k_\alpha)^{-1}$ in each $d_\alpha$. The bound
deteriorates as $k_\alpha\to 1$, in line with the metric
folklore that slowly contracting maps have ill-conditioned fixed
points. In the formal-ball setting of
Example~\ref{ex:formal-balls}, where $T(x,r)=(Sx,kr)$ for a Banach
contraction $S$ with constant $k$, the proposition specialises to
the familiar bound
$\rho(x^*,y^*)\le(1-k)^{-1}\sup_x\rho(Sx,S'x)$ between fixed
points of two Banach contractions $S,S'\colon M\to M$.
\end{remark}

\begin{example}\label{ex:asymmetric-norm}
Take $X=\R$ with the upper quasi-pseudometric
$d(x,y)=\max\left\lbrace x-y,0 \right \rbrace$, so that $d^{-1}(x,y)=\max\left\lbrace y-x,0 \right \rbrace$ and
$\Us$ is the standard Euclidean uniformity. Any monotone
nondecreasing Euclidean contraction $T\colon\R\to\R$ with constant
$k$ satisfies~\eqref{eq:contraction} for $\mathcal{D}=\left\lbrace d \right \rbrace$ with
$k_\alpha=k$: monotonicity makes $T(x)-T(y)$ have the same sign as
$x-y$, so the truncation $\max\left\lbrace \cdot,0 \right \rbrace$ commutes with the
$k$-Lipschitz bound. The space $(\R,\U_d)$ is $T_0$ and
bicomplete, so Theorem~\ref{thm:contraction} produces the unique
fixed point. Sign-reversing contractions, such as $T(x)=-\tfrac{x}{2}$,
are not $\left\lbrace d \right \rbrace$-contractions in this sense (although they are
contractions in the symmetrisation $\Us$, where the Banach
principle applies directly).
\end{example}

\begin{example}\label{ex:product}
For $i=1,\dots,n$, let $(X_i,d_i)$ be a $T_0$ bicomplete
quasi-pseudometric space and $T_i\colon X_i\to X_i$ a contraction
with constant $k_i\in[0,1)$. The product $X=\prod_i X_i$ carries the
generating family $\mathcal{D}=\left\lbrace D_i \right \rbrace_{i=1}^{n}$, where
$D_i(x,y):=d_i(\pi_i x,\pi_i y)$ and $\pi_i$ is the $i$-th
projection. The family $\mathcal{D}$ generates a $T_0$ bicomplete
quasi-uniformity, and the map $T=T_1\times\cdots\times T_n$
satisfies~\eqref{eq:contraction} with $k_{D_i}=k_i$. The constants
$k_i$ may differ, and Theorem~\ref{thm:contraction} delivers a
unique fixed point without coupling them through a single
quasi-pseudometric.
\end{example}

The next example is more substantial. It places the contraction
theorem inside the formal-ball construction of Edalat and
Heckmann~\cite{EdalatHeckmann}, where the radius coordinate plays
the role of an explicit error indicator.

\begin{example}[Formal balls of a complete metric space]\label{ex:formal-balls}
Let $(M,\rho)$ be a complete metric space. The space of
\emph{formal balls} of $M$ is
$\mathbb{B}(M):=M\times[0,\infty)$, equipped with the
quasi-pseudometric
\[
  d_{\mathbb{B}}\left((x,r),(y,r')\right)
  :=\max\left\lbrace \rho(x,y)-(r-r'),\,0 \right \rbrace.
\]
A formal ball $(x,r)$ is read as the closed ball
$\overline B(x,r):=\left\lbrace z\in M:\rho(x,z)\le r \right \rbrace\subseteq M$, and the
associated preorder
$(x,r)\le_{\mathbb{B}}(y,r')\iff\rho(x,y)\le r-r'$ coincides with
reverse inclusion of balls. The symmetrisation is the product
metric $d_{\mathbb{B}}^{s}((x,r),(y,r'))=\rho(x,y)+|r-r'|$, so
$(\mathbb{B}(M),\U_{\mathbb{B}})$ is $T_0$ and bicomplete because
$M$ is complete and $[0,\infty)$ is complete; see~\cite{EdalatHeckmann}.

Let $S\colon M\to M$ be a Banach contraction with constant
$k\in[0,1)$ and unique fixed point $x^*\in M$. Define
$T\colon\mathbb{B}(M)\to\mathbb{B}(M)$ by $T(x,r):=(Sx,kr)$. A
direct computation gives
\[
  d_{\mathbb{B}}\left(T(x,r),T(y,r')\right)
   =\max\left\lbrace \rho(Sx,Sy)-k(r-r'),0 \right \rbrace
  \le k\,d_{\mathbb{B}}\left((x,r),(y,r')\right),
\]
because $\rho(Sx,Sy)\le k\rho(x,y)$. The conjugate inequality
$d_{\mathbb{B}}(T(y,r'),T(x,r))\le k\,d_{\mathbb{B}}((y,r'),(x,r))$
follows from the same computation with the sign of $(r-r')$
reversed, using the symmetry of $\rho$ (not of $d_{\mathbb{B}}$,
which is genuinely asymmetric) to identify $\rho(Sy,Sx)=\rho(Sx,Sy)$.
Hence $T$ is
a $\left\lbrace d_{\mathbb{B}} \right \rbrace$-contraction with constant $k$, and
Theorem~\ref{thm:contraction} produces the unique fixed point
$(x^*,0)\in\mathbb{B}(M)$. The Picard iterates from any
$(x_0,r_0)$ are $T^n(x_0,r_0)=(S^n x_0,k^n r_0)$: the centre
follows the classical Picard sequence in $M$, and the radius
collapses geometrically. Corollary~\ref{cor:picard}, evaluated on
the second-coordinate-decreasing direction, gives
\[
  \rho(x^*,S^n x_0)+k^n r_0
  \le \frac{k^n}{1-k}\left(\rho(x_0,Sx_0)+(1-k)r_0\right),
\]
which simplifies to the classical Picard error bound
$\rho(x^*,S^n x_0)\le k^n(1-k)^{-1}\rho(x_0,Sx_0)$ on $M$. The
formal-ball machinery thus reproduces the Banach contraction
theorem and its error bound in a single application of
Theorem~\ref{thm:contraction}.

Figure~\ref{fig:formal-balls} pictures the situation. The preorder
$(x,r)\le_{\mathbb{B}}(y,r')$ corresponds to nesting of closed
balls in $M$, and the iterates $T^n(x_0,r_0)$ trace a sequence of
ever-tighter balls collapsing onto $x^*$.

\smallskip\noindent
\emph{Worked iterates.} Take $M=\R$ with $\rho(x,y)=|x-y|$ and
$S(x)=\tfrac{x}{3}$, so $k=\tfrac13$ and $x^*=0$. Starting from $(x_0,r_0)=(1,1)$,
\[
  T^n(1,1)=\left(3^{-n},\,3^{-n}\right),
  \qquad
  T^n(1,1)\to(0,0)\text{ in }\Us.
\]
The corresponding closed balls
$\overline B(3^{-n},3^{-n})=[\,0,\,2\cdot 3^{-n}\,]\subset\R$
form a nested decreasing family with intersection $\left\lbrace 0 \right \rbrace$, and
Corollary~\ref{cor:picard} bounds the symmetric error by
$|3^{-n}-0|+|3^{-n}-0|=2\cdot 3^{-n}$, which is exactly the
diameter of the $n$-th iterate's underlying ball.
\end{example}

\begin{figure}[htbp]
\centering
\begin{tikzpicture}[x=1cm, y=1cm, >={Stealth[length=2.5mm]}]
  \draw[thick] (0,0) circle (2.4);
  \draw[thick] (0.6,-0.2) circle (1.5);
  \draw[thick] (1.0,-0.35) circle (0.85);
  \draw[thick] (1.25,-0.45) circle (0.4);
  \filldraw (1.4,-0.5) circle (1.4pt) node[right, font=\scriptsize] {$x^*$};
  \filldraw (0,0) circle (1pt) node[left, font=\scriptsize] {$x_0$};
  \filldraw (0.6,-0.2) circle (1pt) node[above right=-1pt, font=\scriptsize] {$Sx_0$};
  \filldraw (1.0,-0.35) circle (1pt);
  \filldraw (1.25,-0.45) circle (1pt);
  \draw[->, thick, gray] (0,0) -- (-1.7,1.7) node[midway, above left, font=\scriptsize, black] {$r_0$};
  \draw[->, thick, gray] (0.6,-0.2) -- (-0.45,0.85) node[midway, above left, font=\scriptsize, black, xshift=2pt] {$kr_0$};
  \node[font=\small] at (4.6,0.6) {$(x_0,r_0)\le_{\mathbb{B}}(Sx_0,kr_0)\le_{\mathbb{B}}\cdots$};
  \node[font=\small] at (4.6,-0.4) {$\bar B(Sx_0,kr_0)\subseteq\bar B(x_0,r_0)$};
  \node[font=\small] at (4.6,-1.4) {$T^n(x_0,r_0)\to(x^*,0)$};
\end{tikzpicture}
\caption{The formal-ball picture of Example~\ref{ex:formal-balls}.
Each iterate $T^n(x_0,r_0)=(S^n x_0,k^n r_0)$ corresponds to a
closed ball in $M$, and the preorder
$(x,r)\le_{\mathbb{B}}(y,r')$ is reverse inclusion. The radii
collapse geometrically and the centres converge to the Banach
fixed point $x^*\in M$, so the iterates converge to the
``infinitely sharp'' formal ball $(x^*,0)$.}
\label{fig:formal-balls}
\end{figure}
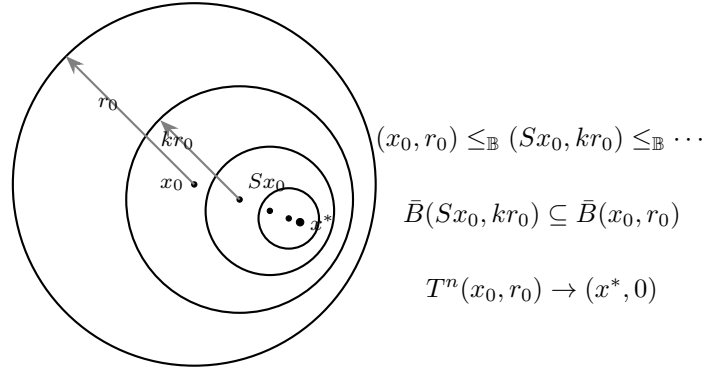

\begin{example}[A Volterra-type integral equation]\label{ex:volterra}
Let $X:=C([0,1],\R_{\ge 0})$ be the space of continuous nonnegative
functions on $[0,1]$, equipped with the asymmetric sup
quasi-pseudometric
$d(f,g):=\sup_{t\in[0,1]}(f(t)-g(t))^+$, where
$(\cdot)^+:=\max\left\lbrace \cdot,0 \right \rbrace$. The symmetrisation $d^s$ is the
standard sup metric, in which $X$ is closed in $C([0,1],\R)$, so
$(X,\U_d)$ is $T_0$ and bicomplete. Fix $\varphi\in X$, a
continuous kernel $K\colon[0,1]^2\to[0,\infty)$ bounded by
$K_{\max}$, and a function $\mu\colon[0,1]\to[0,\infty)$ continuous
with $\sup_s\mu(s)\le L$. Consider the integral operator
$T\colon X\to X$ defined by
\[
  (Tf)(t):=\varphi(t)+\int_0^t K(t,s)\,\mu(s)\,f(s)\,ds.
\]
The image $Tf$ is continuous and nonnegative, so $T$ is
well-defined on $X$. For $f,g\in X$,
\begin{align*}
  (Tf)(t)-(Tg)(t)
  &=\int_0^t K(t,s)\mu(s)\left(f(s)-g(s)\right)\,ds.
\end{align*}
Taking the positive part of the integrand and bounding
$K(t,s)\mu(s)\le K_{\max}L$ gives
\[
  \left((Tf)(t)-(Tg)(t)\right)^+
  \le K_{\max}L\int_0^t\left(f(s)-g(s)\right)^+\,ds
  \le K_{\max}L\,t\,d(f,g)\le K_{\max}L\,d(f,g),
\]
where the last step uses $t\in[0,1]$. Taking $\sup_t$ gives
$d(Tf,Tg)\le K_{\max}L\,d(f,g)$. The conjugate inequality
$d(Tg,Tf)\le K_{\max}L\,d(g,f)$ is obtained by the same
computation with the roles of $f$ and $g$ exchanged, so $T$ is a
$\left\lbrace d \right \rbrace$-contraction with constant $k:=K_{\max}L$ whenever
$K_{\max}L<1$. By Theorem~\ref{thm:contraction}, the Volterra
equation
\[
  f(t)=\varphi(t)+\int_0^t K(t,s)\,\mu(s)\,f(s)\,ds
\]
has a unique solution $f^*\in X$, obtained as the limit of Picard
iterates $T^n f_0$ from any starting function $f_0\in X$. The
error bound of Corollary~\ref{cor:picard} gives
$d(T^n f_0,f^*)\le k^n(1-k)^{-1}d(f_0,Tf_0)$, the asymmetric
analogue of the geometric Picard error bound for linear Volterra
equations.

\smallskip\noindent
\emph{A concrete instance.} Take $K(t,s)=1$, $\mu(s)=\tfrac12$ and
$\varphi(t)=1$. Then $T f(t)=1+\tfrac12\int_0^t f(s)\,ds$, the
Lipschitz constant is $k=\tfrac12$, and the unique fixed point in
$X$ satisfies $(f^*)'(t)=\tfrac12 f^*(t)$ with $f^*(0)=1$, hence
$f^*(t)=e^{t/2}$. Starting from $f_0\equiv 0$, $Tf_0\equiv 1$ and
the first few iterates $T^n f_0(t)$ are
$1$, $1+\tfrac{t}{2}$, $1+\tfrac{t}{2}+\tfrac{t^2}{8}$, the
truncations of the Taylor series of $e^{t/2}$. With $f_0\equiv 0$,
the forward asymmetric distance vanishes, $d(f_0,Tf_0)=\sup_t(0-1)^+=0$,
so $d(T^n f_0,f^*)=0$ for every $n$, since the iterates approach $f^*$
from below in the forward direction, which the upper
quasi-pseudometric does not see. The conjugate direction carries
the non-trivial bound: $d(Tf_0,f_0)=1$, and
$d(f^*,T^n f_0)\le 2^{-n}(1-\tfrac12)^{-1}d(Tf_0,f_0)=2^{-(n-1)}$,
the geometric rate of Picard iteration for this kernel.
\end{example}

\begin{remark}[Bielecki rescaling and the full Picard--Lindel\"of regime]\label{rem:bielecki}
The contraction range $K_{\max}L<1$ in
Example~\ref{ex:volterra} reflects the choice of the uniform
quasi-pseudometric $d(f,g)=\sup_{t}(f(t)-g(t))^+$ as the generating
family. A larger range is recovered by replacing $\mathcal{D}=\left\lbrace d \right \rbrace$
with the Bielecki-rescaled family
\[
  \mathcal{D}_\lambda := \left\lbrace d_\lambda \right \rbrace,\qquad
  d_\lambda(f,g):=\sup_{t\in[0,1]}e^{-\lambda t}\left(f(t)-g(t)\right)^+,
\]
for $\lambda>0$. A short computation gives
\[
  d_\lambda(Tf,Tg)\le K_{\max}L\,\sup_{t\in[0,1]}e^{-\lambda t}\!\int_0^t e^{\lambda s}\,ds\cdot d_\lambda(f,g)
  \le \frac{K_{\max}L}{\lambda}\,d_\lambda(f,g),
\]
so $T$ is a $\mathcal{D}_\lambda$-contraction with constant
$K_{\max}L/\lambda$ for every $\lambda>K_{\max}L$, irrespective of
the size of $K_{\max}L$. The symmetrisations $d_\lambda^{\,s}$ are
equivalent to the sup metric on $[0,1]$ for each $\lambda>0$, so
the underlying quasi-uniformity and the bicompleteness of
$(X,\U_{d_\lambda})$ are unchanged, and
Theorem~\ref{thm:contraction} delivers the unique fixed point of
the Volterra operator over the full classical Picard--Lindel\"of
regime. The flexibility to choose the generating family is
therefore a genuine feature of the quasi-uniform setting, and
parallels the classical Bielecki argument.
\end{remark}

\section{Endpoints via the dual quasi-uniformity}\label{sec:endpoint}

The contraction theorem of Section~\ref{sec:contraction} produces
a fixed point witnessing the startpoint relation
$(x^*,Tx^*)\in\bigcap\U$. To obtain the endpoint relation
$(Tx^*,x^*)\in\bigcap\U$, we apply the same theorem to the
conjugate quasi-uniformity $\Uinv$. Rather than repeat the
argument in a dual guise, we isolate a single structural fact
(Lemma~\ref{lem:conjugation}): every hypothesis appearing in
Sections~\ref{sec:trivial}--\ref{sec:contraction} is invariant
under the passage $\U\mapsto\Uinv$, so each result of
Section~\ref{sec:contraction} gives rise to an endpoint counterpart
by symmetry. The section is organised in three parts. In
Subsection~\ref{subsec:conjugation} we record the two lemmas
identifying $\U$-endpoints with $\Uinv$-startpoints and prove the
conjugation lemma. In
Subsection~\ref{subsec:endpoint-versions} we deduce endpoint
versions of the contraction principle, the Picard error bound, the
Boyd--Wong companion, and the stability estimate. In
Subsection~\ref{subsec:endpoint-examples} we read the running
examples of Section~\ref{sec:contraction} through the endpoint
lens and close with a remark on the collapse of the two notions
when $\U$ is symmetric.

\subsection{The conjugation principle}\label{subsec:conjugation}

\begin{lemma}\label{lem:dual-startpoint}
Let $(X,\U)$ be a quasi-uniform space and $T\colon X\to 2^X$. A point
$x_0\in X$ is a $\U$-endpoint of $T$ if and only if it is an
$\Uinv$-startpoint of $T$.
\end{lemma}

\begin{proof}
By definition, $x_0$ is a $\U$-endpoint of $T$ if there exists
$z\in Tx_0$ with $(z,x_0)\in\bigcap\U$, and $x_0$ is a
$\Uinv$-startpoint of $T$ if there exists $z\in Tx_0$ with
$(x_0,z)\in\bigcap\Uinv$. The conjugate quasi-uniformity is
$\Uinv=\left\lbrace U^{-1}:U\in\U\right\rbrace$, where
$U^{-1}=\left\lbrace (x,y):(y,x)\in U\right\rbrace$, so
$\bigcap\Uinv=\bigcap_{U\in\U}U^{-1}=\left(\bigcap_{U\in\U}U\right)^{-1}=(\bigcap\U)^{-1}$.
Therefore $(x_0,z)\in\bigcap\Uinv$ if and only if
$(z,x_0)\in\bigcap\U$, and the two existential statements coincide
on the same witness $z\in Tx_0$.
\end{proof}

Applying Proposition~\ref{prop:triv-multi} to $\Uinv$ gives:

\begin{proposition}\label{prop:triv-endpoint}
Let $(X,\U)$ be a quasi-uniform space, $T\colon X\to 2^X$ and
$x^*\in X$ with $Tx^*\neq\emptyset$. The following are equivalent:
\begin{enumerate}[label=\textup{(\roman*)}]
  \item $U^{-1}(x^*)\cap U(z)\neq\emptyset$ for every $U\in\U$ and every
        $z\in Tx^*$;
  \item $(z,x^*)\in\bigcap\U$ for every $z\in Tx^*$.
\end{enumerate}
\end{proposition}

We now record the workhorse of the section: every hypothesis
appearing in Section~\ref{sec:contraction} transfers verbatim
between $\U$ and its conjugate $\Uinv$.

\begin{lemma}[Conjugation invariance]\label{lem:conjugation}
Let $\mathcal{D}=\left\lbrace d_\alpha\right\rbrace_{\alpha\in A}$
be a family of quasi-pseudometrics on $X$ generating a
quasi-uniformity $\U$, and let
$\mathcal{D}^{-1}:=\left\lbrace d_\alpha^{-1}\right\rbrace_{\alpha\in A}$
denote the family of conjugates, where
$d_\alpha^{-1}(x,y):=d_\alpha(y,x)$. Then:
\begin{enumerate}[label=\textup{(\alph*)}]
  \item $\mathcal{D}^{-1}$ generates $\Uinv$.
  \item $(X,\U)$ is bicomplete iff $(X,\Uinv)$ is, and $\U$ is
        $T_0$ iff $\Uinv$ is; in both cases the symmetrisations
        coincide, $\Us=\Uinvs$.
  \item A self-map $T\colon X\to X$ is a
        $\mathcal{D}$-contraction with constants
        $\left\lbrace k_\alpha\right\rbrace$ iff $T$ is a
        $\mathcal{D}^{-1}$-contraction with the same constants
        $\left\lbrace k_\alpha\right\rbrace$.
  \item A self-map $T\colon X\to X$ is a $\mathcal{D}$-Boyd--Wong
        contraction with comparison family
        $\left\lbrace \varphi_\alpha\right\rbrace$ iff $T$ is a
        $\mathcal{D}^{-1}$-Boyd--Wong contraction with the same
        family $\left\lbrace \varphi_\alpha\right\rbrace$.
\end{enumerate}
\end{lemma}

\begin{proof}
(a) A basic entourage of $\U$ generated by $\mathcal{D}$ is a set
of the form
$B_{\alpha,\varepsilon}:=\left\lbrace (x,y):d_\alpha(x,y)<\varepsilon\right\rbrace$
for $\alpha\in A$ and $\varepsilon>0$. Its inverse
$B_{\alpha,\varepsilon}^{-1}=\left\lbrace (x,y):d_\alpha(y,x)<\varepsilon\right\rbrace
 =\left\lbrace (x,y):d_\alpha^{-1}(x,y)<\varepsilon\right\rbrace$
is precisely the corresponding basic entourage generated by
$d_\alpha^{-1}$. The family of inverses of a base for $\U$ is a
base for $\Uinv$, so $\mathcal{D}^{-1}$ generates $\Uinv$.

(b) The symmetrisation of a quasi-pseudometric is
$d^s=\max\left\lbrace d,d^{-1}\right\rbrace$; since
$(d_\alpha^{-1})^{-1}=d_\alpha$, one has
$(d_\alpha^{-1})^s=\max\left\lbrace d_\alpha^{-1},d_\alpha\right\rbrace=d_\alpha^s$.
Thus the generating family of $\Us$ coincides with the generating
family of $\Uinvs$, and $\Us=\Uinvs$. Bicompleteness of
$(X,\U)$ is defined as Cauchy-completeness of the pseudometric
uniformity $\Us$ and therefore depends only on $\Us$; the same
holds for $(X,\Uinv)$ via $\Uinvs=\Us$. The $T_0$ axiom for
$\U$ asserts that the preorder $\bigcap\U$ is antisymmetric,
i.e.\ that $\bigcap\U\cap(\bigcap\U)^{-1}=\diag$; the same
condition on $\Uinv$ reads
$\bigcap\Uinv\cap(\bigcap\Uinv)^{-1}=\diag$, and
$\bigcap\Uinv=(\bigcap\U)^{-1}$ (as noted in the proof of
Lemma~\ref{lem:dual-startpoint}). Hence the two $T_0$ conditions
are the same equation.

(c) Fix $\alpha\in A$ and $x,y\in X$. The two lines
of~\eqref{eq:contraction} for $T$ with respect to $d_\alpha$
read
\[
  d_\alpha(Tx,Ty)\le k_\alpha\,d_\alpha(x,y),
  \qquad
  d_\alpha(Ty,Tx)\le k_\alpha\,d_\alpha(y,x).
\]
Substituting $d_\alpha(u,v)=d_\alpha^{-1}(v,u)$ in each line and
exchanging the variable names transforms the pair into
\[
  d_\alpha^{-1}(Ty,Tx)\le k_\alpha\,d_\alpha^{-1}(y,x),
  \qquad
  d_\alpha^{-1}(Tx,Ty)\le k_\alpha\,d_\alpha^{-1}(x,y),
\]
which is~\eqref{eq:contraction} for $T$ with respect to
$d_\alpha^{-1}$. Since the transformation is an involution, the
two contraction conditions are equivalent, with the same
constants $\left\lbrace k_\alpha\right\rbrace$.

(d) The same substitution applied to
Definition~\ref{def:boyd-wong} maps the two lines
of~\eqref{eq:bw-contraction} to each other with $d_\alpha$
replaced by $d_\alpha^{-1}$, and the comparison functions
$\varphi_\alpha$ are unchanged.
\end{proof}

\begin{remark}\label{rem:conjugation-scope}
Lemma~\ref{lem:conjugation} does more than record a duality of
notation. Parts (a)--(b) show that the ambient hypotheses of
Theorem~\ref{thm:contraction} ($T_0$ and bicompleteness) are
properties of $\Us$, and parts (c)--(d) show that the
contraction and Boyd--Wong conditions are, by construction, closed
under the swap of the two lines of~\eqref{eq:contraction} (respectively
of~\eqref{eq:bw-contraction}); this closure is precisely what the
two-sided form of the definition was designed to guarantee
(cf.\ Remark~\ref{rem:sharpness}). Consequently, every theorem
and corollary of Section~\ref{sec:contraction} whose hypothesis
list is exhausted by these items admits an endpoint counterpart
obtained by a purely mechanical substitution
$(\U,\mathcal{D})\leftrightarrow(\Uinv,\mathcal{D}^{-1})$; we
harvest the four principal such counterparts in
Subsection~\ref{subsec:endpoint-versions}.
\end{remark}

\subsection{Endpoint versions of the contraction results}\label{subsec:endpoint-versions}

Combining Lemma~\ref{lem:dual-startpoint} with
Lemma~\ref{lem:conjugation}, each of the four principal results
of Section~\ref{sec:contraction} yields an endpoint corollary.
We state them in the same order in which they were proved:
Theorem~\ref{thm:contraction}, Corollary~\ref{cor:picard},
Theorem~\ref{thm:boyd-wong}, and Proposition~\ref{prop:stability}.

\begin{corollary}[Endpoint contraction theorem]\label{cor:contraction-endpoint-thm}
Let $(X,\U)$ be a $T_0$ bicomplete quasi-uniform space generated
by a family $\mathcal{D}=\left\lbrace d_\alpha\right\rbrace_{\alpha\in A}$
of quasi-pseudometrics, and let $T\colon X\to X$ be a
$\mathcal{D}$-contraction. Then $T$ has a unique fixed point
$x^*\in X$, and $x^*$ is a $\U$-endpoint of $T$ witnessed by
$z=x^*$. Moreover, for every $x_0\in X$ the sequence
$(T^n x_0)_{n\ge 0}$ converges to $x^*$ in $\Us$.
\end{corollary}

\begin{proof}
By Lemma~\ref{lem:conjugation}(a)--(b), $(X,\Uinv)$ is a $T_0$
bicomplete quasi-uniform space generated by $\mathcal{D}^{-1}$;
by Lemma~\ref{lem:conjugation}(c), $T$ is a
$\mathcal{D}^{-1}$-contraction with the same constants.
Theorem~\ref{thm:contraction} applied to
$(X,\Uinv,\mathcal{D}^{-1})$ therefore produces a unique fixed
point $x^{**}\in X$ of $T$. Uniqueness of the fixed point of $T$
(which is an equation in $X$, independent of $\U$) forces
$x^{**}=x^*$. The startpoint content of
Theorem~\ref{thm:contraction} applied to $\Uinv$ asserts
$(x^*,x^*)\in\bigcap\Uinv$, which is trivially true; the
witnessing content is that $x^*$ is a $\Uinv$-startpoint of $T$
witnessed by $z=x^*$, and Lemma~\ref{lem:dual-startpoint}
translates this to $x^*$ being a $\U$-endpoint of $T$ witnessed
by $z=x^*$. Convergence in $\Us$ is unchanged because
$\Us=\Uinvs$.
\end{proof}

\begin{corollary}[Endpoint Picard error bound]\label{cor:picard-endpoint}
Under the hypotheses of Theorem~\ref{thm:contraction}, for every
$x_0\in X$, every $\alpha\in A$ and every $n\ge 0$,
\[
  d_\alpha(x^*,T^n x_0)\le \frac{k_\alpha^n}{1-k_\alpha}\,d_\alpha(Tx_0,x_0),
  \qquad
  d_\alpha(T^n x_0,x^*)\le \frac{k_\alpha^n}{1-k_\alpha}\,d_\alpha(x_0,Tx_0).
\]
\end{corollary}

\begin{proof}
Apply Corollary~\ref{cor:picard} in the quasi-uniform space
$(X,\Uinv)$ with generating family $\mathcal{D}^{-1}$; the
hypotheses transfer by Lemma~\ref{lem:conjugation}(a)--(c). The
two inequalities of Corollary~\ref{cor:picard}, applied to
$d_\alpha^{-1}$ in place of $d_\alpha$, give
\[
  d_\alpha^{-1}(T^n x_0,x^*)\le\frac{k_\alpha^n}{1-k_\alpha}\,d_\alpha^{-1}(x_0,Tx_0),
  \qquad
  d_\alpha^{-1}(x^*,T^n x_0)\le\frac{k_\alpha^n}{1-k_\alpha}\,d_\alpha^{-1}(Tx_0,x_0),
\]
and the identity $d_\alpha^{-1}(u,v)=d_\alpha(v,u)$ rewrites
both inequalities in the stated form.
\end{proof}

\begin{remark}\label{rem:picard-endpoint}
Corollaries~\ref{cor:picard} and~\ref{cor:picard-endpoint}
together account for the four numbers
$d_\alpha(T^n x_0,x^*)$, $d_\alpha(x^*,T^n x_0)$,
$d_\alpha^{-1}(T^n x_0,x^*)$, $d_\alpha^{-1}(x^*,T^n x_0)$,
each bounded by $k_\alpha^n(1-k_\alpha)^{-1}$ times one of the
four seed distances
$d_\alpha(x_0,Tx_0)$, $d_\alpha(Tx_0,x_0)$,
$d_\alpha^{-1}(x_0,Tx_0)$, $d_\alpha^{-1}(Tx_0,x_0)$. Only two
of the four seed distances are independent, since
$d_\alpha^{-1}(x_0,Tx_0)=d_\alpha(Tx_0,x_0)$; likewise only two
of the four errors are independent. The two independent forms of
Corollary~\ref{cor:picard-endpoint} coincide with the two
independent forms of Corollary~\ref{cor:picard} up to a swap of
$d_\alpha$ and $d_\alpha^{-1}$, and the joint statement is what
records the per-direction Picard bookkeeping of the abstract in
its cleanest form.
\end{remark}

\begin{corollary}[Endpoint Boyd--Wong contraction principle]\label{cor:boyd-wong-endpoint}
Let $(X,\U)$ be a $T_0$ bicomplete quasi-uniform space generated
by a family $\mathcal{D}=\left\lbrace d_\alpha\right\rbrace_{\alpha\in A}$
of quasi-pseudometrics, and let $T\colon X\to X$ be a
$\mathcal{D}$-Boyd--Wong contraction. Then the unique fixed
point $x^*\in X$ produced by Theorem~\ref{thm:boyd-wong} is a
$\U$-endpoint of $T$, witnessed by $z=x^*$.
\end{corollary}

\begin{proof}
Lemma~\ref{lem:conjugation}(a)--(b) transfers the ambient
hypotheses to $(X,\Uinv,\mathcal{D}^{-1})$;
Lemma~\ref{lem:conjugation}(d) transfers the Boyd--Wong
condition. Theorem~\ref{thm:boyd-wong} applied to
$(X,\Uinv,\mathcal{D}^{-1})$ produces a $\Uinv$-startpoint of
$T$ witnessed by the same fixed point $x^*\in X$, and
Lemma~\ref{lem:dual-startpoint} translates this to a
$\U$-endpoint.
\end{proof}

\begin{corollary}[Endpoint stability]\label{cor:stability-endpoint}
Under the hypotheses of Proposition~\ref{prop:stability}, the
unique fixed points $x^*,y^*\in X$ of the two
$\mathcal{D}$-contractions $T,T'$ with common constants
$\left\lbrace k_\alpha\right\rbrace\subseteq[0,1)$ satisfy
\[
  d_\alpha(y^*,x^*)\le\frac{1}{1-k_\alpha}\sup_{x\in X}d_\alpha^{-1}(Tx,T'x)
  =\frac{1}{1-k_\alpha}\sup_{x\in X}d_\alpha(T'x,Tx)
\]
for every $\alpha\in A$, and symmetrically
\[
  d_\alpha(x^*,y^*)\le\frac{1}{1-k_\alpha}\sup_{x\in X}d_\alpha(Tx,T'x).
\]
The two bounds coincide with those of
Proposition~\ref{prop:stability}, exhibiting the stability
estimate as an intrinsically two-directional statement.
\end{corollary}

\begin{proof}
Apply Proposition~\ref{prop:stability} to
$(X,\Uinv,\mathcal{D}^{-1},T,T')$, using
Lemma~\ref{lem:conjugation}(a)--(c) to transfer the ambient and
contraction hypotheses. The bound
$d_\alpha^{-1}(x^*,y^*)\le(1-k_\alpha)^{-1}\sup_x d_\alpha^{-1}(Tx,T'x)$
so obtained rewrites as
$d_\alpha(y^*,x^*)\le(1-k_\alpha)^{-1}\sup_x d_\alpha(T'x,Tx)$
via $d_\alpha^{-1}(u,v)=d_\alpha(v,u)$, which is the first
displayed inequality; the second is the forward bound of
Proposition~\ref{prop:stability} itself.
\end{proof}

Corollaries~\ref{cor:contraction-endpoint-thm}--\ref{cor:stability-endpoint}
show that the fixed-point content of
Section~\ref{sec:contraction} is intrinsically symmetric under
$(\U,\mathcal{D})\leftrightarrow(\Uinv,\mathcal{D}^{-1})$: the
same fixed point $x^*$, the same per-direction geometric bounds,
the same Boyd--Wong extension and the same Lipschitz continuity
in the map are produced by either viewing. What differs between
the two viewings is only which of the two quasi-pseudometrics
$d_\alpha$ and $d_\alpha^{-1}$ carries the information; the
numerical asymmetry between them (evident in
Example~\ref{ex:picard-numerical}) does not translate into any
asymmetry of the abstract statements.

\begin{corollary}\label{cor:contraction-endpoint}
Under the hypotheses of Theorem~\ref{thm:contraction}, the unique
fixed point $x^*$ is simultaneously a $\U$-fixed point, a
$\U$-startpoint, and a $\U$-endpoint of $T$, witnessed by $z=x^*$.
\end{corollary}

\begin{proof}
The hypotheses of Theorem~\ref{thm:contraction} are satisfied:
$(X,\U)$ is $T_0$ bicomplete and generated by $\mathcal{D}$, and
$T$ is a $\mathcal{D}$-contraction. The theorem produces the
unique fixed point $x^*\in X$ with $Tx^*=x^*$.

\emph{$x^*$ is a $\U$-fixed point.} By
Definition~\ref{def:fixed-start-end}(a), this means
$(x^*,Tx^*)\in\bigcap\U$. Since $Tx^*=x^*$,
$(x^*,Tx^*)=(x^*,x^*)$. Axiom~(U1) of the quasi-uniformity
$\U$ asserts that every $U\in\U$ contains the diagonal
$\diag=\left\lbrace (y,y):y\in X\right\rbrace$, so
$(x^*,x^*)\in\diag\subseteq U$ for every $U\in\U$, and therefore
$(x^*,x^*)\in\bigcap\U$.

\emph{$x^*$ is a $\U$-startpoint and $\U$-endpoint.} Take
$z:=x^*\in Tx^*$ (using $Tx^*=x^*$, so $z\in Tx^*$ is
automatic). For the startpoint, we need
$(x^*,z)\in\bigcap\U$: this is the same point pair
$(x^*,x^*)\in\diag\subseteq\bigcap\U$ as above. For the
endpoint, we need $(z,x^*)\in\bigcap\U$: again the same
$(x^*,x^*)\in\bigcap\U$. By
Definition~\ref{def:fixed-start-end}(b), $x^*$ is both a
$\U$-startpoint and a $\U$-endpoint, witnessed by $z=x^*$.
\end{proof}

For a single-valued contraction, Corollary~\ref{cor:contraction-endpoint}
records what is essentially a terminological fact: a genuine
fixed point sits on the diagonal, and the diagonal lies in every
entourage, so it witnesses both startpoint and endpoint
conditions for free. The endpoint side acquires content only in
the multivalued setting, which the next section reaches through
the selector correspondence; see
Corollary~\ref{cor:multi-contraction}.

\subsection{Endpoint readings of the running examples}\label{subsec:endpoint-examples}

We now revisit two examples through the endpoint lens and
conclude with a remark on the collapse of the two notions when
$\U$ is symmetric.

\begin{example}[Endpoint reading of the Volterra example]\label{ex:volterra-endpoint}
Return to Example~\ref{ex:volterra} with the concrete instance
$K\equiv 1$, $\mu\equiv\tfrac12$, $\varphi\equiv 1$, so
$(Tf)(t)=1+\tfrac12\int_0^t f(s)\,ds$ on $X=C([0,1],\R_{\ge 0})$
with the upper quasi-pseudometric
$d(f,g)=\sup_t(f(t)-g(t))^+$. The unique fixed point is
$f^*(t)=e^{t/2}$. Read through
Corollary~\ref{cor:contraction-endpoint-thm}, the fixed point
$f^*$ is a $\U_d$-endpoint of $T$ witnessed by itself: this is
just the trivial diagonal witness of
Corollary~\ref{cor:contraction-endpoint}. The substantive
information sits in the endpoint Picard bound
(Corollary~\ref{cor:picard-endpoint}) applied to the seed
$f_0\equiv 0$. Since $Tf_0\equiv 1$,
\[
  d(Tf_0,f_0)=\sup_t(1-0)^+=1,\qquad
  d(f_0,Tf_0)=\sup_t(0-1)^+=0,
\]
so Corollary~\ref{cor:picard-endpoint} yields the two bounds
\[
  d(f^*,T^n f_0)\le\frac{\left(\tfrac{1}{2}\right)^n}{1-\tfrac{1}{2}}\cdot 1=2^{-(n-1)},
  \qquad
  d(T^n f_0,f^*)\le\frac{\left(\tfrac{1}{2}\right)^n}{1-\tfrac{1}{2}}\cdot 0=0.
\]
The first bound is the non-trivial one. It records that the
Picard iterates $T^n f_0$, which are the truncations of the
Taylor series of $e^{t/2}$, approach the fixed point $f^*$ from
below in the pointwise ordering of $[0,1]\to\R_{\ge 0}$; the
forward pseudodistance $d(T^n f_0,f^*)$ is identically zero at
every $n$, so the entire information about the rate of Picard
convergence lives in the conjugate direction. This is precisely
the endpoint direction. In the language of
Corollary~\ref{cor:contraction-endpoint-thm}, each $T^n f_0$ is
witnessed as an approximate $\U_d$-endpoint of $T$ (endpoint at
distance $2^{-(n-1)}$), whereas its status as an approximate
$\U_d$-startpoint is unconstrained by the theorem beyond the
trivial bound $0\le 0$. The endpoint side of the theorem is
therefore not a duplicate of its startpoint side: on any concrete
map whose iterates approach the fixed point one-sidedly, one of
the two sides carries all the numerical content and the other is
vacuous, with the assignment of ``carrier'' vs.\ ``vacuous''
determined by the direction of the approach relative to the
asymmetry of $d$.
\end{example}

\begin{example}[Multivalued endpoints on the unit interval]\label{ex:endpoint-multivalued}
Let $X=[0,1]$ with the upper quasi-pseudometric of
Example~\ref{ex:asymmetric-norm} and consider the multivalued map
$T\colon X\to 2^X$ given by $T(x):=[0,\tfrac{x}{2}]$. For any $x_0\in[0,1]$
and any $z\in T(x_0)\subseteq[0,\tfrac{x_0}{2}]\subseteq[0,x_0]$, the
$\le_\U$-relation $z\le_\U x_0$ holds because the preorder is the
standard order on $[0,1]$. Hence every $x_0\in[0,1]$ is a
$\U$-endpoint of $T$, witnessed by any $z\in T(x_0)$. The set of
endpoint witnesses of $x_0$ equals all of $T(x_0)=[0,\tfrac{x_0}{2}]$, so
endpoint witnesses come in continuum families even though $T$ has
the simple linear structure $T(x)=[0,\tfrac{x}{2}]$. Applying
Lemma~\ref{lem:dual-startpoint}, the same data viewed in $\Uinv$
yields a $\Uinv$-startpoint problem, and the witness sets of the
two perspectives are conjugate to each other.
\end{example}

\begin{remark}[Collapse under symmetry]\label{rem:symmetric-collapse}
When $\U$ is a uniformity, i.e.\ $\U=\Uinv$ (equivalently, every
$U\in\U$ contains its inverse $U^{-1}$), the preorder $\bigcap\U$
coincides with its converse $(\bigcap\U)^{-1}=\bigcap\Uinv$, so
the two membership relations
$(x_0,z)\in\bigcap\U$ and $(z,x_0)\in\bigcap\U$ are the same
condition on the same witness $z\in Tx_0$. Consequently every
$\U$-startpoint of $T$ is a $\U$-endpoint witnessed by the same
$z$, and \emph{vice versa}; the two notions coincide. The
distinction between startpoints and endpoints is therefore a
strictly asymmetric phenomenon of quasi-uniformities that are
not uniformities. When $\U$ is $T_0$ and a uniformity, the two
notions moreover collapse to the ordinary fixed-point notion
$z=x_0\in Tx_0$, because a $T_0$ uniformity satisfies
$\bigcap\U=\diag$ (the intersection of a Hausdorff uniformity is
the diagonal). The examples of this section and of
Section~\ref{sec:contraction} are chosen so that the asymmetry
of the generating family $\mathcal{D}$ is essential, so that the
endpoint theorems above carry information beyond the classical
Banach-principle content available on the symmetrisation
$(X,\Us)$ alone.
\end{remark}

\section{Selectors and the multivalued lift}\label{sec:selector}

The contraction theorem is stated for single-valued maps. The
standard way to extend such a theorem to multivalued maps is to
equip $2^X$ with a quasi-Hausdorff distance and prove a Nadler-type
analogue~\cite{Nadler}; this is the route taken
in~\cite{GabaKarapinarPetruselRadenovic2020} for weakly contractive
multi-valued maps on left $K$-complete $T_0$ quasi-metric
spaces, where the Hausdorff quasi-pseudometric on
$\mathscr{P}_{cb}(X)$ produces start-points directly. We take a
lighter route. A \emph{selector}
of a multivalued map $T$ is a single-valued map $f$ with $f(x)\in
Tx$, and the $\U$-fixed points of $f$ are automatically
$\U$-startpoints of $T$. Every single-valued contraction-type
theorem therefore yields a multivalued startpoint theorem applied
to a selector, and the next section uses this device to reduce a
common-startpoint problem for several multivalued maps to a
common-fixed-point problem for their selectors.

A \emph{selector} of $T\colon X\to 2^X$, with $Tx\neq\emptyset$ for
all $x$, is a map $f\colon X\to X$ with $f(x)\in Tx$ for every
$x$. When $Tx\neq\emptyset$ for every $x$, selectors exist in
abundance by the axiom of choice; the question is whether $T$
admits a selector with a useful regularity property (contractive,
continuous, monotone, measurable). The next two examples make
selectors concrete, and the example after that shows what
happens when the standing assumption $Tx\neq\emptyset$ fails.

\begin{example}[Three selectors of a triangular map]\label{ex:selectors-triangular}
On $X=[0,1]$ equipped with the upper quasi-pseudometric
$d(x,y)=\max\left\lbrace x-y,0 \right \rbrace$ of Example~\ref{ex:asymmetric-norm}, so
that $\bigcap\U$ is the standard order on $[0,1]$, define the
multivalued map $T(x):=[0,x]$, which sends each $x$ to all values
from $0$ up to and including $x$ itself. Three natural selectors
are
\[
  f_{\text{bot}}(x):=0,\qquad
  f_{\text{half}}(x):=\frac{x}{2},\qquad
  f_{\text{id}}(x):=x.
\]
Each picks one point out of $T(x)=[0,x]$ at every $x\in X$: the
bottom of the interval, its midpoint, and its top respectively.
Selectors of $T$ are therefore far from unique even when $T$ is as
simple as a one-parameter family of intervals; selecting a
particular one (the contractive one, $f_{\text{half}}$) is what
the contraction-theorem hypothesis of
Corollary~\ref{cor:multi-contraction} prescribes. The
selector $f_{\text{id}}(x)=x$ is non-contractive and witnesses
every $x\in[0,1]$ as a startpoint; $f_{\text{bot}}(x)=0$ is the
constant selector and produces $0$ as its sole fixed point;
$f_{\text{half}}(x)=\tfrac{x}{2}$ is the $\tfrac12$-contraction selected by
Theorem~\ref{thm:contraction}, with unique fixed point $0$.
\end{example}

\begin{example}[A selector picking the closer branch]\label{ex:selectors-branch}
On $X=\R$ with the upper quasi-pseudometric
$d(x,y)=\max\left\lbrace x-y,0 \right \rbrace$ of Example~\ref{ex:asymmetric-norm},
consider the two-branch map
\[
  T(x):=\left\lbrace \tfrac{x}{3},\,\tfrac{x}{3}+1\right \rbrace,
\]
which sends $x$ to two distinct points, one shifted by $+1$. The
``lower-branch'' selector $f_{\downarrow}(x):=\tfrac{x}{3}$ is a monotone
$\left\lbrace d \right \rbrace$-contraction with constant $\tfrac13$ and unique fixed
point $0$; Corollary~\ref{cor:multi-contraction} then yields $0$
as a $\U$-startpoint of $T$. The ``upper-branch'' selector
$f_{\uparrow}(x):=\tfrac{x}{3}+1$ is also a monotone
$\left\lbrace d \right \rbrace$-contraction with constant $\tfrac13$, with unique fixed
point $\tfrac{3}{2}$, identifying $\tfrac{3}{2}$ as a second $\U$-startpoint of $T$
via a different selector. Mixed selectors $f_{?}(x)$ that switch
branches as $x$ varies are still selectors of $T$, but need not
be contractive: $f_?(x):=\tfrac{x}{3}$ for $x\le 0$ and $f_?(x):=\tfrac{x}{3}+1$
for $x>0$ has a jump at $0$ and is not $\left\lbrace d \right \rbrace$-Lipschitz in any
neighbourhood of $0$. The example illustrates that picking
\emph{which} selector matters: $T$ has many selectors, only some
of which are contractive, and the contractive ones can pick out
distinct startpoints.
\end{example}

\begin{example}[No selector when a fibre is empty]\label{ex:no-selector}
The standing assumption $Tx\neq\emptyset$ for every $x\in X$ in
Proposition~\ref{prop:selector} is not cosmetic: as soon as a
single fibre is empty, no selector of $T$ can exist, because no
function $f\colon X\to X$ can satisfy $f(x_0)\in Tx_0=\emptyset$.
A concrete instance: on $X=[0,1]$ with the upper quasi-pseudometric
of Example~\ref{ex:asymmetric-norm}, define
\[
  T(x):=\begin{cases} [0,\tfrac{x}{2}] & x>0,\\ \emptyset & x=0.\end{cases}
\]
For every $x>0$ the fibre $T(x)$ is non-empty and the lower-half
map $x\mapsto \tfrac{x}{2}$ is a contractive selector on $(0,1]$; but at
$x_0=0$ the fibre is empty, so no $f\colon X\to X$ can satisfy
$f(0)\in T(0)$, and $T$ has no selector on the whole of $X$. The
startpoint question for $T$ is also vacuous at $x_0=0$: a
$\U$-startpoint at $0$ would require some
$z\in T(0)=\emptyset$ with $(0,z)\in\bigcap\U$, which is
impossible. Outside $\left\lbrace 0 \right \rbrace$, the contraction theorem applied to
the restriction $T\!\!\upharpoonright_{(0,1]}$ would seem to
produce a startpoint, but $(0,1]$ is not bicomplete (the Cauchy
sequence $\tfrac{1}{n}$ has no limit in $(0,1]$), so the conclusion
fails for the same reason as in Example~\ref{ex:non-bicomplete}.
The empty-fibre obstruction is therefore both a selector-theoretic
and a quasi-uniform obstruction, and the hypothesis
$Tx\neq\emptyset$ for all $x$ is the cleanest way to rule it out
at the source.
\end{example}

\begin{proposition}[folklore]\label{prop:selector}
Let $(X,\U)$ be a quasi-uniform space and $T\colon X\to 2^X$ with
$Tx\neq\emptyset$ for all $x$.
\begin{enumerate}[label=\textup{(\arabic*)}]
  \item If $f$ is a selector of $T$ and $x^*$ is a $\U$-fixed point
        of $f$, then $x^*$ is a $\U$-startpoint of $T$ witnessed by
        $f(x^*)$.
  \item Conversely, if $x^*$ is a $\U$-startpoint of $T$ witnessed by
        $z^*\in Tx^*$, then there is a selector $f$ of $T$ with
        $f(x^*)=z^*$, and $x^*$ is a $\U$-fixed point of $f$.
\end{enumerate}
In particular, $T$ has a $\U$-startpoint if and only if some selector
of $T$ has a $\U$-fixed point.
\end{proposition}

\begin{proof}
(1) Since $f$ is a selector of $T$, $f(x^*)\in Tx^*$ by definition.
Since $x^*$ is a $\U$-fixed point of $f$, $(x^*,f(x^*))\in\bigcap\U$.
Setting $z^*:=f(x^*)$ gives $z^*\in Tx^*$ and
$(x^*,z^*)\in\bigcap\U$, which is precisely the assertion that
$x^*$ is a $\U$-startpoint of $T$ witnessed by $z^*=f(x^*)$.

(2) The family $\left\lbrace Tx \right \rbrace_{x\in X}$ consists of non-empty sets by
hypothesis. The axiom of choice yields a function
$f_0\colon X\to X$ with $f_0(x)\in Tx$ for every $x\in X$; that
is, $f_0$ is a selector of $T$. Define $f\colon X\to X$ by
\[
  f(x):=\begin{cases} z^* & \text{if } x=x^*,\\
                       f_0(x) & \text{if } x\neq x^*. \end{cases}
\]
For $x\neq x^*$ one has $f(x)=f_0(x)\in Tx$; for $x=x^*$ one has
$f(x^*)=z^*\in Tx^*$ by hypothesis. Hence $f$ is a selector of
$T$. Moreover $(x^*,f(x^*))=(x^*,z^*)\in\bigcap\U$ because $z^*$
witnesses $x^*$ as a $\U$-startpoint of $T$, so $x^*$ is a
$\U$-fixed point of $f$.

The ``in particular'' clause follows by combining (1) and (2):
if some selector of $T$ has a $\U$-fixed point, (1) yields a
$\U$-startpoint of $T$; conversely, if $T$ has a $\U$-startpoint,
(2) produces a selector of $T$ with a $\U$-fixed point.
\end{proof}

\begin{remark}[On selector existence]\label{rem:selector-existence}
Part~(2) of Proposition~\ref{prop:selector} uses only the axiom
of choice and produces a \emph{set-theoretic} selector with no
regularity. The selector required by
Corollary~\ref{cor:multi-contraction} below is a
$\mathcal{D}$-contraction, a much stronger condition than mere
set-theoretic existence. We do not invoke any topological
selection theorem (Michael's theorem or its asymmetric
analogues) and instead verify the contraction property of a
specific selector in each application; see
Examples~\ref{ex:discrete-selector},
\ref{ex:continuous-selector},
and~\ref{ex:multi-formal-balls} for concrete instances. The
existence of a contractive selector is a real hypothesis on
$T$, not a consequence of the multivalued structure alone.
\end{remark}

\begin{corollary}\label{cor:multi-contraction}
Let $(X,\U)$ be a $T_0$ bicomplete quasi-uniform space generated by
$\mathcal{D}=\left\lbrace d_\alpha \right \rbrace$, and let $T\colon X\to 2^X$ have a
selector that is a $\mathcal{D}$-contraction. Then $T$ admits a
$\U$-startpoint (and a $\U$-endpoint).
\end{corollary}

\begin{proof}
By hypothesis there is a single-valued map $f\colon X\to X$ with
$f(x)\in Tx$ for every $x\in X$ (so $f$ is a selector of $T$, and
in particular $Tx\neq\emptyset$ for every $x$), and there is a
family $\left\lbrace k_\alpha:\alpha\in A \right \rbrace\subseteq[0,1)$ such that, for
every $\alpha\in A$ and every $x,y\in X$,
\[
  d_\alpha\left(f(x),f(y)\right)\le k_\alpha\,d_\alpha(x,y)
  \quad\text{and}\quad
  d_\alpha\left(f(y),f(x)\right)\le k_\alpha\,d_\alpha(y,x).
\]
That is, $f$ is a $\mathcal{D}$-contraction in the sense of
Definition~\ref{def:contraction}.

\emph{Existence of a fixed point of $f$.} The space $(X,\U)$ is
$T_0$, bicomplete, and generated by $\mathcal{D}$ by hypothesis,
so Theorem~\ref{thm:contraction} applies to $f$ and yields a
point $x^*\in X$ with $f(x^*)=x^*$.

\emph{$x^*$ is a $\U$-startpoint of $T$.} From $f(x^*)=x^*$ we
have $(x^*,f(x^*))=(x^*,x^*)$. Since every entourage $U\in\U$
contains the diagonal $\diag$ by axiom~(U1), the pair $(x^*,x^*)$
lies in every $U\in\U$, hence in $\bigcap\U$. Therefore
$(x^*,f(x^*))\in\bigcap\U$, which says that $x^*$ is a
$\U$-fixed point of $f$ in the sense of
Definition~\ref{def:fixed-start-end}(a). Setting $z^*:=f(x^*)$,
we have $z^*\in Tx^*$ (because $f$ is a selector of $T$) and
$(x^*,z^*)\in\bigcap\U$. By
Definition~\ref{def:fixed-start-end}(b), $x^*$ is a
$\U$-startpoint of $T$, witnessed by $z^*=f(x^*)=x^*$. This is
also the conclusion of Proposition~\ref{prop:selector}(1),
applied here for clarity.

\emph{$x^*$ is also a $\U$-endpoint of $T$.} The witness is
again $z^*=x^*\in Tx^*$. We must check $(z^*,x^*)\in\bigcap\U$.
But $(z^*,x^*)=(x^*,x^*)\in\diag\subseteq\bigcap\U$ by the same
argument as above. By Definition~\ref{def:fixed-start-end}(b),
$x^*$ is a $\U$-endpoint of $T$, witnessed by $z^*=x^*$. This is
the multivalued specialisation of
Corollary~\ref{cor:contraction-endpoint} via the selector
correspondence.
\end{proof}

Corollary~\ref{cor:multi-contraction} is the quasi-uniform analogue
of Nadler's multivalued contraction theorem, with the Hausdorff
distance machinery replaced by a single contractive selector. We
illustrate it in three settings of increasing complexity: a
discrete countable space (Example~\ref{ex:discrete-selector}), a
function space (Example~\ref{ex:continuous-selector}), and the
formal-ball construction (Example~\ref{ex:multi-formal-balls}).

\subsection{The startpoint set and its witnesses}\label{subsec:witness-set}

The definition of a $\U$-startpoint mixes two pieces of data: the
point $x_0$ and the witness $z\in Tx_0$ realising
$(x_0,z)\in\bigcap\U$. Either can vary independently of the other.
Two natural sets attach to a multivalued $T\colon X\to 2^X$ on a
quasi-uniform space $(X,\U)$:
\begin{align*}
  W(x) &:= \left\lbrace z\in Tx:(x,z)\in\bigcap\U\right \rbrace,\\
  \mathrm{Start}(T) &:= \left\lbrace x\in X:W(x)\neq\emptyset\right \rbrace.
\end{align*}
We call $W(x)$ the \emph{witness set} of $x$ and
$\mathrm{Start}(T)$ the \emph{$\U$-startpoint set} of $T$.

\begin{proposition}[Structure of witness sets]\label{prop:witness-structure}
Let $(X,\U)$ be a quasi-uniform space and $T\colon X\to 2^X$ with
$Tx\neq\emptyset$ for every $x$.
\begin{enumerate}[label=\textup{(\arabic*)}]
  \item For each $x\in X$, $W(x)=Tx\cap[x{,}\,\cdot)_\U$, where
        $[x{,}\,\cdot)_\U:=\left\lbrace z\in X:x\le_\U z \right \rbrace$ is the upset of
        $x$ in the preorder.
  \item $W(x)$ is closed under $\le_\U$-greater elements within
        $Tx$: if $z\in W(x)$, $z'\in Tx$, and $z\le_\U z'$, then
        $z'\in W(x)$.
  \item $x\in\mathrm{Start}(T)$ if and only if $Tx$ meets the upset
        of $x$, and the witnesses of $x$ are precisely the
        elements of $Tx$ in this upset.
\end{enumerate}
\end{proposition}

\begin{proof}
\emph{(1).} By definition of $W$, $z\in W(x)$ if and only if
$z\in Tx$ and $(x,z)\in\bigcap\U$. The relation
$(x,z)\in\bigcap\U$ is equivalent to $x\le_\U z$, which is
$z\in[x{,}\,\cdot)_\U$ by the notation introduced in the
statement. Therefore $z\in W(x)$ iff $z\in Tx$ and
$z\in[x{,}\,\cdot)_\U$, i.e.\ $z\in Tx\cap[x{,}\,\cdot)_\U$.

\emph{(2).} Suppose $z\in W(x)$, $z'\in Tx$, and $z\le_\U z'$.
From $z\in W(x)$, part~(1) gives $x\le_\U z$. Combined with
$z\le_\U z'$, transitivity of the preorder $\le_\U$ on $X$
yields $x\le_\U z'$. Hence $z'\in Tx\cap[x{,}\,\cdot)_\U$,
which by part~(1) is $W(x)$.

\emph{(3).} By definition $x\in\mathrm{Start}(T)$ means
$W(x)\neq\emptyset$, which by part~(1) is the same as
$Tx\cap[x{,}\,\cdot)_\U\neq\emptyset$, i.e.\ $Tx$ meets the
upset of $x$. The witnesses of $x$ are by definition the
elements of $W(x)$, and part~(1) identifies these with the
elements of $Tx$ in $[x{,}\,\cdot)_\U$.
\end{proof}

\begin{remark}[How many witnesses?]\label{rem:witness-count}
The size of $W(x_0)$ depends on $T$ and on the preorder. If $T$
is single-valued, $T(x)=\left\lbrace f(x) \right \rbrace$, then $|W(x_0)|\le 1$, and
equality holds exactly when $x_0$ is a $\U$-fixed point of $f$.
For the half-step map of Example~\ref{ex:start-vs-end} every
startpoint has a single witness, including $x=\tfrac12$, whose
witness is $1$. The multivalued formal-ball map of
Example~\ref{ex:multi-formal-balls} produces a one-parameter
family of witnesses at every startpoint: the witness condition
$(Sx_0,t)\in T(x_0,r_0)$ requires $0\le t\le kr_0$, while the
preorder $(x_0,r_0)\le_\mathbb{B}(Sx_0,t)$ requires
$t\le r_0-\rho(x_0,Sx_0)$, so
$W(x_0,r_0)=\left\lbrace (Sx_0,t):0\le t\le\min(kr_0,\,r_0-\rho(x_0,Sx_0)) \right \rbrace$.
When $\rho(x_0,Sx_0)\le(1-k)r_0$ the second constraint is the
weaker one and the witness range simplifies to $[0,kr_0]$. Witness multiplicity reflects the slack between
$T$ and the preorder, not just the multivaluedness of $T$.
\end{remark}

The startpoint set itself has a closure property under standard
topological hypotheses on $T$.

\begin{proposition}[Closure of the startpoint set]\label{prop:start-closed}
Let $(X,\U)$ be a quasi-uniform space whose preorder graph
$\bigcap\U$ is closed in $X\times X$ for the symmetric topology of
$\Us$. Let $T\colon X\to 2^X$ have $\Us$-closed graph
$\mathrm{Gr}(T):=\left\lbrace (y,z):z\in Ty \right \rbrace\subseteq X\times X$. Suppose
that for every $\Us$-convergent sequence $x_n\to x$ in
$\mathrm{Start}(T)$, there is a choice of witnesses $z_n\in W(x_n)$
admitting a $\Us$-convergent subsequence in $X$. Then
$\mathrm{Start}(T)$ is sequentially closed in $\Us$.
\end{proposition}

\begin{proof}
Let $x_n\in\mathrm{Start}(T)$ with $x_n\to x$ in $\Us$. By
hypothesis, choose witnesses $z_n\in W(x_n)\subseteq Tx_n$ and pass
to a subsequence with $z_n\to z$ in $\Us$. Closedness of
$\mathrm{Gr}(T)$ gives $z\in Tx$, and closedness of $\bigcap\U$
gives $(x,z)\in\bigcap\U$. Hence $z\in W(x)$ and
$x\in\mathrm{Start}(T)$.
\end{proof}

\begin{remark}\label{rem:start-closed-hypotheses}
The hypothesis that $\bigcap\U$ is closed in $X\times X$ holds, for
instance, when the closed entourages of $\U$ form a base for $\U$
(which they do for every quasi-uniformity~\cite[Lemma~1.18]{Kunzi}).
The witness-convergence hypothesis is automatic when $X$ is
$\Us$-compact or when $T$ takes values in a fixed $\Us$-compact
set. In Examples~\ref{ex:discrete-selector},
\ref{ex:continuous-selector} and~\ref{ex:multi-formal-balls},
$\mathrm{Start}(T)$ is closed and contains the fixed point of the
relevant selector.
\end{remark}

\subsection{A Knaster--Tarski companion under monotonicity}\label{subsec:knaster-tarski}

The selector route reaches startpoints via Cauchy iteration. A
complementary route is available when $T$ has a $\le_\U$-monotone
selector and $(X,\le_\U)$ has the right order-theoretic structure:
in that case $\mathrm{Start}(T)$ behaves like a Knaster--Tarski
fixed-point set.

\begin{proposition}\label{prop:knaster-tarski}
Let $(X,\U)$ be a quasi-uniform space such that $(X,\le_\U)$ is a
complete lattice. Suppose $T\colon X\to 2^X$ has a selector
$f\colon X\to X$ that is $\le_\U$-monotone, in the sense that
$x\le_\U y$ implies $f(x)\le_\U f(y)$. Then the set
\[
  \mathrm{Pre}(f):=\left\lbrace x\in X:x\le_\U f(x) \right \rbrace
\]
of pre-fixed-points of $f$ is contained in $\mathrm{Start}(T)$, has
a $\le_\U$-greatest element $\bar x$ given by the Knaster--Tarski
formula~\cite{Tarski}
$\bar x=\bigvee\left\lbrace x\in X:x\le_\U f(x) \right \rbrace$, and $\bar x$ satisfies
$\bar x\equiv_\U f(\bar x)$, so it is a $\U$-startpoint of $T$
witnessed by $f(\bar x)$. When $\U$ is $T_0$, $\bar x$ is moreover
a fixed point of $f$ in the ordinary sense, $\bar x=f(\bar x)$.
\end{proposition}

\begin{proof}
By definition, $x\le_\U f(x)$ says $(x,f(x))\in\bigcap\U$ and
$f(x)\in Tx$, so $f(x)\in W(x)$ and $x\in\mathrm{Start}(T)$. This
gives $\mathrm{Pre}(f)\subseteq\mathrm{Start}(T)$. The set
$\mathrm{Pre}(f)$ is non-empty (it contains the bottom element
$\bot:=\bigwedge X$ of the complete lattice $(X,\le_\U)$, since
$\bot\le_\U f(\bot)$ holds trivially) and closed under existing
suprema in $(X,\le_\U)$
because $f$ is $\le_\U$-monotone: if $\left\lbrace x_i\right\rbrace\subseteq\mathrm{Pre}(f)$
and $x:=\bigvee_i x_i$, then $x_i\le_\U x$ implies
$x_i\le_\U f(x_i)\le_\U f(x)$ for all $i$, so
$x=\bigvee_i x_i\le_\U f(x)$ and $x\in\mathrm{Pre}(f)$. The
supremum $\bar x:=\bigvee\mathrm{Pre}(f)$ therefore exists in
$X$ (by completeness of the lattice) and lies in
$\mathrm{Pre}(f)$ (by the closure under suprema just
established), so $\bar x$ is the $\le_\U$-greatest element of
$\mathrm{Pre}(f)$ and satisfies $\bar x\le_\U f(\bar x)$. By
monotonicity, $f(\bar x)\le_\U f(f(\bar x))$, so
$f(\bar x)\in\mathrm{Pre}(f)$ as well; since $\bar x$ is the
greatest element of $\mathrm{Pre}(f)$,
$f(\bar x)\le_\U\bar x$. Hence $\bar x\equiv_\U f(\bar x)$, and
if $\U$ is $T_0$, antisymmetry gives $\bar x=f(\bar x)$.
\end{proof}

\begin{remark}\label{rem:knaster-tarski-vs-banach}
Proposition~\ref{prop:knaster-tarski} requires no completeness on
$\U$ itself; it trades the bicompleteness of
Theorem~\ref{thm:contraction} for the order-completeness of
$(X,\le_\U)$. The two results cover different regimes:
Theorem~\ref{thm:contraction} works on contractive selectors
without any order-completeness, and
Proposition~\ref{prop:knaster-tarski} works on monotone selectors
without any contraction. In the formal-ball setting of
Example~\ref{ex:formal-balls}, $(\mathbb{B}(M),\le_{\mathbb{B}})$
is not a complete lattice in general, so the proposition does not
apply there; in Example~\ref{ex:start-vs-end}, $(X,\le_\U)=([0,1],\le)$
is a complete lattice but the half-step map has no monotone
selector. The Knaster--Tarski companion finds its natural ground
in domain-theoretic settings where $(X,\le_\U)$ is a continuous
lattice.
\end{remark}

\begin{example}[The Knaster--Tarski companion in action]\label{ex:knaster-tarski-example}
On $X=[0,1]$ with the upper quasi-pseudometric of
Example~\ref{ex:asymmetric-norm}, the preorder $\le_\U$ is the
standard order, and $([0,1],\le)$ is a complete lattice. Consider
the genuinely multivalued map
\[
  T(x):=\left[\min\left\lbrace x+\tfrac14,1 \right \rbrace,\;
                 \min\left\lbrace x+\tfrac12,1 \right \rbrace\right],
\]
which sends each $x\in[0,1]$ to the non-degenerate closed
subinterval of $[0,1]$ (a singleton only at $x=1$, where
$T(1)=\left\lbrace 1 \right \rbrace$). The lower-endpoint rule
$f(x):=\min\left\lbrace x+\tfrac14,1 \right \rbrace$ is a selector of $T$: for
each $x$, $f(x)\in T(x)$ by construction. The map $f$ is monotone
nondecreasing but not a strict $\left\lbrace d \right \rbrace$-contraction, since for
$x,y\in[0,\tfrac34]$ with $x>y$ one has $d(f(x),f(y))=(x-y)=d(x,y)$.
Theorem~\ref{thm:contraction} therefore does not apply. The set of
pre-fixed-points $\mathrm{Pre}(f)$ equals all of $[0,1]$: for
$x\in[0,\tfrac34]$, $f(x)=x+\tfrac14>x$ gives $x\le_\U f(x)$, and
for $x\in[\tfrac34,1]$, $f(x)=1\ge x$ gives the same. The
$\le_\U$-greatest element of $\mathrm{Pre}(f)$ is $\bar x=1$, and
indeed $f(1)=1\in T(1)$; Proposition~\ref{prop:knaster-tarski}
identifies $\bar x=1$ as a $\U$-startpoint of $T$ witnessed by
$1\in T(1)$. The Knaster--Tarski route reaches the fixed point
without any Cauchy estimate; order-completeness replaces
$\U$-completeness when monotonicity is available.

The multivalued content is genuine: for $x\in[0,\tfrac12)$ the
image $T(x)=[x+\tfrac14,x+\tfrac12]$ is a non-degenerate interval,
and the upper-endpoint rule
$g(x):=\min\left\lbrace x+\tfrac12,1 \right \rbrace$ is a second monotone
selector of $T$, with $g(1)=1$ and $\mathrm{Pre}(g)=[0,1]$ as
well, so Proposition~\ref{prop:knaster-tarski} also produces
$1$ as the greatest pre-fixed-point of $g$. Both selectors
witness the same $\U$-startpoint $\bar x=1$, but through
distinct elements of $T(x)$ along the way; the singleton
$T(x)$-content is only reached at the fixed point itself.
\end{example}

\begin{remark}[Tie-back to the contraction theorem]\label{rem:start-vs-contraction}
Whenever Theorem~\ref{thm:contraction} applies to a
$\mathcal{D}$-contractive selector $f$ of $T$, the fixed point
$x^*$ of $f$ lies in $\mathrm{Start}(T)$ with $f(x^*)\in W(x^*)$,
so the structure of the present subsection is inhabited in every
contractive setting.
\end{remark}

\begin{example}[A discrete selector on $\left\lbrace 0 \right \rbrace\cup\left\lbrace 2^{-n} \right \rbrace$]\label{ex:discrete-selector}
Let
\[
  X:=\left\lbrace 0 \right \rbrace\cup\left\lbrace 2^{-n}:n\in\N_0\right \rbrace\subseteq[0,1],
\]
endowed with the restriction of the upper quasi-pseudometric
$d(x,y):=\max\left\lbrace x-y,0 \right \rbrace$ on $\R$ (cf.\ Example~\ref{ex:asymmetric-norm}).
The set $X$ is countable, and the symmetrisation $d^s$ is the
standard Euclidean metric, in which $X$ is a closed subset of
$[0,1]$. Hence $(X,\U_d)$ is $T_0$ and bicomplete.

Define a multivalued map $T\colon X\to 2^X$ by
\[
  T(x):=\left\lbrace y\in X:y\le \tfrac{x}{2}\right \rbrace.
\]
For each $x\in X$, $\tfrac{x}{2}$ also lies in $X$ (since
$\tfrac{2^{-n}}{2}=2^{-(n+1)}$ and $\tfrac{0}{2}=0$), and the assignment
$f(x):=\tfrac{x}{2}$ is a selector of $T$. A direct computation gives, for
$x,y\in X$ with $x\ge y$,
\[
  d\left(f(x),f(y)\right)=\tfrac{x-y}{2}=\tfrac12\,d(x,y),
\]
and $d(f(y),f(x))=0=\tfrac12\,d(y,x)$, so $f$ is a monotone
$\left\lbrace d \right \rbrace$-contraction with constant $\tfrac12$.
Theorem~\ref{thm:contraction} produces the unique fixed point
$f^*=0$, and Corollary~\ref{cor:multi-contraction} identifies
$0\in X$ as a $\U$-startpoint of $T$, witnessed by
$f(0)=0\in T(0)=\left\lbrace 0 \right \rbrace$. The Picard sequence from $1\in X$ is
$1,\tfrac12,\tfrac14,\tfrac18,\dots$, converging geometrically to
$0$ in $\Us$. Other selectors of $T$ exist, for instance
$g(x):=0$, which is also a $\left\lbrace d \right \rbrace$-contraction (with constant $0$)
and produces the same startpoint.
\end{example}

\begin{example}[A continuous selector on a function space]\label{ex:continuous-selector}
Let $X:=C([0,1],\R_{\ge0})$ be the space of continuous nonnegative
functions on $[0,1]$, and let
\[
  d(f,g):=\sup_{t\in[0,1]}\left(f(t)-g(t)\right)^+
\]
be the asymmetric sup quasi-pseudometric, where
$(\cdot)^+:=\max\left\lbrace \cdot,0 \right \rbrace$. The symmetrisation
$d^s(f,g)=\sup_t|f(t)-g(t)|$ is the standard sup metric, in which
$X$ is closed in $C([0,1],\R)$. Hence $(X,\U_d)$ is $T_0$ and
bicomplete.

Fix $\alpha\in[0,1)$ and define a multivalued
$T\colon X\to 2^X$ by
\[
  T(f):=\left\lbrace g\in X:0\le g(t)\le\alpha\,f(t)\;\text{for all }t\in[0,1]\right \rbrace.
\]
The map $\sigma\colon X\to X$, $\sigma(f):=\alpha f$, is a
selector of $T$: $\alpha f$ is continuous and nonnegative, and
$\alpha f(t)\le\alpha f(t)$ trivially. A direct computation gives
\[
  d\left(\sigma(f),\sigma(g)\right)
  =\sup_t\left(\alpha f(t)-\alpha g(t)\right)^+
  =\alpha\,d(f,g),
\]
and the conjugate inequality is identical. So $\sigma$ is a
$\left\lbrace d \right \rbrace$-contraction with constant $\alpha$.
Theorem~\ref{thm:contraction} gives the unique fixed point
$f^*\equiv 0\in X$, and Corollary~\ref{cor:multi-contraction}
identifies $f^*\equiv 0$ as a $\U$-startpoint of $T$, witnessed by
$\sigma(0)=0\in T(0)=\left\lbrace 0 \right \rbrace$. As in the discrete example, the
selector is not unique: $\sigma_\beta(f):=\beta f$ for any
$\beta\in[0,\alpha]$ is also a contractive selector of $T$ with
the same fixed point.
\end{example}

\begin{example}[A multivalued formal-ball map]\label{ex:multi-formal-balls}
Let $(M,\rho)$ and $S\colon M\to M$ be as in
Example~\ref{ex:formal-balls}, with Banach fixed point $x^*\in M$
and contraction constant $k\in[0,1)$. Define a multivalued map
$T\colon\mathbb{B}(M)\to 2^{\mathbb{B}(M)}$ by
\[
  T(x,r):=\left\lbrace(Sx,t):0\le t\le k r\right \rbrace.
\]
The set $T(x,r)$ is the segment of formal balls centred at $Sx$
with radius at most $k r$, and is genuinely multivalued whenever
$r>0$. The single-valued map $f(x,r):=(Sx,k r)$ is a selector of
$T$ (it picks the maximal-radius element) and is the
$\left\lbrace d_{\mathbb{B}} \right \rbrace$-contraction of Example~\ref{ex:formal-balls}.
Corollary~\ref{cor:multi-contraction} therefore yields the
$\U_{\mathbb{B}}$-startpoint $(x^*,0)$, witnessed by
$f(x^*,0)=(x^*,0)\in T(x^*,0)$. The selector route avoids equipping
$2^{\mathbb{B}(M)}$ with a quasi-Hausdorff distance: a Nadler-type
or Mizoguchi--Takahashi~\cite{MizoguchiTakahashi} formulation would
need a two-sided quasi-Hausdorff functional on subsets of
$\mathbb{B}(M)$, with the forward and backward Hausdorff sides
interacting non-trivially with the formal-ball preorder.
\end{example}

\begin{example}[A multivalued Volterra inclusion]\label{ex:multi-volterra}
Let $X:=C([0,1],\R_{\ge 0})$ with the asymmetric sup
quasi-pseudometric
$d(f,g):=\sup_{t\in[0,1]}(f(t)-g(t))^+$ of
Example~\ref{ex:volterra}, so $(X,\U_d)$ is $T_0$ and
bicomplete. Fix $\varphi\in X$, a continuous kernel
$K\colon[0,1]^2\to[0,\infty)$ bounded by $K_{\max}$, and a continuous
$\mu\colon[0,1]\to[0,\infty)$ with $\sup_s\mu(s)\le L$ and
$K_{\max}L<1$. Consider a
\emph{set-valued} multiplier
$\Phi\colon[0,1]\times[0,\infty)\to 2^{[0,\infty)}\setminus\left\lbrace\emptyset\right\rbrace$
satisfying, for every $(s,u)\in[0,1]\times[0,\infty)$:
\begin{enumerate}[label=\textup{($\Phi_\arabic*$)}]
  \item $\Phi(s,u)$ is a non-empty closed subset of $[0,\mu(s)\,u]$;
  \item $\Phi$ is measurable in the sense of a Carath\'eodory
        multifunction (i.e.\ $s\mapsto\Phi(s,u)$ is
        Borel-measurable for each fixed $u$, and $u\mapsto\Phi(s,u)$
        is continuous in the Hausdorff sense for each fixed $s$);
  \item the pointwise infimum $u\mapsto\inf\Phi(s,u)$ is monotone
        nondecreasing and Lipschitz on $[0,\infty)$ with constant
        $\mu(s)$, uniformly in $s$.
\end{enumerate}
Define the multivalued Volterra operator $T\colon X\to 2^X$ by
\[
  T(f):=\left\lbrace g\in X\;:\;g(t)=\varphi(t)+\int_0^t K(t,s)\,\phi_s(f(s))\,ds,\
  \phi_s(f(s))\in\Phi(s,f(s))\text{ measurable in }s\right\rbrace,
\]
the set of continuous solutions obtained from arbitrary
measurable selections $\phi_s$ of $\Phi$. Under ($\Phi_2$), the
Kuratowski--Ryll-Nardzewski selection theorem yields a measurable
selection $\phi^\star\colon[0,1]\times[0,\infty)\to[0,\infty)$ of
$\Phi$; taking $\phi^\star(s,u):=\inf\Phi(s,u)$ works when the
infimum is measurable in $s$ (which holds under ($\Phi_2$), by
closedness of the values). Define the single-valued operator
$F_\star\colon X\to X$ by
\[
  F_\star(f)(t):=\varphi(t)+\int_0^t K(t,s)\,\phi^\star(s,f(s))\,ds.
\]
By ($\Phi_3$), for $f,g\in X$,
\[
  |\phi^\star(s,f(s))-\phi^\star(s,g(s))|\le\mu(s)\,|f(s)-g(s)|,
\]
and the estimate of Example~\ref{ex:volterra} gives
$d(F_\star f,F_\star g)\le K_{\max}L\,d(f,g)$ together with the
conjugate bound. Hence $F_\star$ is a $\{d\}$-contraction with
constant $K_{\max}L<1$; it is a selector of $T$ because
$\phi^\star(s,f(s))\in\Phi(s,f(s))$ and $F_\star(f)\in X$ by
continuity of the integral. By
Corollary~\ref{cor:multi-contraction}, the inclusion $f\in Tf$
has a $\U_d$-startpoint, witnessed by the fixed point $f^*$ of
$F_\star$. The example shows that the selector route absorbs
multivalued integral inclusions of Volterra type whenever the
multivalued nonlinearity admits a measurable Lipschitz selection,
without requiring any quasi-Hausdorff functional on
subsets of $X$.
\end{example}

\subsection{A direct quasi-Hausdorff route}\label{subsec:hausdorff-direct}

The selector correspondence of Proposition~\ref{prop:selector}
goes through a single-valued auxiliary $f\colon X\to X$. An
alternative is to keep the multivalued $T$ in place and read the
contractive condition directly on it, through a family of
asymmetric Hausdorff functionals attached to
$\mathcal{D}=\left\lbrace d_\alpha \right \rbrace_{\alpha\in A}$. The
single-quasi-pseudometric, left $K$-complete instance of this
route is the start-point theorem of
Gaba--Karap{\i}nar--Petru\c{s}el--Radenovi\'c~\cite[Theorem~1]{GabaKarapinarPetruselRadenovic2020};
Theorem~\ref{thm:hausdorff-direct} below lifts that result to a
generating family on a bicomplete quasi-uniform space, with the
contractive condition made two-sided so that the iteration
produces a bi-Cauchy, rather than only left $K$-Cauchy,
sequence and lands inside the bicompleteness framework of the
rest of the paper.

We collect the three notions involved before stating the result.

\begin{definition}[$d_\alpha$-Hausdorff functional, family version]\label{def:hausdorff-family}
For each $\alpha\in A$ and each pair of non-empty subsets
$A,B\subseteq X$, the \emph{$d_\alpha$-Hausdorff functional} is
\[
  H_\alpha(A,B)
  :=\max\left\lbrace \sup_{a\in A}\inf_{b\in B}d_\alpha(a,b),\;
                \sup_{b\in B}\inf_{a\in A}d_\alpha(a,b)\right \rbrace
  \;\in\;[0,\infty].
\]
The single-quasi-pseudometric case
$\mathcal{D}=\left\lbrace q \right \rbrace$ recovers the Hausdorff quasi-pseudometric
$H$ of~\cite[\S 4]{GabaKarapinarPetruselRadenovic2020}. By
construction, $H_\alpha(B,A)$ is the $H_\alpha$ of the conjugate
$d_\alpha^{-1}$ applied to $(A,B)$, so the two arguments record
the forward and backward Hausdorff sides separately.
\end{definition}

\begin{definition}[$d_\alpha$-bounded subsets]\label{def:cb-family}
A non-empty $A\subseteq X$ is \emph{$d_\alpha$-bounded} if
$\sup_{a,a'\in A}d_\alpha(a,a')<\infty$; equivalently (since
the supremum is taken over ordered pairs), if it is
$d_\alpha^s$-bounded. Write
\[
  \mathscr{P}_{cb}^{(\alpha)}(X):=
  \left\lbrace A\subseteq X:A\neq\emptyset,\ A\text{ is }\Us\text{-closed and }d_\alpha\text{-bounded}\right \rbrace.
\]
For $A,B\in\mathscr{P}_{cb}^{(\alpha)}(X)$ the functional
$H_\alpha(A,B)$ is finite-valued. The use of $\Us$-closedness
(rather than the finer $\tau(d_\alpha)$-closedness) keeps the
class consistent with the bicompleteness framework, in which the
ambient convergence is in $\Us$; the single-$d$ specialisation
recovers the $\mathscr{P}_{cb}(X)$ of~\cite{GabaKarapinarPetruselRadenovic2020}
when $\U=\U_d$.
\end{definition}

\begin{definition}[$(c)^*$-comparison function~\cite{GabaKarapinarPetruselRadenovic2020}]\label{def:c-star-comparison}
A function $\gamma\colon[0,\infty)\to[0,\infty)$ is a
\emph{$(c)^*$-comparison function} if
\begin{enumerate}[label=$(\gamma_{\arabic*})^{\!*}$]
  \item $\gamma$ is non-decreasing, $\gamma(0)=0$, and $0<\gamma(t)<t$ for every $t>0$;
  \item for every sequence $(t_n)\subseteq(0,\infty)$,
        $\displaystyle\sum_{n=1}^{\infty}\gamma(t_n)<\infty$ implies
        $\displaystyle\sum_{n=1}^{\infty}t_n<\infty$.
\end{enumerate}
The linear functions $\gamma(t)=kt$ with $k\in[0,1)$ satisfy
both axioms (the second by direct comparison); the inclusion is
strict (Example~\ref{ex:c-star-non-linear}).
\end{definition}

\begin{definition}\label{def:weakly-contractive-multi}
A multivalued map
$T\colon X\to\bigcap_{\alpha\in A}\mathscr{P}_{cb}^{(\alpha)}(X)$
is \emph{$\mathcal{D}$-weakly contractive} if there is a family
$\left\lbrace \gamma_\alpha:\alpha\in A \right \rbrace$ of $(c)^*$-comparison functions
such that, for every $\xi\in X$, there exists $\eta\in T\xi$
satisfying
\begin{align}
  H_\alpha(\left\lbrace \eta \right \rbrace,T\eta)&\le d_\alpha(\xi,\eta)-\gamma_\alpha\left(d_\alpha(\xi,\eta)\right),\label{eq:weakly-contractive-multi}\\
  H_\alpha(T\eta,\left\lbrace \eta \right \rbrace)&\le d_\alpha(\eta,\xi)-\gamma_\alpha\left(d_\alpha(\eta,\xi)\right)\label{eq:weakly-contractive-multi-conj}
\end{align}
for every $\alpha\in A$.
\end{definition}

\begin{theorem}\label{thm:hausdorff-direct}
Let $(X,\U)$ be a $T_0$ bicomplete quasi-uniform space generated
by $\mathcal{D}=\left\lbrace d_\alpha\right\rbrace_{\alpha\in A}$,
and let
$T\colon X\to\bigcap_{\alpha\in A}\mathscr{P}_{cb}^{(\alpha)}(X)$
be $\mathcal{D}$-weakly contractive. Assume in addition that $T$
is \emph{lower $\Us$-hemicontinuous}: for every
$\Us$-convergent sequence $\xi_n\to\xi$ in $X$ and every
$z\in T\xi$, there is a sequence $z_n\in T\xi_n$ with $z_n\to z$
in $\Us$. Then there exists $\xi^*\in X$ with
$T\xi^*=\left\lbrace \xi^*\right\rbrace$; in particular $\xi^*$ is
a $\U$-startpoint, a $\U$-endpoint, and a $\U$-fixed point of $T$
in the singleton sense $\xi^*\in T\xi^*$.
\end{theorem}

\begin{proof}
\emph{Construction of the iteration.} Fix $\xi_0\in X$. By
\eqref{eq:weakly-contractive-multi}--\eqref{eq:weakly-contractive-multi-conj}
applied at $\xi_0$, there exists $\xi_1\in T\xi_0$ such that, for
every $\alpha\in A$,
\[
  H_\alpha(\left\lbrace \xi_1 \right \rbrace,T\xi_1)\le d_\alpha(\xi_0,\xi_1)-\gamma_\alpha\left(d_\alpha(\xi_0,\xi_1)\right),
\]
together with the conjugate bound. Recursively, applying both
inequalities at $\xi_n$ produces $\xi_{n+1}\in T\xi_n$ with
\begin{align*}
  d_\alpha(\xi_{n+1},\xi_{n+2})
  &\le H_\alpha(\left\lbrace \xi_{n+1} \right \rbrace,T\xi_{n+1})
  \le d_\alpha(\xi_n,\xi_{n+1})-\gamma_\alpha\left(d_\alpha(\xi_n,\xi_{n+1})\right),\\
  d_\alpha(\xi_{n+2},\xi_{n+1})
  &\le H_\alpha(T\xi_{n+1},\left\lbrace \xi_{n+1} \right \rbrace)
  \le d_\alpha(\xi_{n+1},\xi_n)-\gamma_\alpha\left(d_\alpha(\xi_{n+1},\xi_n)\right).
\end{align*}

\emph{Bi-Cauchyness.} Fix $\alpha\in A$ and write
$a_n:=d_\alpha(\xi_n,\xi_{n+1})$ and
$b_n:=d_\alpha(\xi_{n+1},\xi_n)$. The estimates above give
$a_{n+1}\le a_n-\gamma_\alpha(a_n)$ and
$b_{n+1}\le b_n-\gamma_\alpha(b_n)$. Both sequences are
non-increasing and bounded below by $0$, hence convergent to some
limits $\ell,\ell'\ge 0$.

Telescoping the inequality
$\gamma_\alpha(a_n)\le a_n-a_{n+1}$ across $n=0,1,\dots,N-1$
gives $\sum_{n=0}^{N-1}\gamma_\alpha(a_n)\le a_0-a_N\le a_0$, so
$\sum_{n}\gamma_\alpha(a_n)\le a_0<\infty$. In particular
$\gamma_\alpha(a_n)\to 0$. Since $a_n\ge\ell$ and $\gamma_\alpha$
is non-decreasing, $\gamma_\alpha(\ell)\le\gamma_\alpha(a_n)$ for
every $n$, so $\gamma_\alpha(\ell)\le 0$, which by axiom
$(\gamma_1)^*$ of Definition~\ref{def:c-star-comparison} forces
$\ell=0$. The same argument on $(b_n)$ gives $\ell'=0$. This step
uses only the monotonicity of $\gamma_\alpha$ and the strict
positivity $\gamma_\alpha(t)>0$ for $t>0$; no right
upper-semicontinuity is required.

From $\sum_n\gamma_\alpha(a_n)<\infty$ and the $(c)^*$-summability
axiom $(\gamma_2)^*$, $\sum_n a_n<\infty$. Hence
$d_\alpha(\xi_n,\xi_m)\le\sum_{k=n}^{m-1}a_k\to 0$ as
$n\to\infty$, uniformly in $m>n$, so $(\xi_n)$ is left $K$-Cauchy
in $d_\alpha$. The analogous argument on $(b_n)$ uses the
conjugate iteration inequality (the second line of
Definition~\ref{def:weakly-contractive-multi}) and the
$(c)^*$-summability of the same $\gamma_\alpha$ applied to
$(b_n)$, delivering right $K$-Cauchyness of $(\xi_n)$ in
$d_\alpha$. Since this holds for every $\alpha\in A$, $(\xi_n)$
is bi-Cauchy in $(X,\Us)$, and bicompleteness yields
$\xi^*\in X$ with $\xi_n\to\xi^*$ in $\Us$.

We note for the record that each $T\xi_n$ lies in
$\bigcap_{\alpha\in A}\mathscr{P}_{cb}^{(\alpha)}(X)$ by the
standing hypothesis on the range of $T$, so every
$H_\alpha(\left\lbrace \xi_n\right\rbrace,T\xi_n)$ is
finite-valued and the iteration step is well-defined.

\emph{$T\xi^*=\left\lbrace \xi^*\right\rbrace$.} Fix $\alpha\in A$
and $z\in T\xi^*$. By lower $\Us$-hemicontinuity of $T$, there is
a sequence $z_n\in T\xi_n$ with $z_n\to z$ in $\Us$. The map
$(\xi,z)\mapsto d_\alpha(\xi,z)$ is $\Us$-continuous jointly in
both arguments: the asymmetric triangle gives
$|d_\alpha(\xi,z)-d_\alpha(\xi',z')|\le d_\alpha^{s}(\xi,\xi')+d_\alpha^{s}(z,z')$
for all $\xi,\xi',z,z'\in X$, and $\Us$-convergence is by
definition convergence in every $d_\alpha^s$. Applied to
$\xi_n\to\xi^*$ and $z_n\to z$, this gives
$d_\alpha(\xi_n,z_n)\to d_\alpha(\xi^*,z)$ as $n\to\infty$. For every $n$,
$d_\alpha(\xi_n,z_n)\le\sup_{w\in T\xi_n}d_\alpha(\xi_n,w)
=H_\alpha(\left\lbrace \xi_n\right\rbrace,T\xi_n)$, and the
right-hand side is at most
$a_{n-1}:=d_\alpha(\xi_{n-1},\xi_n)\to 0$ by the iteration
bound established in the previous step. Hence
$d_\alpha(\xi^*,z)=\lim_n d_\alpha(\xi_n,z_n)=0$. Since
$z\in T\xi^*$ and $\alpha\in A$ were arbitrary,
$H_\alpha(\left\lbrace \xi^*\right\rbrace,T\xi^*)=\sup_{z\in T\xi^*}d_\alpha(\xi^*,z)=0$
for every $\alpha$. The same argument applied to the conjugate
functional
$\xi\mapsto H_\alpha(T\xi,\left\lbrace \xi\right\rbrace)$, using the
contractive bound on $d_\alpha(z_n,\xi_n)$, gives
$H_\alpha(T\xi^*,\left\lbrace \xi^*\right\rbrace)=0$ for every $\alpha$. Unpacking, for
every $z\in T\xi^*$ and every $\alpha\in A$,
$d_\alpha(\xi^*,z)=d_\alpha(z,\xi^*)=0$, so
$(\xi^*,z),(z,\xi^*)\in\bigcap\U$. Antisymmetry of $\bigcap\U$
(the $T_0$ hypothesis) forces $z=\xi^*$. Hence
$T\xi^*\subseteq\left\lbrace \xi^* \right \rbrace$, and since $T\xi^*$ is non-empty by
the standing hypothesis on the range of $T$, $T\xi^*=\left\lbrace \xi^* \right \rbrace$.

The conclusions $\xi^*\in T\xi^*$, $(x^*,Tx^*)\in\bigcap\U$ and
$(Tx^*,x^*)\in\bigcap\U$ are immediate from $T\xi^*=\left\lbrace \xi^* \right \rbrace$.
\end{proof}

\begin{remark}[Relation to Corollary~\ref{cor:multi-contraction} and to~\cite{GabaKarapinarPetruselRadenovic2020}]\label{rem:hausdorff-vs-selector}
Theorem~\ref{thm:hausdorff-direct} and
Corollary~\ref{cor:multi-contraction} are not nested. The
corollary asks for the existence of a contractive single-valued
selector $f$ of $T$ (a structural property of $f$); the theorem
asks for a $\mathcal{D}$-weakly contractive condition on $T$
itself, working directly with the family
$\left\lbrace H_\alpha \right \rbrace_{\alpha\in A}$ on subsets. The two routes deliver
different conclusions: the corollary produces a $\U$-startpoint
at the fixed point of $f$, with witness $f(x^*)=x^*$, but $T$
may have other points in its image $Tx^*$; the theorem produces
a $\xi^*$ at which $T\xi^*$ collapses to the singleton
$\left\lbrace \xi^* \right \rbrace$, a strictly stronger conclusion at the cost of a
strictly stronger hypothesis on $T$.

The single-quasi-pseudometric, left $K$-complete case
$\mathcal{D}=\left\lbrace d \right \rbrace$ recovers
\cite[Theorem~1]{GabaKarapinarPetruselRadenovic2020} on the
forward side: the one-sided condition
\eqref{eq:weakly-contractive-multi} alone, combined with left
$K$-completeness rather than bicompleteness, is enough to
deliver a start-point in the
\cite{GabaKarapinarPetruselRadenovic2020} sense
($H(\left\lbrace \xi^* \right \rbrace,T\xi^*)=0$), because the left $K$-Cauchy sequence
produced by the iteration converges in the left-completeness
topology without invoking the conjugate inequality. The two-sided
condition above is what lets the conclusion be lifted to
bicompleteness and stated symmetrically in
$\U$ and $\Uinv$.
\end{remark}

\begin{example}[A non-linear $(c)^*$-comparison function]\label{ex:c-star-non-linear}
The $(c)^*$-comparison class strictly contains the Banach class
$\left\lbrace \gamma(t)=kt:k\in[0,1) \right \rbrace$. The function
$\gamma\colon[0,\infty)\to[0,\infty)$ defined by
$\gamma(t):=\tfrac{t}{1+t}$ is a $(c)^*$-comparison function: it is
non-decreasing (derivative $\tfrac{1}{(1+t)^2}>0$), $\gamma(0)=0$, and
$\gamma(t)<t$ for $t>0$ since $1<1+t$. For the summability
implication, if
$\sum_n\gamma(t_n)=\sum_n \tfrac{t_n}{1+t_n}<\infty$ then $t_n\to 0$,
so for $n$ large enough that $t_n\le 1$ we have
$\tfrac{t_n}{1+t_n}\ge \tfrac{t_n}{2}$, giving
$\sum_n t_n\le 2\sum_n \tfrac{t_n}{1+t_n}+\text{(finite tail)}<\infty$.
The ratio $\tfrac{\gamma(t)}{t}=\tfrac{1}{1+t}$ tends to $1$ as $t\to 0^+$, so
no constant $k<1$ dominates $\gamma$ on a neighbourhood of $0$:
the linear class is a proper subclass of the $(c)^*$-class.
\end{example}

\begin{example}[Theorem~\ref{thm:hausdorff-direct} on the half-interval map]\label{ex:hausdorff-direct-worked}
On $X=[0,1]$ with the upper quasi-pseudometric
$d(x,y):=\max\left\lbrace x-y,0 \right \rbrace$ of Example~\ref{ex:asymmetric-norm},
take the multivalued map $T(x):=[0,\tfrac{x}{2}]$ of
Example~\ref{ex:endpoint-multivalued}. Each $T(x)$ is
$d^s$-closed and $d^s$-bounded (a closed sub-interval of
$[0,1]$), so $T(x)\in\mathscr{P}_{cb}^{(d)}(X)$.

\emph{Computing $H$.} With $A=\left\lbrace \eta \right \rbrace$
a singleton, Definition~\ref{def:hausdorff-family} collapses:
$\sup_{a\in A}\inf_{b\in B}d(a,b)$ becomes just
$\inf_{b\in B}d(\eta,b)$ and is dominated by the second
component $\sup_{b\in B}d(\eta,b)$, so
$H(\left\lbrace \eta \right \rbrace,B)=\sup_{b\in B}d(\eta,b)$
whenever the supremum is finite. Applied to $B=T\eta=[0,\tfrac{\eta}{2}]$
this gives
\[
  H(\left\lbrace \eta \right \rbrace,T\eta)
  =\sup_{z\in[0,\tfrac{\eta}{2}]}d(\eta,z)
  =\sup_{z\in[0,\tfrac{\eta}{2}]}(\eta-z)^+
  =\eta,
\]
attained at $z=0$. Symmetrically,
$H(T\eta,\left\lbrace \eta \right \rbrace)=\sup_{z\in[0,\tfrac{\eta}{2}]}d(z,\eta)
=\sup_{z\in[0,\tfrac{\eta}{2}]}(z-\eta)^+=0$
(every $z\in[0,\tfrac{\eta}{2}]$ satisfies $z\le\tfrac{\eta}{2}\le\eta$).

\emph{Verifying $\left\lbrace d \right \rbrace$-weak contractivity with
$\gamma(t):=\tfrac{t}{2}$.} For each $\xi\in X$ we choose
$\eta:=\tfrac{\xi}{3}\in[0,\tfrac{\xi}{2}]=T\xi$. Then $H(\left\lbrace \eta \right \rbrace,T\eta)=\eta=\tfrac{\xi}{3}$,
$d(\xi,\eta)=(\xi-\tfrac{\xi}{3})^+=\tfrac{2\xi}{3}$, and
$\gamma(d(\xi,\eta))=\tfrac{\xi}{3}$, so the forward
inequality~\eqref{eq:weakly-contractive-multi} reads
$\tfrac{\xi}{3}\le \tfrac{2\xi}{3}-\tfrac{\xi}{3}$, holding with equality. The conjugate
inequality~\eqref{eq:weakly-contractive-multi-conj} is trivial:
$H(T\eta,\left\lbrace \eta \right \rbrace)=0$ and $d(\eta,\xi)=(\tfrac{\xi}{3}-\xi)^+=0$, so
$0\le 0-\gamma(0)=0$.

\emph{Conclusion.} Theorem~\ref{thm:hausdorff-direct} delivers
$\xi^*\in X$ with $T\xi^*=\left\lbrace \xi^* \right \rbrace$. Unpacking,
$[0,\xi^*/2]=\left\lbrace \xi^* \right \rbrace$ forces $\xi^*=0$ and $T(0)=\left\lbrace 0 \right \rbrace$. The
selector route via $f(x)=\tfrac{x}{2}$
(Corollary~\ref{cor:multi-contraction}) reaches the same point
$0$, but the conclusion delivered by
Theorem~\ref{thm:hausdorff-direct} is strictly stronger: it
identifies $0$ as the unique point where the entire image
collapses to a singleton, not merely as a point with at least one
startpoint witness. The choice $\eta=\tfrac{\xi}{3}$ above is essentially
the unique one that produces equality in both bounds; smaller
choices of $\eta$ (e.g.\ $\eta=\tfrac{\xi}{4}$) verify both inequalities
strictly, with the same conclusion.
\end{example}

\begin{remark}[The conjugate line is often vacuous on a single upper quasi-pseudometric]\label{rem:conjugate-vacuous}
Example~\ref{ex:hausdorff-direct-worked} illustrates a structural
feature of Definition~\ref{def:weakly-contractive-multi} on a single
upper quasi-pseudometric $d(x,y)=(x-y)^+$: whenever $T$ maps each
$\xi$ into a subset of $[0,\xi]$ (a downward contraction toward
$x^*=0$), the witness $\eta\in T\xi$ satisfies $\eta\le\xi$, so
$d(\eta,\xi)=0$ and $H(T\eta,\{\eta\})=0$, and the conjugate
inequality~\eqref{eq:weakly-contractive-multi-conj} reduces to
$0\le 0-\gamma(0)=0$. The forward
inequality~\eqref{eq:weakly-contractive-multi} carries all the
contractive content in this asymmetric single-qpm regime. Genuinely
two-sided instances of
Definition~\ref{def:weakly-contractive-multi} arise in two
natural settings: the multi-qpm family
$\mathcal{D}=\left\lbrace d_1,d_2\right\rbrace$ with
$d_1(x,y)=(x-y)^+$ and $d_2(x,y)=(y-x)^+$, where the two
qpms activate opposite direction bounds independently and a map
$T$ with image straddling the witness delivers non-trivial
constraints on both $H_1$-lines and both $H_2$-lines; and the
function-space setting of Example~\ref{ex:continuous-selector},
in which the asymmetric sup quasi-pseudometric on
$C([0,1],\R_{\ge 0})$ acts non-monotonically on set-valued
operators whose images are not $d$-ordered relative to the
witness. The single-qpm collapse is intrinsic to the asymmetric
geometry, and the definition remains well-posed.
\end{remark}

\begin{example}[Selector route applies, direct route does not]\label{ex:selector-yes-direct-no}
The hypotheses of Corollary~\ref{cor:multi-contraction} and
Theorem~\ref{thm:hausdorff-direct} are not equivalent. Modify
Example~\ref{ex:hausdorff-direct-worked} by adjoining a fixed
upper-edge element:
\[
  T'(x):=[0,\tfrac{x}{2}]\cup\left\lbrace 1 \right \rbrace,\qquad x\in[0,1].
\]
Each $T'(x)$ remains $d^s$-closed and $d^s$-bounded.

\emph{Corollary~\ref{cor:multi-contraction} applies.} The map
$f(x):=\tfrac{x}{2}$ is a selector of $T'$ (since
$\tfrac{x}{2}\in[0,\tfrac{x}{2}]\subseteq T'(x)$) and is the $\left\lbrace d \right \rbrace$-contraction
of Example~\ref{ex:asymmetric-norm} with constant $\tfrac{1}{2}$. Its
unique fixed point is $0$, and
Corollary~\ref{cor:multi-contraction} identifies $0$ as a
$\U$-startpoint of $T'$, witnessed by $f(0)=0\in T'(0)=\left\lbrace 0,1 \right \rbrace$.

\emph{Theorem~\ref{thm:hausdorff-direct} cannot apply.} Its
conclusion would force $T'(\xi^*)=\left\lbrace \xi^* \right \rbrace$ for some $\xi^*$.
But $T'(x)=[0,\tfrac{x}{2}]\cup\left\lbrace 1 \right \rbrace$ contains $1$ for every $x$, and
contains $0$ for every $x$, so $T'(x)$ is a singleton at no
$x\in[0,1]$. Tracing the hypothesis at $\xi=0$: with $\eta=0$, the
conjugate $H$-functional
$H(T'(0),\left\lbrace 0 \right \rbrace)$ is computed from
Definition~\ref{def:hausdorff-family} at
$A=T'(0)=\left\lbrace 0,1 \right \rbrace$ and $B=\left\lbrace 0 \right \rbrace$:
\[
  \sup_{a\in A}\inf_{b\in B}d(a,b)=\max\left\lbrace d(0,0),d(1,0) \right \rbrace=\max\left\lbrace 0,1 \right \rbrace=1,
\]
\[
  \sup_{b\in B}\inf_{a\in A}d(a,b)=\inf\left\lbrace d(0,0),d(1,0) \right \rbrace=\min\left\lbrace 0,1 \right \rbrace=0,
\]
so $H(T'(0),\left\lbrace 0 \right \rbrace)=\max\left\lbrace 1,0 \right \rbrace=1$. This must be bounded by
$d(0,0)-\gamma(0)=0$, which fails; with $\eta=1$, the forward
$H$-functional collapses to the same singleton pattern
$H(\left\lbrace 1 \right \rbrace,B)=\sup_{b\in B}d(1,b)$ observed above, giving
\[
  H(\left\lbrace 1 \right \rbrace,T'(1))=\sup_{z\in[0,\tfrac{1}{2}]\cup\left\lbrace 1 \right \rbrace}(1-z)^+=1,
\]
which must be bounded by $d(0,1)-\gamma(d(0,1))=0-\gamma(0)=0$
and also fails. So no choice of $\eta\in T'(0)$ satisfies the
weakly contractive condition at $\xi=0$.

The example pinpoints what the singleton conclusion of
Theorem~\ref{thm:hausdorff-direct} costs: it rules out target
images that retain ``persistent'' elements (here the constant
upper edge $1$) which the contractive condition cannot squeeze
away. The selector route of
Corollary~\ref{cor:multi-contraction} is more accommodating
because it only requires \emph{some} branch of $T'$ to be
contractive, not the whole graph.
\end{example}

\section{Common \texorpdfstring{$\U$}{U}-startpoints for finite families}\label{sec:common}

A standard strengthening of the startpoint problem asks for a single
point that serves as a startpoint for several multivalued maps at
once. In the metric setting this is the common-fixed-point question
addressed by Jungck~\cite{Jungck} and others. The selector
correspondence of the previous section reduces the quasi-uniform
version to the corresponding common-fixed-point question for
selectors. We treat finite families
$T_1,\dots,T_n$ in this section; an infinite-family extension is
possible under the additional hypothesis that the family is
\emph{equicontractive} (all $T_i$ share a common
$\mathcal{D}$-contraction constant $k_\alpha<1$ for every $\alpha$,
not just per-$i$), which guarantees that the common fixed point of
finite subfamilies stabilises as more maps are added.

A point $x^*\in X$ is a \emph{common $\U$-startpoint} of multivalued
maps $T_1,\dots,T_n\colon X\to 2^X$ if for each $i$ some $z_i\in T_i
x^*$ satisfies $(x^*,z_i)\in\bigcap\U$. Applying
Proposition~\ref{prop:triv-multi} pointwise gives:

\begin{proposition}\label{prop:common-trivial}
Let $(X,\U)$ be a quasi-uniform space, $T_1,\dots,T_n\colon X\to 2^X$
and $x^*\in X$ with $\bigcup_{i=1}^n T_i x^*\neq\emptyset$. The
following are equivalent:
\begin{enumerate}[label=\textup{(\roman*)}]
  \item For every $U\in\U$ and every $z\in\bigcup_{i=1}^n T_i x^*$,
        $U(x^*)\cap U^{-1}(z)\neq\emptyset$;
  \item $(x^*,z)\in\bigcap\U$ for every $z\in\bigcup_{i=1}^n T_i x^*$;
  \item $x^*$ is a common $\U$-startpoint of $T_1,\dots,T_n$ \emph{and}
        every choice of $z_i\in T_i x^*$ realises this.
\end{enumerate}
\end{proposition}

As in the single-map case, condition~(i) is just a restatement of
the conclusion. Producing a genuine common $\U$-startpoint requires
a structural input. Recall that a \emph{common fixed point} of a
family $f_1,\dots,f_n\colon X\to X$ is a point $x^*\in X$ satisfying
$f_i(x^*)=x^*$ for every $i$. Combining
Theorem~\ref{thm:contraction} with the selector correspondence
(Proposition~\ref{prop:selector}) gives:

\begin{theorem}\label{thm:common-existence}
Let $(X,\U)$ be a $T_0$ bicomplete quasi-uniform space generated by
$\mathcal{D}=\left\lbrace d_\alpha \right \rbrace$. For each $i=1,\dots,n$ let
$T_i\colon X\to 2^X$ be a multivalued map with $T_i x\neq\emptyset$
for all $x\in X$, and suppose $T_i$ admits a selector $f_i$
satisfying~\eqref{eq:contraction}, with unique fixed point
$x_i^*\in X$ given by Theorem~\ref{thm:contraction}. If the
selectors $f_1,\dots,f_n$ pairwise commute, then
$x_1^*=\cdots=x_n^*=:x^*$, and $x^*$ is a common $\U$-startpoint
of $T_1,\dots,T_n$ witnessed for each $i$ by $z_i=x^*\in T_i x^*$.
\end{theorem}

\begin{proof}
Fix indices $i\neq j$. Since $f_i$ and $f_j$ commute,
$f_j\left(f_i(x_j^*)\right)=f_i\left(f_j(x_j^*)\right)=f_i(x_j^*)$,
so $f_i(x_j^*)$ is a fixed point of $f_j$; by uniqueness of $f_j$'s
fixed point, $f_i(x_j^*)=x_j^*$. So $x_j^*$ is itself a fixed point
of $f_i$, and uniqueness of $f_i$'s fixed point now forces
$x_j^*=x_i^*$. Hence pairwise commutativity forces
$x_1^*=\cdots=x_n^*$; call the common fixed point $x^*$. Then
$f_i(x^*)=x^*\in T_i x^*$ for each $i$, and
$(x^*,x^*)\in\diag\subseteq\bigcap\U$, so $x^*$ is a common
$\U$-startpoint of $T_1,\dots,T_n$ with witness $z_i=x^*$.
\end{proof}

The next example exhibits Theorem~\ref{thm:common-existence} on a
familiar space and shows that commutativity is the structural
ingredient without which the conclusion fails.

\begin{example}[Two commuting affine selectors]\label{ex:common-affine}
Let $X=\R$ with the upper quasi-pseudometric $d(x,y)=\max\left\lbrace x-y,0 \right \rbrace$
of Example~\ref{ex:asymmetric-norm}, so that $(X,\U_d)$ is $T_0$ and
bicomplete. Fix constants $k_1,k_2\in[0,1)$ and a real number
$c\in\R$, and consider the multivalued maps
\begin{align*}
  T_1(x) &:= \left[\min\left\lbrace k_1 x,k_1 x+(1-k_1)c \right \rbrace,
                    \max\left\lbrace k_1 x,k_1 x+(1-k_1)c \right \rbrace\right],\\
  T_2(x) &:= \left[\min\left\lbrace k_2 x,k_2 x+(1-k_2)c \right \rbrace,
                    \max\left\lbrace k_2 x,k_2 x+(1-k_2)c \right \rbrace\right].
\end{align*}
The affine selectors $f_i(x):=k_i x+(1-k_i)c$ are monotone
nondecreasing $\left\lbrace d \right \rbrace$-contractions with constants $k_1,k_2$
respectively, since each is a Euclidean contraction with the same
constant and is monotone (cf.\ Example~\ref{ex:asymmetric-norm}).
Both have $c$ as their unique fixed point, and they commute:
$f_1\circ f_2(x)=k_1 k_2 x+(1-k_1 k_2)c=f_2\circ f_1(x)$.
Theorem~\ref{thm:common-existence} therefore yields the common
$\U$-startpoint $x^*=c$, witnessed for each $T_i$ by
$z_i=c\in T_i(c)$.
\end{example}

The next example exhibits a concrete instance in which
Theorem~\ref{thm:common-existence} fails to apply, and shows that
the failure mode is not a defect of the theorem but a genuine
obstruction in the data.

\begin{example}[Distinct fixed points without commutativity]\label{ex:common-sharp}
On $\R$ with the upper quasi-pseudometric of
Example~\ref{ex:asymmetric-norm}, consider the single-valued maps
\[
  f_1(x):=\tfrac12 x, \qquad f_2(x):=\tfrac12 x+1,
\]
viewed as selectors of the multivalued maps
$T_i(x):=\left\lbrace f_i(x) \right \rbrace$. Both $f_1$ and $f_2$ are monotone
nondecreasing $\left\lbrace d \right \rbrace$-contractions with constant $\tfrac12$, hence
Theorem~\ref{thm:contraction} gives them unique fixed points
$x_1^*=0$ and $x_2^*=2$. The selectors do not commute:
$f_1\circ f_2(x)=\tfrac14 x+\tfrac12$ while
$f_2\circ f_1(x)=\tfrac14 x+1$, so the commutativity hypothesis of
Theorem~\ref{thm:common-existence} fails and the
unique-common-fixed-point conclusion fails with it: $x_1^*\neq
x_2^*$. The common $\U$-startpoint set is not empty, however. A
candidate $x^*$ is a common $\U$-startpoint iff $x^*\le f_1(x^*)$
and $x^*\le f_2(x^*)$, i.e.\ $x^*\le x^*/2$ and $x^*\le x^*/2+1$;
the first inequality forces $x^*\le 0$, and both inequalities hold
at every $x^*\le 0$. So the common $\U$-startpoint set is
$(-\infty,0]$. The example illustrates the genuine asymmetry of
the situation: the asymmetric preorder $\le$ admits a continuum of
common $\U$-startpoints even when commutativity (and hence the
unique-common-fixed-point clause of
Theorem~\ref{thm:common-existence}) fails. The startpoint and the
fixed-point notions therefore have genuinely different
behaviour under failure of commutativity: startpoints persist,
fixed points need not coincide.
\end{example}

\section{The \texorpdfstring{$T_0$}{T0}-quotient and an involution theorem}\label{sec:quotient}

The involution theorem of~\cite[Theorem~2.3]{Gaba2018} promotes
a $\U$-fixed point of an order-preserving involution to an actual
fixed point, but only under the $T_0$ assumption: without $T_0$ the
preorder $\bigcap\U$ has non-trivial equivalence classes and a
$\U$-fixed point can fail to be one in the ordinary sense. The
$T_0$-quotient is exactly the device that handles this. Quotienting
$X$ by the equivalence $x\equiv_\U y$ that identifies points lying
in the same class collapses the obstruction, and the involution
theorem transfers to the quotient verbatim. The cost is that the
equality $T(x^*)=x^*$ on $X$ is replaced by an equivalence
$T(x^*)\equiv_\U x^*$, which is the natural sharpening of fixed
point in the absence of $T_0$.

For a quasi-uniform space $(X,\U)$, the relation
$x\equiv_\U y\iff (x,y)\in\bigcap\U\text{ and }(y,x)\in\bigcap\U$
is an equivalence on $X$. Let $\widetilde X:=X/\!\equiv_\U$ and
$\pi\colon X\to\widetilde X$ be the quotient map. The filter
$\widetilde\U$ on $\widetilde X\times\widetilde X$ generated by
the base
$\left\lbrace (\pi\times\pi)(U):U\in\U\right\rbrace$ is a
$T_0$-quasi-uniformity called the \emph{$T_0$-quotient}
of $(X,\U)$~\cite[\S2.4]{Kunzi}; the construction satisfies
\begin{equation}\label{eq:quotient-characterisation}
  \pi(x)\le_{\widetilde\U}\pi(y)\iff(x,y)\in\bigcap\U
\end{equation}
and reduces to the identity when $\U$ is already $T_0$. Beyond
the standard treatment in~\cite{Kunzi}, the only step we need to
isolate for Lemma~\ref{lem:descent} is the (U2)-axiom verification
on $\widetilde\U$: given $W\in\widetilde\U$, pick $U\in\U$ with
$(\pi\times\pi)(U)\subseteq W$ and, by three applications of (U2)
in $\U$, choose $V\in\U$ with $V^{\circ 8}\subseteq U$ (hence
also $V^{\circ 5}\subseteq U$). If
$(\pi(x),\pi(z))\in(\pi\times\pi)(V)\circ(\pi\times\pi)(V)$, an
intermediate $\widetilde y$ unpacks to preimages with
$(x_0,y_1),(y_2,z_0)\in V$ and
$x_0\equiv_\U x,\ y_1\equiv_\U y_2,\ z_0\equiv_\U z$, so
$(x,x_0),(y_1,y_2),(z_0,z)\in\bigcap\U\subseteq V$. Composition
along the chain $(x,x_0,y_1,y_2,z_0,z)$, five $V$-steps in all,
gives $(x,z)\in V^{\circ 5}\subseteq U$, and hence
$(\pi(x),\pi(z))\in(\pi\times\pi)(U)\subseteq W$. The
\emph{factor of five} (two endpoint identifications, one middle
identification, two $V$-hops) is the structural feature
exploited by the descent lemma below; the rest of the
construction is standard.

We say $T\colon X\to X$ is \emph{order-preserving} if
$(x,y)\in\bigcap\U$ implies $(Tx,Ty)\in\bigcap\U$. Order-preservation
descends $T$ to a well-defined map
$\widetilde T\colon\widetilde X\to\widetilde X$ given by
$\widetilde T(\pi(x))=\pi(Tx)$. Well-definedness requires that
$x\equiv_\U y$ implies $Tx\equiv_\U Ty$, which we verify:
$x\equiv_\U y$ unpacks to $(x,y),(y,x)\in\bigcap\U$, and
order-preservation applied to each gives $(Tx,Ty),(Ty,Tx)\in\bigcap\U$,
i.e.\ $Tx\equiv_\U Ty$.

\begin{lemma}\label{lem:descent}
Let $T\colon X\to X$ be order-preserving. If $x_0$ is a $\U$-fixed
point of $T$, then $\pi(x_0)$ is a $\widetilde\U$-fixed point of
$\widetilde T$. If $T$ is also an involution, i.e.\ $T\circ T=\mathrm{id}_X$,
then $\widetilde T$ is an involution and
$\widetilde T(\pi(x_0))=\pi(x_0)$ in $\widetilde X$, equivalently
$T(x_0)\equiv_\U x_0$.
\end{lemma}

\begin{proof}
$(x_0,Tx_0)\in\bigcap\U$ implies $\pi(x_0)\le_{\widetilde\U}\pi(Tx_0)
=\widetilde T(\pi(x_0))$, which is precisely a $\widetilde\U$-fixed
point statement.

If $T$ is an involution, then $\widetilde T\circ\widetilde T=
\mathrm{id}_{\widetilde X}$. Order-preservation propagated to the
quotient gives, on top of $\pi(x_0)\le_{\widetilde\U}\widetilde T(\pi(x_0))$,
also
\[
  \widetilde T(\pi(x_0))\le_{\widetilde\U}\widetilde T(\widetilde T(\pi(x_0)))=\pi(x_0).
\]
Antisymmetry on $\widetilde X$ (the $T_0$ property) yields
$\widetilde T(\pi(x_0))=\pi(x_0)$, equivalently $Tx_0\equiv_\U x_0$.
\end{proof}

\begin{theorem}\label{thm:involution}
Let $(X,\U)$ be a quasi-uniform space (not necessarily $T_0$),
$T\colon X\to X$ an order-preserving involution. If there exists
$x^*\in X$ with $U(x^*)\cap U^{-1}(Tx^*)\neq\emptyset$ for every
$U\in\U$, then $T(x^*)\equiv_\U x^*$. In particular, $\pi(x^*)$ is a
genuine fixed point of $\widetilde T$ in $\widetilde X$, and when $\U$
is $T_0$ this gives $T(x^*)=x^*$.
\end{theorem}

\begin{proof}
The hypothesis $U(x^*)\cap U^{-1}(Tx^*)\neq\emptyset$ for every
$U\in\U$ is condition~(i) of
Proposition~\ref{prop:triv-single} applied to $T$ at $x^*$. The
equivalence of (i) and (ii) in that proposition gives
$(x^*,Tx^*)\in\bigcap\U$, i.e.\ $x^*$ is a $\U$-fixed point of
$T$ in the sense of
Definition~\ref{def:fixed-start-end}(a).

We now apply Lemma~\ref{lem:descent} to $x_0:=x^*$. The lemma
requires $T$ to be order-preserving (a standing hypothesis here)
and to have a $\U$-fixed point (just established). The second
clause of the lemma requires $T$ to be an involution, which is
also a standing hypothesis. The lemma therefore concludes
$\widetilde T(\pi(x^*))=\pi(x^*)$ in $\widetilde X$, equivalently
$Tx^*\equiv_\U x^*$. This gives the asserted conclusion
$T(x^*)\equiv_\U x^*$.

When $\U$ is $T_0$, the preorder $\bigcap\U$ is antisymmetric, so
$Tx^*\equiv_\U x^*$ (i.e.\ both $(Tx^*,x^*)\in\bigcap\U$ and
$(x^*,Tx^*)\in\bigcap\U$) forces $Tx^*=x^*$. This recovers
\cite[Theorem~2.3]{Gaba2018} as a special case.
\end{proof}

\begin{remark}
Theorem~\ref{thm:involution} acquires content when paired with
Theorem~\ref{thm:contraction}: contraction produces a $\U$-fixed
point, the quotient lifts it to a fixed point of $\widetilde T$,
and the involution hypothesis closes the gap between $\le$ and $=$
when $T_0$ is absent.
\end{remark}

The next example shows that the conclusion of
Theorem~\ref{thm:involution} is genuinely sharper than ordinary
fixed-point existence in the absence of $T_0$, and that the
$T_0$-quotient is essential for the conclusion to make sense.

\begin{example}[A non-$T_0$ cylinder]\label{ex:cylinder}
Let $X:=[0,1]\times\left\lbrace 0,1 \right \rbrace$, with quasi-pseudometric
\[
  d\left((s,a),(t,b)\right):=\max\left\lbrace s-t,0 \right \rbrace
\]
that depends only on the first coordinate (the upper
quasi-pseudometric of Example~\ref{ex:asymmetric-norm} composed
with projection onto the first factor). The associated
quasi-uniformity $\U$ has
\[
  \bigcap\U=\left\lbrace((s,a),(t,b)):s\le t\right \rbrace,
\]
so $(s,0)\equiv_\U(s,1)$ for every $s\in[0,1]$ and $\U$ is not
$T_0$. The $T_0$-quotient $\widetilde X$ identifies with $[0,1]$
endowed with the upper quasi-pseudometric of
Example~\ref{ex:asymmetric-norm}; its symmetrisation is the
Euclidean uniformity on $[0,1]$, which is complete, so
$(\widetilde X,\widetilde\U)$ is bicomplete.

Define $T\colon X\to X$ by $T(s,a):=(s,1-a)$. The map $T$ swaps
the two sheets of the cylinder, fixes the first coordinate, and
satisfies $T\circ T=\mathrm{id}_X$, so it is an order-preserving
involution. For every $x^*=(s,a)\in X$,
$d(x^*,Tx^*)=d(x^*,(s,1-a))=0$, so the existence hypothesis
$U(x^*)\cap U^{-1}(Tx^*)\neq\emptyset$ of
Theorem~\ref{thm:involution} holds (with witness $x^*$ itself).
The conclusion $T(x^*)\equiv_\U x^*$ also holds, since $T(x^*)$
and $x^*$ share their first coordinate. On the quotient,
$\widetilde T$ is the identity on $[0,1]$, and every point of
$\widetilde X$ is a fixed point of $\widetilde T$.

The example is sharp in two ways. First, the conclusion cannot be
strengthened to $T(x^*)=x^*$, because $T$ has no fixed point in
$X$: $1-a\neq a$ for $a\in\left\lbrace 0,1 \right \rbrace$. The involution sharpens the
preorder relation $x^*\le Tx^*$ to the equivalence
$x^*\equiv_\U Tx^*$, and that equivalence is the strongest
relation on $X$ available in the absence of $T_0$. Second, the
quotient is essential: without it, the statement ``$T$ has a fixed
point'' is false, while at the quotient level it becomes
universally true.

The example also shows why the involution hypothesis is not
compatible with the contraction hypothesis of
Theorem~\ref{thm:contraction}: an involution that is a strict
contraction would force $\mathrm{id}_X=T\circ T$ to be a strict
contraction, which is impossible unless $X$ is a point. The two
theorems cover disjoint regimes, and the conclusion of
Theorem~\ref{thm:involution} is the strongest one that the
involution hypothesis can deliver on its own. Figure~\ref{fig:cylinder}
shows the geometry.
\end{example}

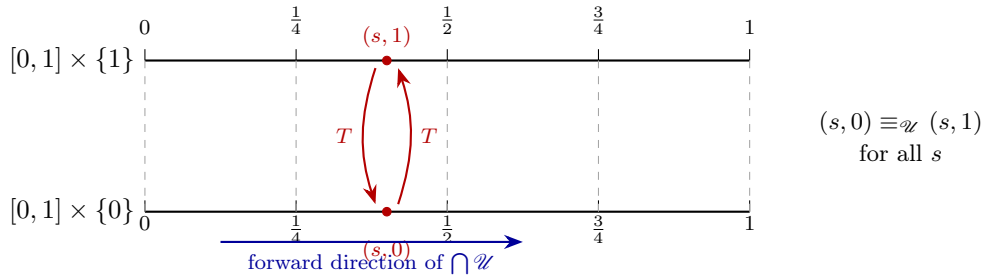
\begin{figure}[htbp]
\centering
\begin{tikzpicture}[x=1cm, y=1cm, >={Stealth[length=2.5mm]}]
  \draw[thick] (0,1.4) -- (8,1.4);
  \node[left, font=\small] at (0,1.4) {$[0,1]\times\left\lbrace 1 \right \rbrace$};
  \foreach \x/\lab in {0/0, 2/{\frac14}, 4/{\frac12}, 6/{\frac34}, 8/1} {
    \draw (\x,1.4) -- (\x,1.55);
    \node[above, yshift=2pt, font=\scriptsize] at (\x,1.55) {$\lab$};
  }
  \draw[thick] (0,-0.6) -- (8,-0.6);
  \node[left, font=\small] at (0,-0.6) {$[0,1]\times\left\lbrace 0 \right \rbrace$};
  \foreach \x/\lab in {0/0, 2/{\frac14}, 4/{\frac12}, 6/{\frac34}, 8/1}
    \draw (\x,-0.6) -- (\x,-0.45) node[below, yshift=-2pt, font=\scriptsize] {$\lab$};
  \foreach \x in {0,2,4,6,8} {
    \draw[dashed, gray!70] (\x,-0.6) -- (\x,1.4);
  }
  \filldraw[red!70!black] (3.2,1.4) circle (1.6pt) node[above=1pt, font=\scriptsize] {$(s,1)$};
  \filldraw[red!70!black] (3.2,-0.6) circle (1.6pt) node[below=1pt, font=\scriptsize, yshift=-6pt] {$(s,0)$};
  \draw[->, thick, red!70!black] (3.05,1.3) to[bend right=18] node[midway, left, font=\scriptsize] {$T$} (3.05,-0.5);
  \draw[->, thick, red!70!black] (3.35,-0.5) to[bend right=18] node[midway, right, font=\scriptsize] {$T$} (3.35,1.3);
  \node[font=\footnotesize, align=center] at (10,0.4)
    {$(s,0)\equiv_\U(s,1)$\\[1pt] for all $s$};
  \draw[->, thick, blue!60!black] (1,-1.0) -- (5,-1.0)
        node[midway, below, font=\scriptsize] {forward direction of $\bigcap\U$};
\end{tikzpicture}
\caption{The non-$T_0$ cylinder of Example~\ref{ex:cylinder}.
Dashed verticals indicate the equivalence
$(s,0)\equiv_\U(s,1)$. The map $T$ swaps the two sheets along
each equivalence class, has no ordinary fixed point, but satisfies
$T(x^*)\equiv_\U x^*$ for every $x^*\in X$. The $T_0$-quotient
$\widetilde X=[0,1]$ collapses the dashed pairs and turns
$\widetilde T$ into the identity.}
\label{fig:cylinder}
\end{figure}

\begin{remark}[Functoriality of the $T_0$-quotient]\label{rem:functoriality}
The $T_0$-quotient is functorial. Given an order-preserving
uniformly continuous map $T\colon(X,\U)\to(Y,\mathscr{V})$, the
fact that $T$ respects $\equiv_\U$ and $\equiv_{\mathscr{V}}$ gives
a unique descent
$\widetilde T\colon(\widetilde X,\widetilde\U)\to(\widetilde Y,\widetilde{\mathscr{V}})$
that is again order-preserving and uniformly continuous and that
fits in the expected commutative square with the projections.
Lemma~\ref{lem:descent} is the case $Y=X$, $\mathscr{V}=\U$. The
same functoriality underlies the $T_0$-reflection adjunction we
mention in Section~\ref{sec:conclusion}: an arbitrary fixed-point
theorem on $(\widetilde X,\widetilde\U)$ should lift to a
statement on $(X,\U)$ modulo $\equiv_\U$ via the unit of the
adjunction.
\end{remark}

\section{Concluding remarks}\label{sec:conclusion}

The framework of~\cite{Gaba2018} provides a working language for
fixed-point statements in the asymmetric setting. The contraction
principle of Section~\ref{sec:contraction} and the extensions of
Sections~\ref{sec:endpoint}--\ref{sec:quotient} attach to that
language a body of existence results whose hypotheses are
checkable directly from $(X,\U,\mathcal{D},T)$.

The paper develops five routes to $\U$-fixed points and
$\U$-startpoints, each with its own scope.
Theorem~\ref{thm:contraction} (the Banach-style contraction
principle) requires a contractive self-map on a $T_0$ bicomplete
quasi-uniform space and produces a unique fixed point together
with a Picard error bound. Theorem~\ref{thm:boyd-wong} (the
Boyd--Wong companion) loosens the contraction constant to a
comparison function at the cost of a non-geometric convergence
rate. Corollary~\ref{cor:multi-contraction} (the selector route)
lifts the contraction principle to multivalued maps by passing to
a contractive selector, and
Theorem~\ref{thm:hausdorff-direct} (the direct quasi-Hausdorff
route, extending~\cite{GabaKarapinarPetruselRadenovic2020} from
quasi-pseudometric to quasi-uniform spaces) does so without a
selector, at the cost of a two-sided weakly contractive condition
on $T$ that forces $T\xi^*$ to collapse to the singleton
$\left\lbrace \xi^* \right \rbrace$. Proposition~\ref{prop:knaster-tarski} (the
Knaster--Tarski companion) abandons completeness of $\U$
altogether and instead asks for order-completeness of
$(X,\le_\U)$ and a monotone selector. Finally,
Theorem~\ref{thm:involution} (the involution theorem) abandons
$T_0$ and produces a fixed point at the level of the
$T_0$-quotient, with the equality $T(x^*)=x^*$ relaxed to the
equivalence $T(x^*)\equiv_\U x^*$. The five results cover
essentially disjoint regimes: contractive maps need completeness,
involutions need $T_0$ to upgrade to genuine fixed points, and
monotone non-contractive maps need order-completeness instead of
either.

Three directions remain open. First, the $T_0$-quotient
construction of Section~\ref{sec:quotient} suggests a transfer
principle beyond the involution theorem treated here: every
fixed-point theorem on a $T_0$ quasi-uniform space should lift,
modulo $\equiv_\U$, to a statement on an arbitrary quasi-uniform
space. The natural form is that the $T_0$-reflection
$\mathrm{QU}\to\mathrm{QU}_{T_0}$ admits a left adjoint
compatible with the underlying-set functor and that fixed-point
theorems extend along the adjunction; this would replace the
hand-built lemmas of Section~\ref{sec:quotient} by a single
functorial statement. Second, the Ulam--Hyers stability analysis
of the start-point problem
in~\cite[Theorem~4]{GabaKarapinarPetruselRadenovic2020}, in
which $\varepsilon$-approximate end-points lie within
$C\varepsilon$ of an exact start-point, deserves a quasi-uniform
analogue compatible with the perturbation-stability estimate of
Proposition~\ref{prop:stability}. Third, bicompleteness is a
strong completeness notion. Because the two lines
of~\eqref{eq:contraction} are logically equivalent
(Remark~\ref{rem:sharpness}), the substantive question is purely
about \emph{completeness}: does left $K$-completeness, the
regime of~\cite{GabaKarapinarPetruselRadenovic2020}, suffice for
Theorem~\ref{thm:contraction} as stated? A
$\mathcal{D}$-contraction $T$ on a left $K$-complete space
generates a Picard sequence that is left $K$-Cauchy in every
$d_\alpha$, hence $\tau(d_\alpha)$-convergent to some $\xi^*$ by
left $K$-completeness; the forward inequality and the
$\Us$-continuity of $T$ (which it inherits from non-expansiveness
in each $d_\alpha^s$) give $T\xi_n\to T\xi^*$ in
$\tau(d_\alpha)$, and the $T_0$ hypothesis on $\U$ together with
the Hausdorff property of $\tau(d_\alpha^s)$ should force
$T\xi^*=\xi^*$. The missing piece is the bridge between
$\tau(d_\alpha)$-convergence and $\Us$-convergence, which is
automatic under bicompleteness and not automatic under left
$K$-completeness alone. We conjecture
Theorem~\ref{thm:contraction} extends to left $K$-completeness
verbatim, with the conjugate-side $K$-completeness of $\Uinv$
needed for the endpoint corollary of
Section~\ref{sec:endpoint}; settling this would close the
remaining gap between the present framework and that
of~\cite{GabaKarapinarPetruselRadenovic2020}. The analogue for
$D$-completeness in the sense of Doitchinov~\cite{Doitchinov}
and Smyth completeness~\cite{Smyth} requires identifying the
right modification of~\eqref{eq:contraction} on each
completeness type
(see~\cite{Andrikopoulos,CarlsonHicks,Kunzi,Romaguera} for the
underlying machinery), and remains open.

\section*{Acknowledgements}

The author thanks the AI Research and Innovation Nexus for Africa
(AIRINA Labs) and the African Center for Advanced Studies (ACAS)
for institutional support during the preparation of this work.

\section*{Funding and competing interests}

No external funding supported this work. The author declares no
competing interests.

\end{document}